\documentclass[reqno]{amsart}
\usepackage{tikz,xcolor,listings,ytableau,bbm,mathdots,mathtools,arydshln,amssymb}
\usepackage[all,cmtip]{xy}
\usetikzlibrary{arrows,matrix,patterns}

\usepackage[english]{babel}
\usepackage[margin=1.0in]{geometry}

\usepackage[colorlinks=true, allcolors=blue]{hyperref}

\definecolor{darkred}{rgb}{0.7,0,0} 
\newcommand{\defn}[1]{{\color{darkred}\emph{#1}}} 

\definecolor{UQgold}{RGB}{196, 158, 54} 
\definecolor{UQpurple}{RGB}{73, 7, 94} 

\newcommand{\ds}{\displaystyle}
\newcommand{\fieldalt}{K}
\newcommand{\iso}{\cong}

\newcommand{\xx}{\mathbf{x}}
\newcommand{\yy}{\mathbf{y}}
\newcommand{\zz}{\mathbf{z}}
\newcommand{\zero}{\mathbf{0}}  
\newcommand{\abs}[1]{\lvert #1 \rvert}  
\newcommand{\tw}[1]{\widetilde{#1}}
\newcommand{\ov}[1]{\overline{#1}}
\newcommand{\wh}[1]{\widehat{#1}}
\newcommand{\bb}[1]{\mathbb{#1}}
\newcommand{\mb}[1]{\mathbf{#1}}
\newcommand{\mc}[1]{\mathcal{#1}}
\newcommand{\vp}{\varphi}
\newcommand{\ve}{\varepsilon}
\newcommand{\la}{\lambda}
\newcommand{\lx}[2]{x_{#1}^{(#2)}}
\newcommand{\lE}[2]{E_{#1}^{(#2)}}

\newcommand{\ls}[2]{s_{#1}^{(#2)}}
\newcommand{\lH}[2]{H_{#1}^{(#2)}}

\newcommand{\bracebelow}[2]{\underbrace{\hspace{#1 cm}}_{#2}}

\DeclareMathOperator{\Mat}{Mat}
\DeclareMathOperator{\Hom}{Hom}
\DeclareMathOperator{\sh}{sh}
\DeclareMathOperator{\gMin}{gMin}
\DeclareMathOperator{\gMax}{gMax}
\DeclareMathOperator{\RSK}{RSK}
\DeclareMathOperator{\gRSK}{gRSK}
\DeclareMathOperator{\SSYT}{SSYT}

\DeclareMathOperator{\Id}{Id}
\DeclareMathOperator{\SL}{SL}
\DeclareMathOperator{\GL}{GL}

\DeclareMathOperator{\GT}{GT}
\DeclareMathOperator{\Trop}{Trop} 
\DeclareMathOperator{\Frac}{Frac} 
\DeclareMathOperator{\Sym}{Sym} 
\DeclareMathOperator{\LSym}{LSym} 

\def\Inv{{\rm Inv}}
\def\Pol{{\rm Pol}}

\newcommand{\ZZ}{\mathbb{Z}}

\newcommand{\RR}{\mathbb{R}}
\newcommand{\CC}{\mathbb{C}}

\theoremstyle{plain}
\newtheorem{thm}{Theorem}[section]
\newtheorem{prob}[thm]{Problem}
\newtheorem{lemma}[thm]{Lemma}
\newtheorem{conj}[thm]{Conjecture}
\newtheorem{prop}[thm]{Proposition}
\newtheorem{cor}[thm]{Corollary}
\theoremstyle{definition}
\newtheorem{dfn-prop}[thm]{Definition-Proposition}
\newtheorem{dfn}[thm]{Definition}
\newtheorem{ex}[thm]{Example}
\newtheorem{remark}[thm]{Remark}
\numberwithin{equation}{section}
\usepackage[colorinlistoftodos]{todonotes}

\makeatletter
\newcommand{\vast}{\bBigg@{3}}
\newcommand{\Vast}{\bBigg@{4.5}}
\makeatother

\title{Crystal invariant theory I: Geometric RSK}

\author[B.~Brubaker]{Benjamin Brubaker}
\address[B. Brubaker]{School of Mathematics, University of Minnesota, 206 Church St. SE, Minneapolis, MN 55455}
\email{brubaker@math.umn.edu}
\urladdr{http://www-users.math.umn.edu/~brubaker/}

\author[G.~Frieden]{Gabriel Frieden}
\address[G. Frieden]{LaCIM, Universit\'e du Qu\'ebec \`a Montr\'eal, Montr\'eal, QC, Canada}
\email{gabriel.frieden@lacim.ca}
\urladdr{https://sites.google.com/a/umich.edu/gfrieden/}

\author[P.~Pylyavskyy]{Pavlo Pylyavskyy}
\address[P. Pylyavskyy]{School of Mathematics, University of Minnesota, 206 Church St. SE, Minneapolis, MN 55455}
\email{ppylyavs@math.umn.edu}
\urladdr{https://sites.google.com/site/pylyavskyy/}

\author[T.~Scrimshaw]{Travis Scrimshaw}
\address[T. Scrimshaw]{School of Mathematics and Physics, The University of Queensland, St. Lucia, QLD 4072, Australia}
\email{tcscrims@gmail.com}
\curraddr{5 Ch\=ome Kita 8 J\=onishi, Kita Ward, Sapporo, Hokkaid\=o 060--0808, Japan}
\urladdr{https://tscrim.github.io/}


\begin{document}

\begin{abstract}
Berenstein and Kazhdan's theory of geometric crystals gives rise to two commuting families of geometric crystal operators acting on the space of complex $m \times n$ matrices. These are birational actions, which we view as a crystal-theoretic analogue of the usual action of $\SL_m \times \SL_n$ on $m \times n$ matrices. We prove that the field of rational invariants (and ring of polynomial invariants) of each family of geometric crystal operators is generated by a set of algebraically independent polynomials, which are generalizations of the elementary symmetric polynomials in $m$ (or $n$) variables. We also give a set of algebraically independent generators for the intersection of these fields, and we explain how these fields are situated inside the larger fields of geometric $R$-matrix invariants, which were studied by Lam and the third-named author under the name loop symmetric functions. The key tool in our proof is the geometric RSK correspondence of Noumi and Yamada, which we show to be an isomorphism of geometric crystals.

In an appendix jointly written with Thomas Lam, we prove the fundamental theorem of loop symmetric functions, which says that the polynomial invariants of the geometric $R$-matrix are generated by the loop elementary symmetric functions.
\end{abstract}

\maketitle

\tableofcontents

\section{Introduction}
The original problem of classical invariant theory was to understand the invariants of the action of $\SL_m$ on the space of $m \times n$ matrices. The first fundamental theorem of invariant theory says that the ring of invariants is trivial if $n < m$, and generated by the $m \times m$ minors of the matrix if $n \geq m$. The second fundamental theorem describes the relations among these generators. We refer the reader to~\cite{Weyl46} for a classical exposition of this subject and to~\cite{GoodWall00,PopVin94} for more modern ones.

In this paper, we study the analogous problem for a ``crystallized'' version of the $\SL_m$-action. Specifically, we consider the invariants of the $\GL_m$-geometric crystal operators, which were introduced by Berenstein and Kazhdan~\cite{BK00} as ``de-tropicalizations'' of Kashiwara's crystal operators~\cite{K90,K91}. In contrast to the classical $\SL_m$-action, these operators act birationally on $m \times n$ matrices. We prove that the ring of polynomial invariants is generated by a set of algebraically independent polynomials, which arise as the entries of a certain $n \times n$ matrix.

\subsection{Main result}
\label{sec:intro main result}

Consider the space $\Mat_{m \times n}(\CC^*)$ of $m \times n$ matrices over $\CC^*$. We write $\xx = (x_i^j)_{i \in \{1, \ldots, m\}}^{j \in \{1, \ldots, n\}}$ for an element of $\Mat_{m \times n}(\CC^*)$, and we denote the $i$th row and $j$th column of $\xx$ by $\xx_i = (x_i^1, \ldots, x_i^n)$ and $\xx^j = (x_1^j, \ldots, x_m^j)$, respectively.\footnote{This is the reverse of the usual convention for covariant and contravariant indices in tensor calculus.}
For $i = 1, \ldots, m-1$, the geometric crystal operator $e_i$ is a rational $\CC^*$-action on $\Mat_{m \times n}(\CC^*)$, and we denote the action of $c \in \CC^*$ on $\xx$ by $e_i^c(\xx)$. The action is given by
\[
e_i^c \colon
\begin{pmatrix}
& \vdots \\
x_i^1 & \ldots & x_i^n \medskip\\
x_{i+1}^1 & \ldots & x_{i+1}^n \\
& \vdots
\end{pmatrix} \mapsto
\begin{pmatrix}
& \vdots \\
\dfrac{\sigma^1(\xx_i, \xx_{i+1}; c)}{\sigma^0(\xx_i, \xx_{i+1}; c)} x_i^1 & \ldots & \dfrac{\sigma^n(\xx_i, \xx_{i+1}; c)}{\sigma^{n-1}(\xx_i, \xx_{i+1}; c)} x_i^n \bigskip \\
\dfrac{\sigma^0(\xx_i, \xx_{i+1}; c)}{\sigma^1(\xx_i, \xx_{i+1}; c)} x_{i+1}^1 & \ldots & \dfrac{\sigma^{n-1}(\xx_i, \xx_{i+1}; c)}{\sigma^n(\xx_i, \xx_{i+1}; c)} x_{i+1}^n \\
& \vdots
\end{pmatrix},
\]
where
\[
\sigma^j\bigl((x^1, \ldots, x^n), (y^1, \ldots, y^n); c\bigr) = \sum_{r = 1}^n c^{\mathbbm{1}_{r \leq j}} y^1 \cdots y^{r-1} x^{r+1} \cdots x^n.
\]
One can view this as analogous to the action of the copy of $\SL_2$ inside $\GL_m$ corresponding to the $i$th simple root, whose generators add a multiple of row $i$ to row $i+1$ (or vice versa), or scale rows $i$ and $i+1$ by $t,t^{-1}$.

For example, when $m=3$ and $n=2$, the geometric crystal operator $e_1$ acts by
\begin{equation}
\label{eq:m2n2}
e_1^c \colon
\begin{pmatrix}
x_1^1 & x_1^2 \medskip \\
x_2^1 & x_2^2 \medskip \\
x_3^1 & x_3^2
\end{pmatrix} \mapsto
\begin{pmatrix}
c_1 x_1^1 & c_2 x_1^2 \medskip \\
c_1^{-1} x_2^1 & c_2^{-1} x_2^2 \medskip \\
x_3^1 & x_3^2
\end{pmatrix}, \quad \text{ where} \quad c_1 = \dfrac{cx_1^2 + x_2^1}{x_1^2+x_2^1}, \quad c_2 = \dfrac{cx_1^2 + cx_2^1}{cx_1^2+x_2^1}.
\end{equation}
When $m=2$ and $n=3$, $e_1$ acts by
\begin{equation}
\label{eq:m2n3}
e_1^c \colon
\begin{pmatrix}
x_1^1 & x_1^2 & x_1^3 \medskip\\
x_2^1 & x_2^2 & x_2^3
\end{pmatrix} \mapsto
\begin{pmatrix}
c_1 x_1^1 & c_2 x_1^2 & c_3 x_1^3 \medskip\\
c_1^{-1} x_2^1 & c_2^{-1} x_2^2 & c_3^{-1} x_2^3
\end{pmatrix},
\end{equation}
where
\[
c_1 = \dfrac{cx_1^2 x_1^3 + x_ 2^1 x_1^3 + x_2^1 x_2^2}{x_1^2 x_1^3 + x_ 2^1 x_1^3 + x_2^1 x_2^2}, \qquad c_2 = \dfrac{cx_1^2 x_1^3 + cx_ 2^1 x_1^3 + x_2^1 x_2^2}{cx_1^2 x_1^3 + x_ 2^1 x_1^3 + x_2^1 x_2^2}, \qquad c_3 = \dfrac{cx_1^2 x_1^3 + cx_ 2^1 x_1^3 + cx_2^1 x_2^2}{cx_1^2 x_1^3 + cx_ 2^1 x_1^3 + x_2^1 x_2^2}.
\]

Given $x_1, \ldots, x_n \in \CC^*$, let $W(x_1, \ldots, x_n)$ be the $n \times n$ matrix with entries $x_1, \ldots, x_n$ along the main diagonal, 1's directly beneath the main diagonal, and 0's elsewhere. Associate to $\xx \in \Mat_{m \times n}(\CC^*)$ the $n \times n$ matrix
\[
M(\xx) = W(\xx_1) \cdots W(\xx_m).
\]

\begin{thm}
\label{thm:main}
\
\begin{enumerate}
\item The non-trivial entries of $M(\xx)$ (\textit{i.e.}, those that are not 0 or 1) are algebraically independent generators of the subfield $\Inv_e \subset \CC(x_i^j)$ of invariants of the $\GL_m$-geometric crystal operators $e_1, \ldots, e_{m-1}$.
\item The non-trivial entries of $M(\xx)$ generate the subring $\Inv_e \cap \CC[x_i^j]$ of polynomial invariants of the $e_i$.
\end{enumerate}
\end{thm}

For example, when $m=2$ and $n=3$, we have
\begin{equation}
\label{eq:m2n3 matrix}
M(\xx) = \begin{pmatrix}
x_1^1 & 0 & 0 \medskip\\
1 & x_1^2 & 0 \medskip\\
0 & 1 & x_1^3
\end{pmatrix}
\begin{pmatrix}
x_2^1 & 0 & 0 \medskip\\
1 & x_2^2 & 0 \medskip\\
0 & 1 & x_2^3
\end{pmatrix}
=
\begin{pmatrix}
x_1^1x_2^1 & 0 & 0 \medskip \\
x_1^2+x_2^1 & x_1^2x_2^2 & 0 \medskip \\
1 & x_1^3+x_2^2 & x_1^3x_2^3
\end{pmatrix}.
\end{equation}
Using~\eqref{eq:m2n3}, the reader may easily verify that each of the five non-trivial entries of this matrix is invariant under $e_1$. Theorem~\ref{thm:main} asserts that every polynomial invariant of $e_1$ has a unique expression as a polynomial in these five entries.

\subsection{Crystal operators}
\label{sec:intro crystal ops}

Let $\Sym^L \CC^m$ denote the $L$-fold symmetric power of the vector representation of $\GL_m(\CC)$. This representation has a basis consisting of the degree $L$ monomials $v_1^{a_1} \cdots v_m^{a_m}$, where $\{v_1, \ldots, v_m\}$ is a fixed basis of $\CC^m$. The set of all exponent vectors $\mb{a} = (a_1, \ldots, a_m) \in (\ZZ_{\geq 0})^m$ labels a basis of the representation $\Sym \CC^m = \bigoplus_{L \geq 0} \Sym^L \CC^m$. One can also represent the monomial $v_1^{a_1} \cdots v_m^{a_m}$ as a semistandard Young tableau consisting of a single row, in which the number $i$ appears $a_i$ times.

Now consider the $n$-fold tensor product $(\Sym \CC^m)^{\otimes n}$. A basis of this representation is labeled by $m \times n$ matrices of nonnegative integers $\mb{a} = (a_i^j)$, where the $j$th column $\mb{a}^j = (a_1^j, \ldots, a_m^j)$ represents a basis vector of the $j$th tensor factor. The set of such matrices is endowed with $\GL_m$-crystal operators $\widetilde{e}_1, \ldots, \widetilde{e}_{m-1}$, where $\widetilde{e}_i$ modifies the entries in rows $i$ and $i+1$ (see, \textit{e.g.},~\cite{K02book,BumpSchill}).

\begin{ex}
\label{ex:3x2 crystal}
Let $\mb{a} = \begin{pmatrix}
a_1^1 & a_1^2 \medskip \\
a_2^1 & a_2^2 \medskip \\
a_3^1 & a_3^2
\end{pmatrix}$
represent a basis element of the $\GL_3$-representation $\Sym \CC^3 \otimes \Sym \CC^3$. In tableau notation, $\mb{a}$ corresponds to
\begin{center}
\begin{tikzpicture}[scale=0.5]
\draw (0,0) rectangle (12,1);
\foreach \x in {1,3,4,5,7,8,9,11} {\draw (\x,0) -- (\x,1);}
\draw (1/2,1/2) node{$1$};
\draw (7/2,1/2) node{$1$};
\draw (9/2,1/2) node{$2$};
\draw (15/2,1/2) node{$2$};
\draw (17/2,1/2) node{$3$};
\draw (23/2,1/2) node{$3$};
\draw (4/2,1/2) node{$\cdots$};
\draw (12/2,1/2) node{$\cdots$};
\draw (20/2,1/2) node{$\cdots$};
\draw (2,-0.75) node{$\bracebelow{1.8}{a_1^1}$};
\draw (6,-0.75) node{$\bracebelow{1.8}{a_2^1}$};
\draw (10,-0.75) node{$\bracebelow{1.8}{a_3^1}$};
\begin{scope}[xshift=14cm]
\draw (0,0) rectangle (12,1);
\foreach \x in {1,3,4,5,7,8,9,11} {\draw (\x,0) -- (\x,1);}
\draw (1/2,1/2) node{$1$};
\draw (7/2,1/2) node{$1$};
\draw (9/2,1/2) node{$2$};
\draw (15/2,1/2) node{$2$};
\draw (17/2,1/2) node{$3$};
\draw (23/2,1/2) node{$3$};
\draw (4/2,1/2) node{$\cdots$};
\draw (12/2,1/2) node{$\cdots$};
\draw (20/2,1/2) node{$\cdots$};
\draw (2,-0.75) node{$\bracebelow{1.8}{a_1^2}$};
\draw (6,-0.75) node{$\bracebelow{1.8}{a_2^2}$};
\draw (10,-0.75) node{$\bracebelow{1.8}{a_3^2}$};
\end{scope}
\draw (13,1/2) node{$\otimes$};
\draw (26.5,0) node{$.$};
\end{tikzpicture}
\end{center}
The $\GL_3$-crystal operator $\widetilde{e}_1$ acts on $\mb{a}$ by the following piecewise-linear formula:
\[
\widetilde{e}_1 \colon \begin{pmatrix}
a_1^1 & a_1^2 \medskip \\
a_2^1 & a_2^2 \medskip \\
a_3^1 & a_3^2
\end{pmatrix}
\mapsto
\begin{pmatrix}
a_1^1+\widetilde{c}_1 & a_1^2+\widetilde{c}_2 \medskip \\
a_2^1-\widetilde{c}_1 & a_2^2-\widetilde{c}_2 \medskip \\
a_3^1 & a_3^2
\end{pmatrix},
\]
where
\[
\widetilde{c}_1 = \min(1 + a_1^2, a_ 2^1) - \min(a_1^2, a_ 2^1), \qquad \widetilde{c}_2 = \min(1 + a_1^2, 1+a_ 2^1) - \min(1+a_1^2, a_ 2^1).
\]
\end{ex}

The reader will observe that the piecewise-linear formula in Example~\ref{ex:3x2 crystal} is obtained from~\eqref{eq:m2n2} by tropicalizing---replacing the operations $(+,\cdot,\div)$ with the operations $(\min,+,-)$---and setting $c = 1$ (we also rename the variables to emphasize that they now represent integers rather than nonzero complex numbers). This is no coincidence, as the geometric crystal operators were defined by Berenstein and Kazhdan (in the context of an arbitrary reductive algebraic group) so that they tropicalize to piecewise-linear formulas for the Kashiwara operators acting on the crystal bases of finite-dimensional representations~\cite{BK00,BK07} (see \S \ref{sec:geom to comb} for more details).

\subsection{Commuting crystal actions, RSK, and gRSK}
\label{sec:intro RSK}

One way to prove the first and second fundamental theorems of invariant theory is to consider the $\GL_n$- and $\GL_m$-actions on $m \times n$ matrices simultaneously. These two actions commute, so one can ask how to decompose the polynomial ring $\CC[x_i^j]$ into $\GL_n \times \GL_m$-irreducible representations. Using this decomposition (and some knowledge of the representation theory of $\GL_m$), one can read off the first fundamental theorem by looking at the isotypic components in which the $\GL_m$-irreducible is trivial as an $\SL_m$-module. The second fundamental theorem can also be deduced from representation theory; for details of this approach to invariant theory, see~\cite[\S 9.2]{Fulton}.

Our proof of Theorem~\ref{thm:main} proceeds in a similar spirit. The decomposition of $\CC[x_i^j]$ into $\GL_n \times \GL_m$-irreducible representations has a beautiful analogue at the level of crystal bases. Just as there are $\GL_m$-crystal operators acting on adjacent rows of a matrix of nonnegative integers, there are $\GL_n$-crystal operators acting on adjacent columns. Although it is not obvious from the combinatorial definition, it turns out that these two sets of crystal operators commute~\cite{Lascoux03,DanilovKoshevoy05,vanLeeuwen06} (this can be viewed as a crystal version of $(\GL_n, \GL_m)$-Howe duality~\cite{Howe89}). The connected components of the resulting $\GL_n \times \GL_m$-crystal, which correspond to the irreducible components of $\CC[x_i^j]$ under the usual $\GL_n \times \GL_m$-action, are determined by the Robinson--Schensted--Knuth (RSK) correspondence (see, \textit{e.g.,}~\cite[Ch. 9.1]{BumpSchill}). RSK is a well-known bijection between $m \times n$ matrices of nonnegative integers and pairs $(P,Q)$ of semistandard Young tableaux of the same shape, where the entries of $P$ lie in $\{1, \ldots, n\}$, and the entries of $Q$ lie in $\{1, \ldots, m\}$~\cite{Schensted,Knuth}.

A geometric lifting of RSK was introduced by Kirillov~\cite{Kir01} and extensively developed by Noumi and Yamada~\cite{NoumiYamada}. This map was originally called the ``tropical RSK correspondence,'' but following the convention of more recent work on this map, we will call it the \defn{geometric RSK correspondence} and abbreviate it gRSK. Geometric RSK is a birational automorphism of the variety of $m \times n$ matrices over $\CC^*$. For example, when $m=2$ and $n=3$, gRSK has the form
\[
\gRSK \colon \begin{pmatrix}
x_1^1 & x_1^2 & x_1^3 \medskip \\
x_2^1 & x_2^2 & x_2^3
\end{pmatrix}
\mapsto
\begin{pmatrix}
\textcolor{blue}{z_{2,2}} & \textcolor{blue}{z_{2,3}} = \textcolor{red}{z'_{2,2}} & \textcolor{red}{z'_{1,1}} \medskip \\
\textcolor{blue}{z_{1,1}} & \textcolor{blue}{z_{1,2}} & \textcolor{blue}{z_{1,3}} = \textcolor{red}{z'_{1,2}}
\end{pmatrix} \; ,
\]
where $z_{a,b}$ and $z'_{a,b}$ are rational functions in the $x_i^j$ with positive integer coefficients. The rational functions $z_{a,b}$ and $z'_{a,b}$ tropicalize to piecewise-linear formulas for the semistandard tableaux $P$ and $Q$ (more precisely, for the entries of the corresponding Gelfand--Tsetlin patterns) associated to a matrix $\mb{a}$ by RSK. In particular, the rational functions $z_{k,n} = z'_{k,m}$ tropicalize to formulas for the entries of the common shape of $P$ and $Q$. We refer to the arrays consisting of the $z_{a,b}$ and $z'_{a,b}$ as the $P$-pattern and $Q$-pattern, respectively.

In addition to the $\GL_m$-geometric crystal operators $e_i$ acting on adjacent rows of $\xx \in \Mat_{m \times n}(\CC^*)$, there are $\GL_n$-geometric crystal operators acting on adjacent columns of $\xx$, which we denote by $\ov{e}_j$. These two families of actions are known to commute with each other~\cite{LP13II}, so the variety $\Mat_{m \times n}(\CC^*)$ has the structure of a $\GL_n \times \GL_m$-geometric crystal, which we denote by $X^{\Mat}$. Berenstein and Kazhdan's theory also provides a geometric lifting of the crystal operators on Gelfand--Tsetlin patterns. By viewing an $m \times n$ matrix as a $P$-pattern and a $Q$-pattern glued together along a common shape, we obtain a different $\GL_n \times \GL_m$-geometric crystal structure on $\Mat_{m \times n}(\CC^*)$, in which the $\GL_n$- (resp., $\GL_m$-) operators act on the $P$- (resp., $Q$-) pattern. We denote this geometric crystal by $X^{\GT}$.

We prove in Theorem~\ref{thm:gRSK isom} that
\[
\gRSK \colon X^{\Mat} \rightarrow X^{\GT}
\]
is an isomorphism of $\GL_n \times \GL_m$-geometric crystals. Since the $\GL_m$-geometric crystal operators on $X^{\GT}$ act only on the $Q$-pattern, this result implies that the entries $z_{i,j}$ of the $P$-pattern are invariant under the $\GL_m$-geometric crystal operators. It is straightforward to show that the rational functions $z_{i,j}$ generate the same subfield of $\CC(x_i^j)$ as the entries of the matrix $M(\xx)$ in Theorem~\ref{thm:main}. The key to proving that there are no other invariants is Theorem~\ref{thm:connected}, which says that the action of the $\GL_m$-geometric crystal operators on the $Q$-pattern has a dense orbit for each fixed shape. This result, in turn, relies on work of Kanakubo--Nakashima~\cite{KanNak} and Kashiwara--Nakashima--Okado~\cite{KNO10}.

By interchanging the roles of $m$ and $n$ in Theorem~\ref{thm:main}, one sees that the subfield $\Inv_{\ov{e}} \subset \CC(x_i^j)$ fixed by the $\GL_n$-geometric crystal operators is generated by the entries of the $m \times m$ matrix $M(\xx^t)$, where $\xx^t$ is the transpose of $\xx$. We prove in Corollary~\ref{cor:X generators} that the subfield fixed by both the $\GL_n$- and $\GL_m$-geometric crystal operators, $\Inv_e \cap \Inv_{\ov{e}}$, is generated by $\min(m,n)$ polynomials which are generalizations of Schur polynomials of rectangular shape. These polynomials are products of the rational functions $z_{k,n} = z'_{k,m}$ that describe the common shape of the $P$- and $Q$-patterns.

\subsection{Birational Weyl group actions}
\label{sec:intro R}
Berenstein and Kazhdan's primary motivation in developing the theory of geometric crystals was to construct birational actions of Weyl groups. In the case of the $\GL_m$-geometric crystal on $\Mat_{m \times n}(\CC^*)$ described in~\S \ref{sec:intro main result}, the resulting action of the Weyl group $S_m$ is generated by the action of the \defn{geometric} (or \defn{birational}) \defn{$R$-matrix} on adjacent rows. The geometric $R$-matrix was introduced by Yamada~\cite{Yamada01} as a ``de-tropicalization'' of the \defn{combinatorial $R$-matrix}, a map which comes from the theory of affine crystals~\cite{KKMMNN91,HKOTY99,HKOTT02}, and plays a central role in the study of solvable lattice models and the box-ball system~\cite{HHIKTT01,HKOTY02,HKT00,HKT01,KSY07,LS19,Takagi05,TNS99,TS90,Yamada01,Yamada04}.

The subfield $\Inv_R \subset \CC(x_i^j)$ consisting of geometric $R$-matrix invariants has been previously studied~\cite{LP12,LP13,LPS15, LPaff, Lamloop}. In particular, it is known that this field (and, in fact, its subring of polynomial invariants) is generated by a set of $mn$ algebraically independent polynomials called \defn{loop elementary symmetric functions}. When $n=1$, these polynomials are simply the elementary symmetric polynomials in $m$ variables; in general, there are $n$ loop elementary symmetric functions of degree $k$ for $k = 1, \ldots, m$, each of which can be viewed as a generalization of the $k$th elementary symmetric polynomial in $m$ variables.

The Weyl group action on a geometric crystal is generated by the geometric crystal operators $e_i^c$ for specific values of $c$, so $\Inv_e$ is a subfield of $\Inv_R$. The generators of $\Inv_e$ described in Theorem~\ref{thm:main} are in fact a subset of the loop elementary symmetric polynomials. For example, when $m=2$ and $n=3$, five of the six loop elementary symmetric polynomials appear as entries of $M(\xx)$ in~\eqref{eq:m2n3 matrix}, and the remaining loop elementary symmetric polynomial is $x_1^1 + x_2^3$. We emphasize that our proof of Theorem~\ref{thm:main}(1) does not make any reference to geometric $R$-matrix invariants. On the other hand, our proof of the stronger result that the entries of $M(\xx)$ generate the polynomial subring of $\Inv_e$ (Theorem~\ref{thm:main}(2)) relies on the fundamental theorem of loop symmetric functions, which says that the polynomial subring of $\Inv_R$ is generated by the loop elementary symmetric functions.

\subsection{Future work}
\label{sec:intro future}

In the sequel to this paper~\cite{BFPSII}, we study the fields $\Inv_R \cap \Inv_{\ov{e}}, \Inv_{\ov{R}} \cap \Inv_e,$ and $\Inv_R \cap \Inv_{\ov{R}}$, where $\Inv_{\ov{R}}$ is the field of invariants of the $S_n$-action generated by applying geometric $R$-matrices to adjacent columns of $\xx$. In particular, we give conjectural algebraically independent generating sets for each field. The field $\Inv_R \cap \Inv_{\ov{e}}$ contains two important functions: the \defn{central charge}~\cite{BK07} and the \defn{geometric energy function}~\cite{KKMMNN91, HKOTY99, HKOTT02, LP13}. One of our motivations for this project was to find formulas for these functions which simultaneously exhibit their $R$- and $\ov{e}$-invariance, and we succeed in expressing both functions as polynomials in the (conjectural) generators of $\Inv_R \cap \Inv_{\ov{e}}$. We also obtain a new derivation of a piecewise-linear formula for cocharge due to Kirillov and Berenstein~\cite{KB95}.

In a different direction, we propose the following general problem.

\begin{prob}
Given geometric crystals $X_1, \ldots, X_n$ associated to a fixed reductive group, describe the invariants of the geometric crystal operators (or the birational Weyl group action) as a subfield of the fraction field $\CC(X_1 \times \cdots \times X_n)$.
\end{prob}

The existence of commuting type $A$ crystal structures on $m \times n$ matrices is a crystal interpretation of Howe duality~\cite{Howe89} for the dual reductive pair $(\GL_n, \GL_m)$, which states that
\[
\Sym(\CC^n \boxtimes \CC^m) \iso \bigoplus_{\lambda \colon \ell(\lambda) \leq \min(m,n)} V_{\GL_n}(\lambda) \boxtimes V_{\GL_m}(\lambda)
\]
as $\GL_n \times \GL_m$-representations. (Here $V_{\GL_k}(\la)$ is the irreducible highest weight representation of $\GL_k$ with highest weight $\la$, and $\ell(\la)$ is the number of parts in $\la$.) It may be fruitful to investigate products of geometric crystals that correspond to other examples of Howe duality or skew Howe duality.

\subsection{Outline of paper}
\label{sec:intro outline}

In~\S \ref{sec:geometric crystals}, we review fundamental definitions and results in the theory of geometric crystals, and we explain how combinatorial crystals are obtained from decorated geometric crystals by tropicalization. We illustrate the theory with the example of the basic geometric crystal, the geometric analogue of the crystal structure on one-row tableaux.

In~\S \ref{sec:GT}, we construct a geometric crystal structure on the geometric analogue of a Gelfand--Tsetlin pattern. This is essentially a special case of a general construction due to Berenstein and Kazhdan~\cite{BK07}, but we derive many of its properties in an elementary manner, using only the Lindstr\"om/Gessel--Viennot Lemma. In particular, in~\S \ref{sec:dec} we obtain explicit formulas for the geometric crystal operators and decoration, and show that they tropicalize to a piecewise-linear description of the usual crystal structure on Gelfand--Tsetlin patterns.

In~\S \ref{sec:gRSK}, we present two definitions of the geometric RSK correspondence, one due to Noumi--Yamada~\cite{NoumiYamada}, and one due to O'Connell--Sepp\"al\"ainen--Zygouras~\cite{OSZ}. Using these two definitions, we derive several important properties of gRSK, including the new result that gRSK is an isomorphism of geometric crystals (Theorem~\ref{thm:gRSK isom}). Our presentation is self-contained, aside from the fact that we refer to~\cite{OSZ} for the proof that the two definitions agree. We caution the reader that our version of gRSK differs from the usual one by a birational involution that corresponds to interchanging $\min$ and $\max$ (see~\S \ref{sec:max min}). The usual version is not compatible with the geometric crystal structures.

As an application of the results in \S\S \ref{sec:GT}-\ref{sec:gRSK}, we show in~\S \ref{sec:central charge} that for $m$-fold products of the basic $\GL_n$-geometric crystal, Berenstein and Kazhdan's central charge~\cite{BK07} is equal to the decoration of the $Q$-pattern (plus an extra term when $m=n$). This implies, in particular, that the central charge is positive.

In~\S \ref{sec:e}, we prove part Theorem~\ref{thm:main}(1), and deduce the result about generators of $\Inv_e \cap \Inv_{\ov{e}}$ as a corollary.

In~\S \ref{sec:geometric R}, we show that the Weyl group action on products of the basic geometric crystal agrees with the geometric $R$-matrix. In~\S \ref{sec:LSym}, we review results of Thomas Lam and the third-named author about the ring of loop symmetric functions, and deduce Theorem~\ref{thm:main}(2) from Theorem~\ref{thm:main}(1) and the fundamental theorem of loop symmetric functions. The fundamental theorem of loop symmetric functions, which says that the ring of loop symmetric functions is equal to the ring of polynomial $R$-invariants, is proved in Appendix~\ref{app:FTLSF} (jointly written with Lam). In~\S \ref{sec:Schur}, we discuss loop Schur functions, a class of loop symmetric functions which contains the generators of $\Inv_e \cap \Inv_{\ov{e}}$. We conjecture that the ring $\Inv_e \cap \Inv_{\ov{e}} \cap \CC[x_i^j]$ has the same generating set.

\begin{remark}
The definitions and results in \S\S \ref{sec:comb crystal}-\ref{sec:geom to comb}, \ref{sec:dec}, and \ref{sec:dec again}-\ref{sec:central charge} are not needed for the proof of our main result. We have included these sections in the hope that this paper will be a convenient reference for readers interested in learning the basics of type $A$ geometric crystals and geometric RSK. In addition, some of the results in these sections (in particular those concerning the central charge) will be used in~\cite{BFPSII}.
\end{remark}

\subsection*{Notation}
For integers $a,b$, we write $[a,b]$ for the interval $\{k \in \ZZ \mid a \leq k \leq b\}$. We often abbreviate $[1,b]$ to $[b]$.

\subsection*{Acknowledgements}

We thank Thomas Lam, during discussions with whom some of the ideas in this paper were born. We thank Jake Levinson for his help with the proof of Lemma~\ref{lem:orbits}.

This work benefited from computations using {\sc SageMath}~\cite{sage,combinat}. B.B.~was supported in part by NSF grant DMS-2101392. G.F.~was supported in part by the Canada Research Chairs program. P.P.~was supported in part by NSF grant DMS-1949896. T.S.~was supported in part by Grant-in-Aid for JSPS Fellows 21F51028. This work was partly supported by Osaka City University Advanced Mathematical Institute (MEXT Joint Usage/Research Center on Mathematics and Theoretical Physics JPMXP0619217849).


\section{Background on geometric and unipotent crystals}
\label{sec:geometric crystals}

In~\S\S \ref{sec:basic}-\ref{sec:unipotent}, we recall the definitions of geometric and unipotent crystals in the case of $\GL_m$. These definitions are due to Berenstein and Kazhdan~\cite{BK00,BK07}. In \S\S \ref{sec:comb crystal}-\ref{sec:geom to comb}, we describe how geometric crystals give rise to the combinatorial crystals introduced by Kashiwara~\cite{K90,K91} via the tropicalization functor.

\subsection{Geometric crystals}
\label{sec:basic}

Let $T = (\CC^*)^m$ be the maximal torus of $\GL_m$. For $i \in [m-1]$, let $\alpha_i \colon T \rightarrow \CC^*$ and $\alpha_i^\vee \colon \CC^* \rightarrow T$ denote the character and co-character
\[
\alpha_i(x_1, \ldots, x_m) = \dfrac{x_i}{x_{i+1}} \qquad \text{ and } \qquad \alpha_i^\vee(c) = (1, \ldots, c, c^{-1}, \ldots, 1),
\]
where $c$ and $c^{-1}$ are in positions $i$ and $i+1$.

\begin{dfn}
\label{defn:geom crystal}
A \defn{$\GL_m$-geometric crystal} is an irreducible complex algebraic variety $X$, together with a rational map $\gamma \colon X \rightarrow T$, and for each $i \in [m-1]$, rational functions\footnote{Berenstein and Kazhdan's definition~\cite{BK00,BK07} differs from ours by replacing $\ve_i, \vp_i$ with $1/\ve_i, 1/\vp_i$. Our convention has the advantage that $\ve_i, \vp_i$ tropicalize to the corresponding functions $\tw{\ve}_i, \tw{\vp}_i$ on combinatorial crystals, rather than their negatives.} $\ve_i, \vp_i \colon X \rightarrow \CC$ and a rational $\CC^*$-action $e_i : \CC^* \times X \rightarrow X$. We write $e_i^c(x)$ for the action of $c \in \CC^*$ on $x \in X$. These maps must satisfy the following identities (when they are defined):
\begin{enumerate}
\item $\dfrac{\vp_i(x)}{\ve_i(x)} = \alpha_i\bigl(\gamma(x)\bigr)$;
\item
$\gamma\bigl(e_i^c(x)\bigr) = \alpha_i^\vee(c) \gamma(x), \qquad \ve_i\bigl(e_i^c(x)\bigr) = c^{-1} \ve_i(x), \qquad \vp_i\bigl(e_i^c(x)\bigr) = c \vp_i(x);$
\item
\begin{enumerate}
\item if $|i-j| > 1$, then $e_i^c e_j^{c'} = e_j^{c'} e_i^c$;
\item if $|i-j| = 1$, then $e_i^c e_j^{cc'} e_i^{c'} = e_j^{c'} e_i^{cc'} e_j^c$.
\end{enumerate}
\end{enumerate}
The maps $e_i^c$ are called \defn{geometric crystal operators}.

An \defn{isomorphism} of $\GL_m$-geometric crystals is a birational isomorphism of underlying varieties that commutes with the geometric crystal structures.
\end{dfn}

\begin{remark}
Axioms (1) and (2) are analogues of the axioms of a combinatorial crystal (see~\S \ref{sec:comb crystal}). The purpose of axiom (3) is to guarantee that the maps $s_i(\xx) = e_i^{1/\alpha_i(\gamma(\xx))}$ generate a birational action of the Weyl group $S_m$ on $X$. This action is discussed further in~\S \ref{sec:R}.
\end{remark}

The geometric analogue of the $\GL_m$-representation $\Sym \CC^m$ is the \defn{basic geometric crystal} $X_m$ introduced in~\cite{KOTY03}, which has underlying variety $X_m  = (\CC^*)^m$, and the following geometric crystal structure: for $\xx = (x_1, \ldots, x_m) \in X_m$,
\begin{equation}
\label{eq:basic geometric crystal}
\gamma(\xx) = \xx,
\qquad\quad
\varepsilon_i(\xx) = x_{i+1},
\qquad\quad
\varphi_i(\xx) = x_i,
\qquad\quad
e_i^c(\xx) = (x_1, \dotsc, c x_i, c^{-1} x_{i+1}, \dotsc, x_m).
\end{equation}

Given two geometric crystals $X$ and $X'$, Berenstein and Kazhdan~\cite{BK00} define
\begin{equation}
\label{eq:functions prod}
\gamma(x, x') = \gamma(x) \gamma(x'),
\quad
\varepsilon_i(x, x') = \frac{\varepsilon_i(x) \varepsilon_i(x')}{\varepsilon_i(x) + \varphi_i(x')},
\quad
\varphi_i(x, x') = \frac{\varphi_i(x) \varphi_i(x')}{\varepsilon_i(x) + \varphi_i(x')},
\end{equation}
and
\begin{equation}
\label{eq:e prod}
e_i^c(x, x') = \bigl( e_i^{c^+}(x), e_i^{c/c^+}(x') \bigr),
\qquad \text{ where } \,
c^+ = \frac{c \varphi_i(x') + \varepsilon_i(x)}{\varphi_i(x') + \varepsilon_i(x)}.
\end{equation}
This definition is associative, and $X \times X'$ satisfies the first two axioms of Definition~\ref{defn:geom crystal}. In order to guarantee that the third axiom is satisfied, however, additional structure on $X$ and $X'$ is needed. This additional structure is introduced in \S \ref{sec:unipotent}, and plays a crucial role throughout the paper.

We end this section with explicit formulas for the geometric crystal structure on $(X_m)^n$. For $j \in [0,n]$, $c \in \CC^*$, and $x = (x^1, \ldots, x^n), y = (y^1, \ldots, y^n) \in (\CC^*)^n$, define
\begin{equation}
\label{eq:sigma^j}
\sigma^j(x,y; c) = \sum_{r = 1}^n c^{\mathbbm{1}_{r \leq j}} y^1 \cdots y^{r-1} x^{r+1} \cdots x^n,
\end{equation}
where $\mathbbm{1}_{r \leq j}$ is the indicator function for $r \leq j$. Also define $\sigma(x,y) = \sigma^j(x,y;1)$ (this is independent of $j$).

\begin{lemma}
\label{lem:explicit formulas basic}
Suppose $\xx = (\xx^1, \ldots, \xx^n) \in (X_m)^n$, where $\xx^j = (x_1^j, \ldots, x_m^j)$. For $k \in [m]$, let $\xx_k = (x_k^1, \ldots, x_k^n)$ be the vector of $k$th coordinates of the $\xx^j$, and let $\pi_k(\xx) = \prod_{j=1}^n x_k^j$. The $\GL_m$-geometric crystal structure on $(X_m)^n$ is given by the following formulas:
\begin{gather*}
\textstyle \gamma(\xx) = (\pi_1(\xx), \ldots, \pi_m(\xx)), \qquad\quad
\displaystyle \ve_i(\xx) = \dfrac{\pi_{i+1}(\xx)}{\sigma(\xx_i, \xx_{i+1})}, \quad\qquad \vp_i(\xx) = \dfrac{\pi_i(\xx)}{\sigma(\xx_i, \xx_{i+1})},
\\
e_i^c(\xx) = (\xx'), \quad \text{ where } (x')_k^j = \begin{cases}
\dfrac{\sigma^j(\xx_i, \xx_{i+1}; c)}{\sigma^{j-1}(\xx_i, \xx_{i+1}; c)} x_i^j & \text{if } k = i, \\[10pt]
\dfrac{\sigma^{j-1}(\xx_i, \xx_{i+1}; c)}{\sigma^j(\xx_i, \xx_{i+1}; c)} x_{i+1}^j & \text{if } k = i+1, \\[10pt]
x_k^j & \text{otherwise.}
\end{cases}
\end{gather*}
\end{lemma}

The proof appears in Appendix~\ref{app:proofs}.

\subsection{Unipotent crystals}
\label{sec:unipotent}

Let $B^-,U \subset \GL_m(\CC)$ be the subgroups of lower triangular matrices and upper uni-triangular matrices, respectively. The subgroup $U$ is generated by the \defn{(upper) Chevalley generators}
\[
x_i(a) = I + aE_{i,i+1} \quad (i \in [m-1], a \in \CC),
\]
where $E_{i,j}$ is the matrix unit with 1 in position $(i,j)$ and 0's elsewhere. Define a rational action of $U$ on $B^-$ by $u.b = b'$ if $ub = b'u'$ for some $u' \in U, b' \in B^-$. If $ub$ does not have such a factorization, then the action of $u$ on $b$ is undefined. It is easy to see that the action of a Chevalley generator on $M \in B^-$ is given by
\begin{equation}
\label{eq:U action}
x_i(a).M = x_i(a) \cdot M \cdot x_i\left(\dfrac{-aM_{i+1,i+1}}{M_{i,i} + aM_{i+1,i}}\right).
\end{equation}

\begin{dfn}
A \defn{$U$-variety} is an irreducible complex algebraic variety $X$, together with a rational $U$-action $U \times X \rightarrow X$. A \defn{morphism} of $U$-varieties is a rational map which commutes with the $U$-actions.
\end{dfn}

We view $B^-$ as a $U$-variety with respect to the rational action defined above.

\begin{dfn}
A \defn{$\GL_m$-unipotent crystal} is a pair $(X,\iota)$, where $X$ is a $U$-variety, and $\iota \colon X \rightarrow B^-$ is a morphism of $U$-varieties such that for each $i \in [m-1]$, the function $\iota(x)_{i+1,i}$ is not identically zero. (We will omit $\GL_m$ from the name when there is no danger of confusion.)
\end{dfn}

\begin{ex}
\label{ex:unipotent crystal}
The pair $(B^-, \Id)$ is a unipotent crystal. More generally, define
\[
(B^-)^{\leq n} = \{A \in \GL_m \mid A_{ij} = 0 \text{ if } i < j \text{ or } i-j > n, \text{ and } A_{ij} = 1 \text{ if } i-j = n\}
\]
for $n \geq 1$ (note that $(B^-)^{\leq n} = B^-$ if $n \geq m$). It is clear that the rational action of $U$ on $B^-$ restricts to a rational action on $(B^-)^{\leq n}$, so the pair $((B^-)^{\leq n}, \iota)$ is a unipotent crystal, where $\iota$ is the inclusion of $(B^-)^{\leq n}$ into $B$.
\end{ex}

Given unipotent crystals $(X,\iota)$ and $(X',\iota')$, endow the product $X \times X'$ with the rational $U$-action
\[
u.(x,x') = (u.x, u'.x'),
\]
where $u' \in U$ is determined by $u\iota(x) = b'u'$ for some $b' \in B^-$ (if $u\iota(x)$ does not have such a factorization, then $u.(x,x')$ is undefined). Also define $\iota * \iota' \colon X \times X' \rightarrow B^-$ by
\[
(x,x') \mapsto \iota(x)\iota'(x').
\]
Berenstein and Kazhdan proved that the pair $(X \times X', \iota * \iota')$ is a unipotent crystal, and that this product of unipotent crystals is associative~\cite[Thm.~3.3, Prop.~3.4]{BK00}.

\begin{thm}[{\cite[Thm.~3.8, Lem.~3.9]{BK00}}]
\label{thm:unipotent to geometric}
Let $(X,\iota)$ be a $\GL_m$-unipotent crystal. Fix $i \in [m-1]$. For $x \in X$, set $M = \iota(x)$, and define
\[
\gamma(x) = (M_{1,1}, \ldots, M_{m,m}),
\qquad\quad
\ve_i(x) = \dfrac{M_{i+1,i+1}}{M_{i+1,i}},
\qquad\quad
\vp_i(x) = \dfrac{M_{i,i}}{M_{i+1,i}},
\]
\[
e_i^c(x) = x_i((c-1)\vp_i(x)).x \qquad \text{(here } . \text{ is the rational action of } U \text{ on } X \text{).}
\]
\begin{enumerate}
\item These maps make $X$ into a $\GL_m$-geometric crystal. We say that this geometric crystal is \defn{induced} from the unipotent crystal $(X,\iota)$.

\item If $(X',\iota')$ is another $\GL_m$-unipotent crystal, then the geometric crystal induced from the product $(X \times X', \iota * \iota')$ is the product of the geometric crystals induced from $(X,\iota)$ and $(X',\iota')$.
\end{enumerate}
\end{thm}

Part (2) implies that we need not worry about the product of geometric crystals failing to be a geometric crystal, as long as we restrict our attention to geometric crystals induced from unipotent crystals. This is not a serious restriction in practice, as to the best of our knowledge, all examples of geometric crystals appearing in the literature are induced from unipotent crystals. In addition, most computations with geometric crystals are done at the level of unipotent crystals, largely because of the following property.

\begin{cor}
\label{cor:unipotent property}
If $X$ is the geometric crystal induced from a unipotent crystal $(X,\iota)$, then
\[
\iota(e_i^c(x)) = x_i\bigl( (c-1)\vp_i(x) \bigr) \cdot \iota(x) \cdot x_i \bigl( (c^{-1} - 1)\ve_i(x) \bigr).
\]
\end{cor}

\begin{proof}
Since $\iota$ commutes with the $U$-actions on $X$ and $B^-$, $\iota(e_i^c(x))$ is equal to the action of $x_i((c-1)\vp_i(x))$ on $\iota(x)$. Using~\eqref{eq:U action} and the definitions of $\vp_i(x)$ and $\ve_i(x)$, one obtains the result.
\end{proof}

We now explain how the basic geometric crystal $X_m$ arises from a unipotent crystal. Identify the variety $X_m = (\CC^*)^m$ with $(B^-)^{\leq 1} \subset \GL_m$ via the map
\[
(x_1, \ldots, x_m) \mapsto W(x_1, \ldots, x_m) = \begin{pmatrix}
x_1 & 0 & 0 & \cdots & 0 & 0 \\
1 & x_2 & 0 & \cdots & 0 & 0 \\
0 & 1 & x_3 & \cdots & 0 & 0 \\
\vdots& \ddots & \ddots& \ddots &\ddots& \vdots \\
0 & 0 & 0 & \ddots & x_{m-1} & 0 \\
0 & 0 & 0 & \cdots & 1 & x_m
\end{pmatrix}.
\]
By Example~\ref{ex:unipotent crystal}(1), this identification makes $X_m$ into a unipotent crystal, and one can easily verify that the geometric crystal structure on $X_m$ defined in the previous section is induced from this unipotent crystal. For example, the following calculation shows that the geometric crystal operator $e_2^c \colon X_4 \rightarrow X_4$ is induced from the action of $e_2^c$ on the unipotent crystal $(B^-)^{\leq 1}$:
\[
\begin{pmatrix}
1 & 0 & 0 & 0 \\
0 & 1 & (c-1)x_2 & 0 \\
0 & 0 & 1 & 0 \\
0 & 0 & 0 & 1
\end{pmatrix}
\begin{pmatrix}
x_1 & 0 & 0 & 0 \\
1 & x_2 & 0 & 0 \\
0 & 1 & x_3 & 0 \\
0 & 0 & 1 & x_4
\end{pmatrix}
\begin{pmatrix}
1 & 0 & 0 & 0 \\
0 & 1 & (c^{-1}-1)x_3 & 0 \\
0 & 0 & 1 & 0 \\
0 & 0 & 0 & 1
\end{pmatrix}
=
\begin{pmatrix}
x_1 & 0 & 0 & 0 \\
1 & cx_2 & 0 & 0 \\
0 & 1 & c^{-1}x_3 & 0 \\
0 & 0 & 1 & x_4
\end{pmatrix}.
\]

Given $\xx = (\xx^1, \ldots, \xx^n) \in (X_m)^n$, define
\[
M(\xx) = W(\xx^1) \cdots W(\xx^n).
\]
Theorem~\ref{thm:unipotent to geometric}(2) and Corollary~\ref{cor:unipotent property} imply that $\gamma \colon (X_m)^n \rightarrow T$ and $\ve_i, \vp_i \colon (X_m)^n \rightarrow \CC$ are computed by
\begin{equation}
\label{eq:basic functions M}
\gamma(\xx) = (M(\xx)_{1,1}, \ldots, M(\xx)_{m,m}),
\qquad\quad
\ve_i(\xx) = \dfrac{M(\xx)_{i+1,i+1}}{M(\xx)_{i+1,i}},
\qquad\quad
\vp_i(\xx) = \dfrac{M(\xx)_{i,i}}{M(\xx)_{i+1,i}},
\end{equation}
and that the geometric crystal operators on $(X_m)^n$ satisfy
\begin{equation}
\label{eq:basic crystal operators M}
M(e_i^c(\xx)) = x_i\bigl( (c-1)\varphi_i(\xx) \bigr) \cdot M(\xx) \cdot x_i\bigl( (c^{-1} - 1)\varepsilon_i(\xx) \bigr).
\end{equation}

\subsection{Combinatorial crystals}
\label{sec:comb crystal}

Let $\Lambda = \ZZ^m$ and $\Lambda^\vee = \Hom(\ZZ^m, \ZZ)$ denote the weight and coweight lattices associated to $\GL_m$. For $i \in [m-1]$, define the simple root $\tw{\alpha}_i \in \Lambda$ and simple coroot $\tw{\alpha}_i^\vee \in \Lambda^\vee$ by
\[
\tw{\alpha}_i = v_i - v_{i+1}, \qquad\qquad \tw{\alpha}_i^\vee(a_1 \ldots, a_m) = a_i - a_{i+1},
\]
where $v_i$ is the $i$th standard basis vector of $\ZZ^n$ and $a_i$ is the corresponding dual basis vector.

\begin{dfn}
\label{defn:comb crystal}
A \defn{$\GL_m$-combinatorial crystal} (or \defn{abstract crystal}) consists of a set $\mc{B}$, a \defn{weight map} $\tw{\gamma} \colon \mc{B} \rightarrow \Lambda$, and for each $i \in [m-1]$, functions $\tw{\ve}_i, \tw{\vp}_i \colon \mc{B} \rightarrow \ZZ$ and \defn{crystal operators} $\tw{e}_i, \tw{f}_i \colon \mc{B} \rightarrow \mc{B} \sqcup \{\zero\}$. Here the symbol $\zero$ represents an element which is not in $\mc{B}$, and one says that $\tw{e}_i(b)$ is \defn{undefined} if $\tw{e}_i(b) = \zero$. These maps must satisfy the following properties:
\begin{enumerate}
\item For all $b \in \mc{B}$, $\tw{\vp}_i(b) = \tw{\alpha}_i^\vee(\tw{\gamma}(b)) + \tw{\ve}_i(b)$;
\item If $b,b' \in \mc{B}$, then $\tw{e}_i(b) = b'$ if and only if $\tw{f}_i(b') = b$. In this case, one has
\[
\tw{\gamma}(b') = \tw{\gamma}(b) + \tw{\alpha}_i, \qquad\qquad \tw{\ve}_i(b') = \tw{\ve}_i(b) - 1, \qquad\qquad \tw{\vp}_i(b') = \tw{\vp}_i(b) + 1.
\]
\end{enumerate}
The combinatorial crystal $\mc{B}$ is said to be \defn{free} if the crystal operators $\tw{e}_i$ and $\tw{f}_i$ are mutually inverse bijections from $\mc{B} \rightarrow \mc{B}$, and \defn{regular} (or \defn{seminormal}) if for all $b \in \mc{B}$,
\begin{equation}
\label{eq:regular}
\tw{\ve}_i(b) = \max\{k > 0 \mid \tw{e}_i^k(b) \neq \zero\} \qquad \text{ and } \qquad \tw{\vp}_i(b) = \max\{k > 0 \mid \tw{f}_i^k(b) \neq \zero\}.
\end{equation}
\end{dfn}

\begin{ex}
\label{ex:one row crystal}
Let $\mc{B} = (\ZZ_{\geq 0})^m$. For $\mb{a} = (a_1, \ldots, a_m) \in \mc{B}$, define
\[
\tw{\gamma}(\mb{a}) = \mb{a}, \qquad\qquad \tw{\ve}_i(\mb{a}) = a_{i+1}, \qquad\qquad \tw{\vp}_i(\mb{a}) = a_i.
\]
Define $\tw{e}_i(\mb{a}) = (a_1, \ldots, a_i + 1, a_{i+1}-1, \ldots, a_m)$ if $a_{i+1} > 0$, and otherwise $\tw{e}_i(\mb{a}) = \zero$. Define $\tw{f}_i(\mb{a}) = (a_1, \ldots, a_i - 1, a_{i+1}+1, \ldots, a_m)$ if $a_i > 0$, and otherwise $\tw{f}_i(\mb{a}) = \zero$. These maps make $\mc{B}$ into a regular $\GL_m$-combinatorial crystal. Moreover, each of the subsets
\[
\mc{B}^L = \{(a_1, \ldots, a_m) \in (\ZZ_{\geq 0})^m \mid a_1 + \cdots + a_m = L\}
\]
is a finite regular combinatorial crystal.
\end{ex}

Kashiwara~\cite{K90,K91} proved that every finite-dimensional $\GL_m$-representation gives rise to a $\GL_m$-combinatorial crystal whose underlying set $\mc{B}$ is in bijection with a basis of the representation. The crystal operators $\tw{e}_i$ and $\tw{f}_i$ are a ``combinatorial approximation'' of the action of the Chevalley generators of the Lie algebra $\mathfrak{gl}_m$ on this basis. The map $\tw{\gamma}$ encodes the weight of a basis element with respect to the action of the Cartan subalgebra $\mathfrak{h} \subset \mathfrak{gl}_m$. The functions $\tw{\ve}_i, \tw{\vp}_i$ are defined by~\eqref{eq:regular} (so regularity is automatic). For example, the finite crystal $\mc{B}^L$ defined in Example~\ref{ex:one row crystal} arises from the $L$th symmetric power of the vector representation of $\GL_m$.

Given two $\GL_m$-combinatorial crystals $\mc{B}$ and $\mc{B}'$, Kashiwara defined the \defn{tensor product} $\mc{B} \otimes \mc{B}'$ to be the crystal on the set $\mc{B} \times \mc{B}'$, with
\begin{gather*}
\tw{\gamma}(b, b') = \tw{\gamma}(b) + \tw{\gamma}(b'),
\\
\tw{\ve}_i(b,b') = \tw{\ve}_i(b) + \tw{\ve}_i(b') - \min(\tw{\ve}_i(b), \tw{\vp}_i(b')), \qquad \tw{\vp}_i(b,b') = \tw{\vp}_i(b) + \tw{\vp}_i(b') - \min(\tw{\ve}_i(b), \tw{\vp}_i(b')),
\\
\tw{e}_i(b,b') = \begin{cases}
(\tw{e}_i(b), b') & \text{ if } \tw{\ve}_i(b) > \tw{\vp}_i(b') \\
(b, \tw{e}_i(b')) & \text{ if } \tw{\ve}_i(b) \leq \tw{\vp}_i(b') \\
\end{cases},
\qquad
\tw{f}_i(b,b') = \begin{cases}
(\tw{f}_i(b), b') & \text{ if } \tw{\ve}_i(b) \geq \tw{\vp}_i(b') \\
(b, \tw{f}_i(b')) & \text{ if } \tw{\ve}_i(b) < \tw{\vp}_i(b') \\
\end{cases}.
\end{gather*}
This definition has the property that if $\mc{B}, \mc{B}'$ arise from $\GL_m$-representations $V,V'$, then $\mc{B} \otimes \mc{B}'$ is the crystal arising from the tensor product $V \otimes V'$.
We follow the tensor product convention of~\cite{BumpSchill}, which is the opposite of Kashiwara's convention~\cite{K90,K91}.

\subsection{Tropicalization}
\label{sec:trop}

Let $x_1, \ldots, x_d$ be variables. For $I = (i_1, \ldots, i_d) \in (\ZZ_{\geq 0})^d$, let $x^I = x_1^{i_1} \cdots x_d^{i_d}$. A rational function $f \in \CC(x_1, \ldots, x_d)$ is \defn{positive} (or \defn{subtraction-free}) if it is nonzero, and has an expression of the form
\begin{equation}
\label{eq:positive expression}
f(x_1, \ldots, x_d) = \dfrac{\sum_I a_I x^I}{\sum_I b_I x^I}
\end{equation}
with $a_I,b_I \in \RR_{\geq 0}$, and $a_I,b_I = 0$ for all but finitely many $I$. If $f$ is positive, define the \defn{tropicalization} of $f$ to be the piecewise-linear function given by
\[
\Trop(f)(x_1, \ldots, x_d) = \min_{I \colon a_I \neq 0} (i_1x_1 + \cdots + i_dx_d) - \min_{I \colon b_I \neq 0} (i_1x_1 + \cdots + i_dx_d).
\]
We view $\Trop(f)$ as a map from $\ZZ^d \rightarrow \ZZ$. More generally, a rational map $f = (f_1, \ldots, f_k) \colon (\CC^*)^d \rightarrow (\CC^*)^k$ is positive if each $f_i$ is positive, and in this case its tropicalization is defined
by
\[
\Trop(f) = \bigl(\Trop(f_1), \ldots, \Trop(f_k)\bigr) \colon \ZZ^d \rightarrow \ZZ^k.
\]

\begin{ex}
The rational function $f = \dfrac{3x^2 + xy^3 + 5}{2x^5y^3 + 6xz}$ has tropicalization
\[
\Trop(f) = \min(2x, x + 3y, 0) - \min(5x + 3y, x+z).
\]
\end{ex}

It is straightforward to verify (see, \textit{e.g.},~\cite[Lem.~2.1.6]{BFZ}) that tropicalization is independent of the choice of positive expression~\eqref{eq:positive expression}, and that it satisfies
\begin{equation}
\label{eq:trop is a homom}
\Trop(f+g) = \min\bigl( \Trop(f),\Trop(g) \bigr), \qquad \Trop(fg^{\pm 1}) = \Trop(f) \pm \Trop(g).
\end{equation}
In other words, tropicalization is a homomorphism from the semi-field of positive rational functions over $\CC$ with operations $(+,\cdot,\div)$ to the semi-field of piecewise-linear functions over $\ZZ$ with operations $(\min,+,-)$.

We say that a positive rational function $f$ is a \defn{geometric lift} of a piecewise-linear function $F$ if $\Trop(f) = F$; note that every piecewise-linear function has many geometric lifts. We will make frequent use of a particular geometric lift of the maximum function. Given positive rational functions $f_1, \ldots, f_k$, define the \defn{geometric maximum} of $f_1, \ldots, f_k$ by
\begin{equation}
\label{eq:gMax}
\gMax(f_1, \ldots, f_k) = \dfrac{1}{f_1^{-1} + \cdots + f_k^{-1}}.
\end{equation}
Using~\eqref{eq:trop is a homom}, we see that
\[
\Trop\bigl( \gMax(f_1, \ldots, f_k) \bigr) = -\min\bigl( -\Trop(f_1), \ldots, -\Trop(f_k) \bigr) = \max\bigl( \Trop(f_1), \ldots, \Trop(f_k) \bigr).
\]

\subsection{Tropicalization of geometric crystals}
\label{sec:geom to comb}

Let $X$ be a $\GL_m$-geometric crystal of dimension $d$. A \defn{parametrization} of $X$ is a birational isomorphism $\Phi \colon (\CC^*)^d \rightarrow X$. Given a parametrization $\Phi$ of $X$, we call the pair $(X,\Phi)$ a \defn{positive $\GL_m$-geometric crystal} if for each $i \in [m-1]$, the rational functions $\gamma \Phi, \ve_i \Phi, \vp_i \Phi \colon (\CC^*)^d \rightarrow \CC$ are positive, and the rational map $\Phi^{-1}e_i\Phi \colon (\CC^*)^{d+1} \rightarrow (\CC^*)^d$ given by
\[
(c, x_1, \ldots, x_d) \mapsto \Phi^{-1} e_i^c \Phi(x_1, \ldots, x_d)
\]
is positive. Given a positive $\GL_m$-geometric crystal $(X,\Phi)$ of dimension $d$, set $\mc{B} = \ZZ^d$, and define $\tw{\gamma} \colon \mc{B} \rightarrow \ZZ^m$, $\tw{\ve}_i, \tw{\vp}_i : \mc{B} \rightarrow \ZZ$, and $\tw{e}_i, \tw{f}_i \colon \mc{B} \rightarrow \mc{B}$ by
\begin{equation}
\label{eq:free maps}
\tw{\gamma} = \Trop(\gamma \Phi), \qquad\quad \tw{\ve}_i = \Trop(\ve_i \Phi), \qquad\quad \tw{\vp}_i = \Trop(\vp_i \Phi),
\end{equation}
\[
\tw{e}^{{\rm free}}_i = \Trop(\Phi^{-1} e_i^c \Phi)|_{c = 1}, \qquad\quad \tw{f}^{{\rm free}}_i = \Trop(\Phi^{-1} e_i^c \Phi)|_{c=-1}.
\]
By comparing axioms (1) and (2) in Definitions~\ref{defn:geom crystal} and~\ref{defn:comb crystal}, one sees immediately that these maps make $\mc{B}$ into a free $\GL_m$-combinatorial crystal. We call $\mc{B}$ the \defn{tropicalization of $X$ (with respect to $\Phi$)}, and we denote it by $\Trop(X)$.

\begin{remark}
In general, a positive geometric crystal associated to a reductive group $G$ tropicalizes to a free combinatorial crystal for the Langlands dual group $G^\vee$. This is due to the fact that the roles of $\alpha_i$ and $\alpha_i^\vee$ in Definition~\ref{defn:geom crystal} are played by $\tw{\alpha}_i^\vee$ and $\tw{\alpha}_i$ in Definition~\ref{defn:comb crystal}. We are able to ignore this subtlety because $\GL_m$ is its own Langlands dual.
\end{remark}

In order to recover the crystals of finite-dimensional representations, Berenstein and Kazhdan introduced in~\cite{BK07} an elegant mechanism for ``cutting out'' finite vertex sets from $\ZZ^d$, which we now review. A \defn{decoration} on a $\GL_m$-geometric crystal $X$ is a rational function $F : X \rightarrow \CC$ such that
\begin{equation}
\label{eq:dec}
F(e_i^c(x)) = F(x) + (c-1)\vp_i(x) + (c^{-1}-1)\ve_i(x).
\end{equation}
A \defn{positive decorated $\GL_m$-geometric crystal} is a triple $(X,\Phi,F)$, where $(X,\Phi)$ is a positive geometric crystal, $F$ is a decoration on $X$, and $F \Phi$ is positive.

\begin{thm}[{\cite[Prop.~6.6, 6.7]{BK07}}]
\label{thm:geom to comb}
Let $(X,\Phi,F)$ be a positive decorated $\GL_m$-geometric crystal of dimension $d$. Define
\[
\mc{B}_F = \{b \in \ZZ^d \mid \Trop(F \Phi)(b) \geq 0\}.
\]
Let $\tw{\gamma}, \tw{\ve}_i, \tw{\vp}_i$ be the restrictions to $\mc{B}_F$ of the maps~\eqref{eq:free maps}, and define $\tw{e}_i, \tw{f}_i \colon \mc{B}_F \rightarrow \mc{B}_F \sqcup \{\zero\}$ by
\[
\tw{e}_i(b) = \begin{cases}
\tw{e}_i^{{\rm free}}(b) & \text{if } \tw{e}_i^{{\rm free}}(b) \in \mc{B}_F, \\
\zero & \text{otherwise,}
\end{cases}
\qquad\qquad
\tw{f}_i(b) = \begin{cases}
\tw{f}_i^{{\rm free}}(b) & \text{if } \tw{f}_i^{{\rm free}}(b) \in \mc{B}_F, \\
\zero & \text{otherwise.}
\end{cases}
\]
\begin{enumerate}
\item These maps make $\mc{B}_F$ into a regular $\GL_m$-combinatorial crystal. We call this crystal the \defn{tropicalization of $(X,F)$ (with respect to $\Phi$)}, and denote it by $\Trop(X,F)$.
\item Suppose $(X',\Phi',F')$ is another positive decorated $\GL_m$-geometric crystal. Define $\Phi \times \Phi' \colon (\CC^*)^d \times (\CC^*)^{d'} \rightarrow X \times X'$ and $F+F' \colon X \times X' \rightarrow \CC$ by
\[
(\Phi \times \Phi')(z,z') = \bigl(\Phi(z),\Phi(z')\bigr), \qquad\qquad (F+F')(x,x') = F(x) + F(x').
\]
The triple $(X \times X', \Phi \times \Phi', F+F')$ is a positive decorated $\GL_m$-geometric crystal, and $\Trop(X \times X', F+F') = \Trop(X,F) \otimes \Trop(X',F')$.
\end{enumerate}
\end{thm}

\begin{ex}
\label{ex:basic dec}
The function $F \colon X_m \rightarrow \CC$ given by $F(x_1, \ldots, x_m) = x_1 + \ldots + x_m$ is easily seen to be a decoration. The tropicalization of $(X_m,F)$ (with respect to $\Phi = \Id$) is the regular combinatorial crystal $\mc{B}$ defined in Example~\ref{ex:one row crystal}. Furthermore, Theorem~\ref{thm:geom to comb}(2) implies that the function
\[
F(\xx^1, \ldots, \xx^n) = \sum_{i \in [m], j \in [n]} x_i^j
\]
is a decoration on $(X_m)^n$, and $\bigl((X_m)^n,F\bigr)$ tropicalizes to the $n$-fold tensor product of $\mc{B}$.
\end{ex}

\section{Gelfand--Tsetlin geometric crystal}
\label{sec:GT}

\subsection{Gelfand--Tsetlin patterns}
\label{sec:GT patterns}

A \defn{Gelfand--Tsetlin (GT) pattern of height $n$} is a triangular array of nonnegative integers $(a_{i,j}), 1 \leq i \leq j \leq n,$ satisfying the inequalities 
\begin{equation}
\label{eq:GT def}
a_{i,j+1} \geq a_{i,j} \geq a_{i+1,j+1}
\end{equation}
whenever $j < n$. We will represent GT patterns as triangles whose entries weakly increase from right to left along every diagonal, as illustrated below in the case $n = 4$:
\[
\begin{array}{ccccccc}
&&& a_{11} \\
&& a_{12} && a_{22} \\
& a_{13} && a_{23} && a_{33} \\
a_{14} && a_{24} && a_{34} && a_{44}
\end{array} \, .
\]

GT patterns of height $n$ are in bijection with semistandard tableaux consisting of entries at most $n$: the partition $(a_{1,j}, \ldots, a_{j,j})$ in the $j$th row corresponds to the shape formed by the entries less than or equal to $j$ in the tableau. In particular, the bottom row $(a_{1,n}, \ldots, a_{n,n})$ is the shape of the tableau, so we define the \defn{shape} of the GT pattern to be this $n$-tuple. Another description of the bijection is that the number of $j$'s in the $i$th row of the tableau is $a_{i,j} - a_{i,j-1}$ (with $a_{i,i-1} = 0$).

\ytableausetup{centertableaux}
\begin{ex}
Here is a GT pattern of height $4$ and the corresponding semistandard tableau:
\[
\begin{array}{ccccccc}
&&& 3 \\
&& 6 && 1 \\
& 6 && 4 && 1 \\
8 && 5 && 3 && 0
\end{array}
\quad \longleftrightarrow \quad
\ytableaushort{11122244,23334,344} \;.
\]
\end{ex}

More generally, define a \defn{Gelfand--Tsetlin pattern of height $n$ and width $m$} to be a trapezoidal array of nonnegative integers $(a_{i,j})_{1 \leq i \leq m, i \leq j \leq n}$ which satisfy the inequalities~\eqref{eq:GT def} whenever both sides of an inequality are in the array. For example, a GT pattern of height 5 and width 3 looks like this:
\[
\begin{array}{ccccccc}
&&&& a_{11} \\
&&& a_{12} && a_{22} \\
&& a_{13} && a_{23} && a_{33} \\
& a_{14} && a_{24} && a_{34} \\
a_{15} && a_{25} && a_{35}
\end{array} \; .
\]
These patterns are in bijection with semistandard tableaux with entries at most $n$, and at most $m$ rows; again, we define the shape of the GT pattern to be the bottom row $(a_{1,n}, \ldots, a_{\min(m,n),n})$, since this is the shape of the corresponding tableau. Note that if $m \geq n$, then a GT pattern of height $n$ and width $m$ is simply a GT pattern of height $n$.

At the geometric level, we consider the torus
\begin{equation}
\label{eq:GT_defn}
\GT_n^{\leq m} = \{\mb{z} = (z_{i,j}) \mid z_{i,j} \in \CC^*, 1 \leq i \leq m, i \leq j \leq n\}.
\end{equation}
We refer to points of $\GT_n^{\leq m}$ as \defn{patterns}, and we define the \defn{shape} of $\mb{z} \in \GT_n^{\leq m}$ to be the vector
\begin{equation}
\label{eq:GT_shape}
\sh(\mb{z}) = (z_{1,n}, \ldots, z_{p,n}) \in (\CC^*)^p,
\end{equation}
where $p = \min(m,n)$.

\subsection{Geometric crystal structure on $\GT_n^{\leq m}$}
\label{sec:GT geom cryst}

In this section we work with $\GL_n$-geometric crystals rather than $\GL_m$-geometric crystals. Let $B^-$ denote the subgroup of lower triangular matrices in $\GL_n(\CC)$, and consider the subsets of $B^-$ defined by
\[
(B^-)^{\leq m} = \{A \in \GL_n(\CC) \mid A_{ij} = 0 \text{ if } i < j \text{ or } i-j > m, \text{ and } A_{ij} = 1 \text{ if } i-j = m\}
\]
(note that $(B^-)^{\leq m} = B^-$ for $m \geq n$). By Example~\ref{ex:unipotent crystal}(a), $(B^-)^{\leq m}$ has the structure of a $\GL_n$-unipotent crystal, and therefore of a $\GL_n$-geometric crystal. The dimension of $(B^-)^{\leq m}$ is equal to the dimension of the torus $\GT_n^{\leq m}$. We will introduce a parametrization $\Phi_n^{\leq m} \colon \GT_n^{\leq m} \rightarrow (B^-)^{\leq m}$, and then obtain a geometric crystal structure on the torus $\GT_n^{\leq m}$ by ``pulling back'' the geometric crystal structure on $(B^-)^{\leq m}$. In \S \ref{sec:dec}, we give explicit positive formulas for this geometric crystal structure that make no reference to $(B^-)^{\leq m}$.

For $i \in [n]$, define an $n \times n$ matrix
\begin{equation}
\label{eq:row matrix}
W^i(z_i, \ldots, z_n) = \sum_{k = 1}^{i-1} E_{kk} + \sum_{k = i}^n z_k E_{kk} + \sum_{k = i}^{n-1} E_{k+1,k},
\end{equation}
where $E_{i,j}$ is a matrix unit as in \S \ref{sec:unipotent}. Set $p = \min(m,n)$. Given $\mb{z} = (z_{i,j}) \in \GT_n^{\leq m}$, define
\[
\Phi_n^{\leq m}(\mb{z}) = W^p\left(z_{p,p}, \frac{z_{p,p+1}}{z_{p,p}}, \ldots, \frac{z_{p,n}}{z_{p,n-1}}\right) \cdots W^1\left(z_{1,1}, \frac{z_{1,2}}{z_{1,1}}, \ldots, \frac{z_{1,n}}{z_{1,n-1}}\right).
\]

\begin{ex}
\label{ex:Phi}
For $n=4$ and $m = 2$, we have
\[
\Phi_4^{\leq 2}(\zz) =
\begin{pmatrix}
1 & 0 & 0 & 0 \\
0 & z_{22} & 0 & 0 \\
0 & 1 & \frac{z_{23}}{z_{22}} & 0 \\
0 & 0 & 1 & \frac{z_{24}}{z_{23}}
\end{pmatrix}
\begin{pmatrix}
z_{11} & 0 & 0 & 0 \\
1 & \frac{z_{12}}{z_{11}} & 0 & 0 \\
0 & 1 & \frac{z_{13}}{z_{12}} & 0 \\
0 & 0 & 1 & \frac{z_{14}}{z_{13}}
\end{pmatrix}
=
\begin{pmatrix}
z_{11} & 0 & 0 & 0 \\
z_{22} & \frac{z_{12}z_{22}}{z_{11}} & 0 & 0 \\
1 & \frac{z_{12}}{z_{11}} + \frac{z_{23}}{z_{22}} & \frac{z_{13}z_{23}}{z_{12}z_{22}} & 0 \\
0 & 1 & \frac{z_{13}}{z_{12}} + \frac{z_{24}}{z_{23}} & \frac{z_{14}z_{24}}{z_{13}z_{23}}
\end{pmatrix}.
\]
\end{ex}

We now introduce a very useful tool for working with the map $\Phi_n^{\leq m}$. Let $\Gamma_n^{\leq m}$ be a planar, edge-weighted, directed network, with
\begin{itemize}
\item vertex set
\[
V = \{(x,y) \mid x \in [0,n], y \in [0,p], x+y \leq n\} \setminus \{(0,0)\};
\]
\item a vertical edge of weight 1 from $(x,y)$ to $(x,y-1)$ whenever $x \geq 1$ and both vertices are in $V$;
\item a diagonal edge of weight $z_{y,x+y}/z_{y,x+y-1}$ from $(x,y)$ to $(x+1,y-1)$ whenever both vertices are in $V$ (with $z_{y,y-1} = 1$).
\end{itemize}
This network has $n$ source vertices $(0,1), \ldots, (0,p), (1,p), \ldots, (n-m,p)$ and $n$ sink vertices $(1,0), \ldots, (n,0)$, which we label $1, \ldots, n$ and $1', \ldots, n'$, respectively. Examples of this network are shown in Figure~\ref{fig:network}.

\begin{figure}
\begin{center}
\begin{tikzpicture}

\foreach \a/\b/\c/\d in {0/1/1/0, 0/2/2/0, 0/3/3/0, 0/4/4/0, 1/3/1/0, 2/2/2/0, 3/1/3/0} {\draw (\a,\b) -- (\c,\d);}
\foreach \a/\b in {0/1,0/2,0/3,0/4} {\filldraw (\a,\b) circle[radius=.04cm] node[left]{$\b$}; \filldraw (\b,\a) circle[radius=.04cm] node[below]{$\b'$};}
\foreach \a in {1,2,3,4} {\draw (0.6,\a-0.2) node{$z_{\a\a}$};}
\foreach \a/\b in {1/2,2/3,3/4} {\draw (1.6,\a-0.2) node{$\frac{z_{\a\b}}{z_{\a\a}}$};}
\foreach \a/\b/\c in {1/2/3,2/3/4} {\draw (2.6,\a-0.2) node{$\frac{z_{\a\c}}{z_{\a\b}}$};}
\draw (3.6,0.8) node{$\frac{z_{14}}{z_{13}}$};

\begin{scope}[xshift=7cm]
\foreach \a/\b/\c/\d in {0/1/1/0, 0/2/2/0, 0/3/3/0, 0/4/4/0, 1/4/5/0, 2/4/6/0, 3/4/7/0, 1/4/1/0, 2/4/2/0, 3/4/3/0, 4/3/4/0, 5/2/5/0, 6/1/6/0} {\draw (\a,\b) -- (\c,\d);}
\foreach \a/\b in {0/1,0/2,0/3,0/4} {\filldraw (\a,\b) circle[radius=.04cm] node[left]{$\b$}; \filldraw (\b,\a) circle[radius=.04cm] node[below]{$\b'$};}
\foreach \a/\b/\c in {1/4/5, 2/4/6, 3/4/7} {\filldraw (\a,\b) circle[radius=.04cm] node[above]{$\c$};}
\foreach \a/\b in {0/4, 0/5, 0/6, 0/7} {\filldraw (\b,\a) circle[radius=.04cm] node[below]{$\b'$};}
\foreach \a in {1,2,3,4} {\draw (0.6,\a-0.2) node{$z_{\a\a}$};}
\foreach \a/\b in {1/2,2/3,3/4,4/5} {\draw (1.6,\a-0.2) node{$\frac{z_{\a\b}}{z_{\a\a}}$};}
\foreach \a/\b/\c in {1/2/3,2/3/4,3/4/5,4/5/6} {\draw (2.6,\a-0.2) node{$\frac{z_{\a\c}}{z_{\a\b}}$};}
\foreach \a/\b/\c in {1/3/4,2/4/5,3/5/6,4/6/7} {\draw (3.6,\a-0.2) node{$\frac{z_{\a\c}}{z_{\a\b}}$};}
\foreach \a/\b/\c in {1/4/5,2/5/6,3/6/7} {\draw (4.6,\a-0.2) node{$\frac{z_{\a\c}}{z_{\a\b}}$};}
\foreach \a/\b/\c in {1/5/6,2/6/7} {\draw (5.6,\a-0.2) node{$\frac{z_{\a\c}}{z_{\a\b}}$};}
\draw (6.6,0.8) node{$\frac{z_{17}}{z_{16}}$};
\end{scope}

\end{tikzpicture}
\end{center}
\caption{Examples of the planar network $\Gamma_n^{\leq m}$. The network on the left is $\Gamma_4$ (\textit{i.e.}, $\Gamma_4^{\leq m}$ for any $m \geq 4$), and the network on the right is $\Gamma_7^{\leq 4}$. Edges are directed to the south and southeast, and all vertical edges have weight 1.}
\label{fig:network}
\end{figure}
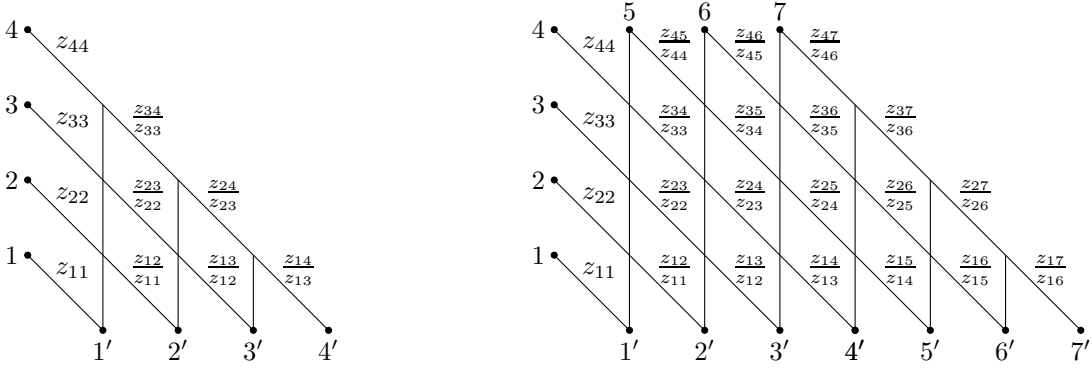

Given a planar, edge-weighted, directed network $\Gamma$ with distinguished sources $1, \ldots, n$ and sinks $1', \ldots, n'$, we associate an $n \times n$ matrix $A(\Gamma)$ as follows: the weight of a path in the network is defined to be the sum of the weights of the edges in the path, and the entry $A(\Gamma)_{i,j}$ is defined to be the sum of the weights of all (directed) paths from source $i$ to sink $j'$. Using the fact that concatenation of networks corresponds to multiplication of the associated matrices, one sees that $\Phi_n^{\leq m}(\zz) = A(\Gamma_n^{\leq m})$.

It is clear from inspection of the network $\Gamma_n^{\leq m}$ that $\Phi_n^{\leq m}(\zz) \in (B^-)^{\leq m}$. To show that $\Phi_n^{\leq m}$ is a birational isomorphism, we will give an explicit formula for its (birational) inverse. For an $n \times n$ matrix $A$ and subsets $I,J \subseteq [n]$ of the same size, let $\Delta_{I,J}(A)$ be the determinant of the submatrix of $A$ consisting of the rows in $I$ and the columns in $J$. When $J$ consists of the first several columns, we will often omit it from the notation, so that $\Delta_I(A) = \Delta_{I,[1,|I|]}(A)$. We call the minor $\Delta_I(A)$ a \defn{flag minor}. We use the convention that $\Delta_{\emptyset}(A) = \Delta_{\emptyset,\emptyset}(A) = 1$. Define a rational map $\Psi_n^{\leq m} \colon (B^-)^{\leq m} \rightarrow \GT_n^{\leq m}$ by $A \mapsto \mb{z} = (z_{i,j})$, where
\begin{equation}
\label{eq:Psi def}
z_{i,j} = \frac{\Delta_{[i,j]}(A)}{\Delta_{[i+1,j]}(A)}
\end{equation}
for $1 \leq i \leq m, i \leq j \leq n$.

When $m \geq n$, the dependence on $m$ in the above definitions disappears. In this case, we will sometimes omit the superscript ``$\leq m$'' and write $B^-, \GT_n, \Phi_n, \Psi_n,\Gamma_n$.

The following lemma appears in~\cite{NoumiYamada} but goes back at least to~\cite{BFZ}.

\begin{lemma}
\label{lem:Phi inverse}
The map $\Phi_n^{\leq m}$ is a birational isomorphism from $\GT_n^{\leq m}$ to $(B^-)^{\leq m}$, with birational inverse $\Psi_n^{\leq m}$. In particular, $\Phi_n$ is a birational isomorphism from $\GT_n$ to $B^-$, with birational inverse $\Psi_n$.
\end{lemma}

\begin{proof}
The Lindstr\"om/Gessel--Viennot Lemma~\cite{Lindstrom, GesselViennot} says that the minor $\Delta_{I,J}(A(\Gamma))$ is equal to the the sum of the weights of non-intersecting collections of paths in $\Gamma$ from the sources $I$ to the sinks $J'$ (the weight of a collection of paths is the product of all edges in the union of the paths, and non-intersecting means no two paths share a vertex). For $i,j$ satisfying $1 \leq i \leq m$ and $i \leq j \leq n$, there is exactly one non-intersecting collection of paths in $\Gamma_n^{\leq m}$ from sources $[i,j]$ to sinks $[1,j-i+1]'$, and this collection has weight $z_{i,j} z_{i+1,j} \cdots z_{j,j}$. Thus, the ratio of minors appearing in~\eqref{eq:Psi def} is equal to $z_{i,j}$, so $\Psi_n^{\leq m} \circ \Phi_n^{\leq m} (\zz) = \zz$.

It remains to show that there is a (non-empty) Zariski open subset $V \subset (B^-)^{\leq m}$ such that $\Phi_n^{\leq m} \circ \Psi_n^{\leq m}(A) = A$ for all $A \in V$. Let $V$ be the subset where the flag minors $\Delta_{[i,j]}$ are non-vanishing for all $1 \leq i \leq m, \, i \leq j \leq n$. If $A \in V$, then it follows from the above argument that $\Phi_n^{\leq m} \circ \Psi_n^{\leq m}(A)$ has the same flag minors as $A$. Now one argues that a point $A \in V$ is uniquely determined by its flag minors: the non-zero entries in the first column of $A$ are equal to $\Delta_{[i,i]}(A)$, the non-zero entries in the second column of $A$ are determined (from top to bottom) by $\Delta_{[i,i+1]}(A)$ and the entries in the first column, etc.
\end{proof}

\begin{dfn}
\label{defn:GT geom cryst}
Define $\ov{\gamma} \colon \GT_n^{\leq m} \rightarrow (\CC^*)^n$ and $\ov{\ve}_j, \ov{\vp}_j \colon \GT_n^{\leq m} \rightarrow \CC$ (for $j \in [n-1]$) by
\[
\ov{\gamma}(\mb{z}) = (M_{1,1}, \ldots, M_{n,n}), \qquad\qquad \ov{\ve}_j(\mb{z}) = \frac{M_{j+1,j+1}}{M_{j+1,j}}, \qquad\qquad \ov{\vp}_j(\mb{z}) = \frac{M_{j,j}}{M_{j+1,j}},
\]
where $M = \Phi_n^{\leq m}(\mb{z})$. Define $\ov{e}_j \colon \CC^* \times \GT_n^{\leq m} \rightarrow \GT_n^{\leq m}$ by
\begin{equation}
\label{eq:e def}
\ov{e}_j^c(\mb{z}) =  \Psi_n^{\leq m}\bigg(x_j\left((c-1)\ov{\vp}_j(\mb{z})\right) \cdot \Phi_n^{\leq m}(\mb{z}) \cdot x_j\left((c^{-1}-1) \ov{\ve}_j(\mb{z})\right)\bigg).
\end{equation}
\end{dfn}

It follows from Theorem~\ref{thm:unipotent to geometric}(1) and Corollary~\ref{cor:unipotent property} that these maps make $\GT_n^{\leq m}$ into a $\GL_n$-geometric crystal, which we call the \defn{Gelfand--Tsetlin geometric crystal}.

\begin{remark}
The geometric crystal $\GT_n^{\leq 1}$ is isomorphic to the basic $\GL_n$-geometric crystal defined in \S \ref{sec:basic}, via the map $(z_{1,1}, \ldots, z_{1,n}) \mapsto (z_{1,1}, z_{1,2}/z_{1,1}, \ldots, z_{1,n}/z_{1,n-1})$.
\end{remark}

\medskip

The following result shows that a fundamental property of the combinatorial crystal operators on semistandard tableaux carries over to the geometric setting. This result can be viewed as a special case of~\cite[Claim 2.9]{BK07}.

\begin{lemma}
\label{lem:shape preserved}
The geometric crystal operators on $\GT_n^{\leq m}$ preserve the shape of a pattern.
\end{lemma}

\begin{proof}
By definition, the shape of $\ov{e}_j^c(\mb{z})$ is given by
\begin{equation}
\label{eq_sh_e}
\left(\dfrac{\Delta_{[1,n]}(M')}{\Delta_{[2,n]}(M')}, \ldots, \dfrac{\Delta_{[p,n]}(M')}{\Delta_{[p+1,n]}(M')}\right),
\end{equation}
where $M'$ is the matrix inside the large parentheses in~\eqref{eq:e def}, and $p = \min(m,n)$. Multiplying a matrix by an element of $U$ (on either side) does not change the bottom-left justified minors, so we may replace $M'$ with $\Phi_n^{\leq m}(\mb{z})$ in~\eqref{eq_sh_e}. By Lemma~\ref{lem:Phi inverse}, the resulting vector is equal to the shape of $\mb{z}$.
\end{proof}

\subsection{Tropicalization of the Gelfand--Tsetlin geometric crystal}
\label{sec:dec}

In order to recover the combinatorial crystals associated to finite-dimensional $\GL_n$-modules from $\GT_n^{\leq m}$, we first introduce a decoration.

\begin{dfn}
\label{defn:GT dec}
For $\mb{z} = (z_{i,j})_{1 \leq i \leq m, i \leq j \leq n} \in \GT_n^{\leq m}$, define
\[
F(\zz) = \sum_{\substack{1 \leq i \leq m \\ \ i \leq j \leq n-1}} \frac{z_{i,j+1}}{z_{i,j}} \quad + \quad \sum_{\substack{1 \leq i \leq m-1 \\ \ i \leq j \leq n-1}} \frac{z_{i,j}}{z_{i+1,j+1}} \quad + \quad \mathbbm{1}_{m < n}z_{n,n},
\]
where $\mathbbm{1}_{m < n}$ is the indicator function for $m < n$.
\end{dfn}

In the case $m \geq n$, this formula appears in~\cite{Lam13} (with a slightly different indexing convention).

\begin{remark}
\label{rem:dec}
An integer array $(a_{i,j})_{1 \leq i \leq m, i \leq j \leq n}$ satisfies $\Trop(F)(a_{i,j}) \geq 0$ if and only if the $a_{i,j}$ satisfy the interlacing inequalities~\eqref{eq:GT def}, and (in the case $m < n$) the coordinate $a_{m,m}$ is nonnegative. The interlacing inequalities imply that $a_{m,m}$ is less than or equal to all other entries $a_{i,j}$, so for $m < n$, $\Trop(F)(a_{i,j}) \geq 0$ if and only if $(a_{i,j})$ is a Gelfand--Tsetlin pattern. For $m \geq n$, the inequality $\Trop(F)(a_{i,j}) \geq 0$ does not require the $a_{i,j}$ to be nonnegative; however, for a fixed partition shape $\la = (a_{1,n} \geq \ldots \geq a_{n,n} \geq 0)$, an array $(a_{i,j})$ of shape $\la$ satisfies $\Trop(F)(a_{i,j}) \geq 0$ if and only if it is a GT pattern. (Adding a constant to all entries $a_{i,j}$ corresponds to tensoring with a power of the determinant representation of $\GL_n$, so nothing fundamentally new is obtained by allowing the shape to have negative entries.)
\end{remark}

Remark~\ref{rem:dec} explains that the tropicalization of $F$ cuts out the underlying sets of finite-dimensional $\GL_n$-crystals, but it is not clear that $F$ satisfies~\eqref{eq:dec}. To prove this, we express $F$ in terms of minors of the matrix $\Phi_n^{\leq m}(\zz)$, and appeal to a general result of Berenstein and Kazhdan.

\begin{lemma}
\label{lem:dec formula}
The decoration on $\GT_n^{\leq m}$ is given by
\begin{multline*}
F(\zz) = \sum_{k = 1}^{\min(m-1,n-1)} \dfrac{\Delta_{\{k\} \cup [k+2,n],[1,n-k]}(M) + \Delta_{[k+1,n],[1,n-k-1] \cup \{n-k+1\}}(M)}{\Delta_{[k+1,n],[1,n-k]}(M)} \\
+ \mathbbm{1}_{m<n} \left(\dfrac{\Delta_{\{m\} \cup [m+2,n],[1,n-m]}(M)}{\Delta_{[m+1,n],[1,n-m]}(M)} + \sum_{j = 1}^{n-m} \frac{\Delta_{[m+1,m+j],[1,j-1] \cup \{j+1\}}(M)}{\Delta_{[m+1,m+j],[1,j]}(M)}\right),
\end{multline*}
where $M = \Phi_n^{\leq m}(\zz)$.
\end{lemma}

The proof appears in Appendix~\ref{app:proofs}. The expression for $F$ appearing in Lemma~\ref{lem:dec formula} is a special case of Berenstein and Kazhdan's general formula for the decoration of a unipotent bicrystal of parabolic type. Indeed, in~\cite[Cor.~1.25]{BK07}, take $P$ corresponding to nodes $\{m+1, \ldots, n-1\}$ in the $A_{n-1}$ Dynkin diagram. This implies that $F$ does in fact satisfy~\eqref{eq:dec}.

Next, one must show that the geometric crystal $(\GT_n^{\leq m}, \Id)$ is positive, as this is not at all obvious from the definition~\eqref{eq:e def} of the geometric crystal operators. We will give explicit positive formulas for the geometric crystal structure on $\GT_n$ (it is straightforward to generalize these formulas to the case of $\GT_n^{\leq m}$, but we focus on the case $m \geq n$ for clarity). In order to state the formulas compactly, we introduce the \defn{diamond ratio}
\[
\phi_{i,j} = \phi_{i,j}(\zz) =
\dfrac{z_{i-1,j}z_{i,j}}{z_{i-1,j-1}z_{i,j+1}}
\]
for $2 \leq i \leq j \leq n-1$. (The entries of $\zz$ that contribute to $\phi_{i,j}$ form the diamond \hspace{-5ex}
$\begin{matrix}
&&&& z_{i-1,j-1} \\
&&& z_{i-1,j} && z_{i,j} \\
&&  && z_{i,j+1} && \\
\end{matrix}$.)
We also use the rational function $\gMax$ defined in~\eqref{eq:gMax}, with the convention that $\gMax(\emptyset) = 1$.

\begin{lemma}
\label{lem:explicit crystal formulas}
The $\GL_n$-geometric crystal structure on $\GT_n$ is given by the explicit formulas
\[
\ov{\gamma}(\mb{z}) = \left(z_{1,1}, \dfrac{z_{1,2}z_{2,2}}{z_{1,1}}, \ldots, \dfrac{\prod_{i=1}^n z_{i,n}}{\prod_{i=1}^{n-1} z_{i,n-1}}\right),
\]
\bigskip
\[
\ov{\varepsilon}_j(\mb{z}) = \dfrac{z_{1,j+1}}{z_{1,j}} \gMax_{1 \leq k \leq j} \left( \prod_{i=2}^k \phi_{i,j}^{-1} \right), \qquad\qquad \ov{\varphi}_j(\mb{z}) = \dfrac{z_{j,j}}{z_{j+1,j+1}} \gMax_{1 \leq k \leq j} \left( \prod_{i=k+1}^j \phi_{i,j} \right),
\]
\bigskip
\[
\ov{e}_j^c(\mb{z}) = (\mb{z}'), \qquad \text{ where } \qquad z_{i,r}' = \begin{cases}
z_{i,j} \dfrac{C_{i,j}}{C_{i+1,j}} & \text{if } r = j, \medskip \\
z_{i,r} & \text{otherwise,}
\end{cases}
\qquad
\text{ and } \qquad
C_{i,j} = \sum_{k=1}^j c^{\mathbbm{1}_{k \geq i}} \prod_{\ell=2}^k \phi_{\ell, j}.
\]
\end{lemma}

The proof appears in Appendix~\ref{app:proofs}.

\begin{ex}
\label{ex:explicit crystal formulas}
If $\zz \in \GT_n$ with $n \geq 4$, then we have
\[
\ov{\varepsilon}_3(\zz) = \dfrac{z_{14}}{z_{13}} \dfrac{1}{1 + \phi_{23} + \phi_{23}\phi_{33}}, \qquad \ov{\varphi}_3(\zz) = \dfrac{z_{33}}{z_{44}} \dfrac{1}{\phi_{33}^{-1}\phi_{23}^{-1} + \phi_{33}^{-1} + 1}, \bigskip
\]
and $\ov{e}_3^c$ acts on the third row of $\zz$ by
\[
z_{13} \quad z_{23} \quad z_{33} \quad \mapsto \quad
z_{13}\dfrac{c+ c\phi_{23} + c\phi_{23}\phi_{33}}{1+ c\phi_{23} + c\phi_{23}\phi_{33}} \quad
z_{23} \dfrac{1+ c\phi_{23} + c\phi_{23}\phi_{33}}{1+ \phi_{23} + c\phi_{23}\phi_{33}} \quad
z_{33}\dfrac{1+ \phi_{23} + c\phi_{23}\phi_{33}}{1+ \phi_{23} + \phi_{23}\phi_{33}}, \bigskip
\]
and leaves the other entries unchanged.
\end{ex}

Finally, we use the explicit formulas of Lemma~\ref{lem:explicit crystal formulas} to show the the tropicalization of $(\GT_n, F)$ is a piecewise-linear description of the usual crystal structure on semistandard tableaux (see, \textit{e.g.},~\cite{K02book,BumpSchill}). For a partition $\la$ with at most $n$ parts, let $\mc{B}(\lambda)$ denote the $\GL_n$-combinatorial crystal corresponding to the irreducible $\GL_n$-representation with highest weight $\lambda$.

\begin{prop}
The tropicalization of $(\GT_n, F)$ is isomorphic to $\bigsqcup_{\la} \mc{B}(\lambda)$, where the union is taken over all partitions with at most $n$ parts.
\end{prop}

\begin{proof}
In keeping with our usual convention, we write $\tw{\ov{\gamma}}, \tw{\ov{\ve}}_j, \tw{\ov{\vp}}_j, \tw{\ov{e}}_j, \tw{\ov{f}}_j$ for the maps of a $\GL_n$-crystal.

Let $\mb{a} = (a_{i,j})$ be a Gelfand--Tsetlin pattern of height $n$, and let $T$ be the corresponding semistandard Young tableau. Recall that $a_{i,j} - a_{i,j-1}$ is the number of $j$'s in row $i$ of $T$ (with $a_{i,0} := 0$). The tropicalization of $\ov{\gamma}$ sends $\mb{a}$ to the vector $(\mu_1, \ldots, \mu_n)$, where
\[
\mu_j = \sum_{i = 1}^j a_{i,j} - \sum_{i=1}^{j-1} a_{i,j-1}.
\]
Clearly $\mu_j$ is the number of $j$'s in $T$, so $\Trop(\ov{\gamma}) = \tw{\ov{\gamma}}$.

Let $r_i$ denote the $i$th row of $T$. The crystal structure on $T$ can be computed by viewing $T$ as the tensor product $r_n \otimes r_{n-1} \otimes \cdots \otimes r_1$, with each factor having the crystal structure on one-row tableaux described in Example~\ref{ex:one row crystal}. Fix $j \in [n-1]$, and define
\[
A_{k,j} = \sum_{\ell = 2}^k \left(\tw{\ov{\vp}}_j(r_{\ell-1}) - \tw{\ov{\ve}}_j(r_\ell)\right)
\]
for $k \in [n]$. According to~\cite[Prop.~2.1.1]{KN94}, the functions $\tw{\ov{\ve}}_j$ and $\tw{\ov{\vp}}_j$ are given by the piecewise-linear formulas\footnote{We have reversed both the tensor product convention and the labeling of the tensor factors used in~\cite{KN94}, resulting in identical formulas.}
\begin{equation}
\label{eq:PL 1}
\tw{\ov{\ve}}_j(T) = \tw{\ov{\ve}}_j(r_1) - \min_{1 \leq k \leq n}(A_{k,j}), \qquad\qquad \tw{\ov{\vp}}_j(T) = \tw{\ov{\vp}}_j(r_n) + \max_{k \in [n]}\left(\sum_{i=k+1}^n \left(\tw{\ov{\vp}}_j(r_{i-1}) - \tw{\ov{\ve}}_j(r_i)\right) \right),
\end{equation}
and the crystal operators are given by
\begin{equation}
\label{eq:PL 2}
\tw{\ov{e}}_j(T) = r_n \otimes \cdots \otimes \tw{\ov{e}}_j(r_{i^*}) \otimes \cdots \otimes r_1, \qquad\qquad \tw{\ov{f}}_j(T) = r_n \otimes \cdots \otimes \tw{\ov{f}}_j(r_{i^{**}}) \otimes \cdots \otimes r_1,
\end{equation}
where $i^*$ (resp., $i^{**}$) is the smallest (resp., largest) value of $i$ for which $A_{i,j} = \min_{k \in [n]} (A_{k,j})$.

It is straightforward to verify that the tropicalizations of the formulas in Lemma~\ref{lem:explicit crystal formulas} agree with the formulas in \eqref{eq:PL 1} and \eqref{eq:PL 2}. For the reader's convenience, we present the details for the functions $\tw{\ov{\ve}}_j$ and the crystal operators $\tw{\ov{e}}_j$.

The first key observation is that
\begin{equation}
\label{eq:Trop phi}
\Trop(\phi_{i,j}) = (a_{i-1,j} - a_{i-1,j-1}) - (a_{i,j+1} - a_{i,j}) = \tw{\ov{\vp}}_j(r_{i-1}) - \tw{\ov{\ve}}_j(r_i)
\end{equation}
for $i \leq j$. This implies that
\begin{equation}
\label{eq:Trop epsilon}
\Trop(\ov{\ve}_j)(\mb{a}) = (a_{1,j+1} - a_{1,j}) + \max_{1 \leq k \leq j} \left(- \sum_{i = 2}^k \Trop(\phi_{i,j})\right) = \tw{\ov{\ve}}_j(r_1) - \min_{1 \leq k \leq j} (A_{k,j}).
\end{equation}
If $k \geq j+1$, then $A_{k,j} = A_{j,j} + (a_{j,j} - a_{j+1,j+1}) \geq A_{j,j}$, so the last expression in~\eqref{eq:Trop epsilon} is equal to $\tw{\ov{\ve}}_j(T)$.

We will now show that
\[
\tw{\ov{e}}_j(T) = \Trop(\ov{e}_j^c)|_{c = 1}(\mb{a})
\]
if the latter is a Gelfand--Tsetlin pattern, and otherwise $\tw{\ov{e}}_j(T)$ is undefined. By Lemma~\ref{lem:explicit crystal formulas} and~\eqref{eq:Trop phi}, $\Trop(\ov{e}_j^c)|_{c = 1}(\mb{a}) = \mb{a}'$, where $a'_{i,r} = a_{i,r}$ for $r \neq j$, and
\[
a'_{i,j} = a_{i,j} + \min_{1 \leq k \leq j} \left(\mathbbm{1}_{k \geq i} + A_{k,j}\right) - \min_{1 \leq k \leq j} \left(\mathbbm{1}_{k \geq i+1} + A_{k,j}\right).
\]
The second key observation is that the difference of the two minimums appearing on the right-hand side is equal to 1 if $i = i^*$, and 0 otherwise.

Suppose $\tw{\ov{e}}_j(T)$ is defined. By~\eqref{eq:PL 2}, $\tw{\ov{e}}_j(T)$ is obtained by replacing $r_{i^*}$ with $\tw{\ov{e}}_j(r_{i^*})$. In other words, $\tw{\ov{e}}_j$ adds 1 to $a_{i^*,j}$, which is precisely what $\Trop(\ov{e}_j^c)|_{c = 1}(\mb{a})$ does. Since $\tw{\ov{e}}_j(T)$ is known to be a semistandard Young tableau, $\mb{a}'$ must be a Gelfand--Tsetlin pattern in this case (with a bit more work, one can see that $\mb{a}'$ is a Gelfand--Tsetlin pattern using only the piecewise-linear formulas).

Now suppose $\tw{\ov{e}}_j(T)$ is undefined. This means that $\tw{\ov{\ve}}_j(T) = 0$. Since $A_{1,j} = 0$ and $\tw{\ov{\ve}}_j(r_1) \geq 0$, we must have $\tw{\ov{\ve}}_j(r_1) = 0$ and $A_{k,j} \geq 0$ for all $k$. The latter condition implies that $i^* = 1$, and the former condition implies that $a_{1,j+1} = a_{1,j}$. This means that $\mb{a}'$ violates the inequality $a'_{1,j+1} \geq a'_{1,j}$, so it is not a Gelfand--Tsetlin pattern.
\end{proof}

\section{Geometric RSK}
\label{sec:gRSK}

To motivate the definition of geometric RSK, we first give a brief review of the RSK correspondence, and explain how it can be viewed as a bijection from the set of all $m \times n$ nonnegative integer matrices to the set of $m \times n$ nonnegative integer matrices with weakly increasing rows and columns.

Let $\mb{a} = (a_i^j)_{i \in [m], j \in [n]}$ be an $m \times n$ matrix of nonnegative integers, and let $\mb{a}_i = (a_i^1, \ldots, a_i^n)$ denote the $i$th row of the matrix. We interpret $\mb{a}_i$ as the weakly increasing word in the alphabet $[n]$ consisting of $a_i^1$ 1's, $a_i^2$ 2's, etc., as in Example~\ref{ex:one row crystal}. The RSK correspondence sends the matrix $\mb{a}$ to a pair $(P,Q)$ of semistandard tableaux of the same shape, such that $P$ has entries in $[n]$, and $Q$ has entries in $[m]$. Specifically, one takes $P_0$ to be the empty tableau, and then recursively defines $P_i$ to be the tableau formed by row inserting the word $\mb{a}_i$ into $P_{i-1}$. The \defn{insertion tableau} $P$ is defined to be $P_m$, and the \defn{recording tableau} $Q$ is the tableau such that for each $i \in [m]$, the subtableau consisting of entries less than or equal to $i$ has the same shape as $P_i$. (We refer the reader to~\cite{EC2} for more details.)

To go from a pair $(P,Q) \in \SSYT_{\leq n}(\la) \times \SSYT_{\leq m}(\la)$ to an $m \times n$ matrix with weakly increasing rows and columns, one forms the GT patterns of width $\min(m,n)$ and heights $n$ and $m$ for $P$ and $Q$, respectively (see \S\ref{sec:GT patterns}). Then one glues the first GT pattern to the transpose of the second GT pattern along the diagonal specifying their common shape. (In this context, we represent GT patterns as left-justified arrays, with the $i$th row from the bottom containing the entries $a_{i,i}, \ldots, a_{i,n}$.) See Figure~\ref{fig:RSK} for an example.

A fundamental property of RSK is that transposing the matrix $\mb{a}$ corresponds to interchanging the tableaux $P$ and $Q$. In other words, $Q$ is the tableau formed by inserting the words associated to the columns $\mb{a}^j = (a_1^j, a_2^j, \ldots, a_m^j)$, and $P$ records this insertion procedure. If we view RSK as a map of $m \times n$ matrices, this property says that RSK commutes with transposition.

\ytableausetup{centertableaux}

\begin{figure}
\begin{gather*}
\mb{a} = \begin{pmatrix}
1 & 4 \\
2 & 1 \\
1 & 0
\end{pmatrix}
\;
\xrightarrow{\displaystyle \RSK}
\;
(P,Q) = \left( \, \ytableaushort{111122,222} \; , \; \ytableaushort{111112,223} \, \right)
\\
(P,Q) \longleftrightarrow \left(\textcolor{blue}{\begin{matrix}
3 & \\
4 & 6
\end{matrix}} \quad
, \quad
\textcolor{red}{\begin{matrix}
2 & 3 & \\
5 & 6 & 6
\end{matrix}}\right)
\xrightarrow{\displaystyle \text{glue}}
\begin{pmatrix}
\textcolor{red}{2} & \textcolor{red}{5} \\
\textcolor{purple}{3} & \textcolor{red}{6} \\
\textcolor{blue}{4} & \textcolor{purple}{6}
\end{pmatrix}
\end{gather*}
\caption{The first line shows the pair $(P,Q)$ of semistandard tableaux associated to a matrix $\mb{a}$ by the RSK correspondence. The second line shows the Gelfand--Tsetlin patterns of width 2 corresponding to $P$ and $Q$, and the matrix with weakly increasing rows and columns obtained by gluing these patterns together along their common shape $(6,3)$.}
\label{fig:RSK}
\end{figure}

The \defn{geometric RSK correspondence (gRSK)} is a birational isomorphism
\[
\gRSK \colon \Mat_{m \times n}(\CC^*) \rightarrow \Mat_{m \times n}(\CC^*).
\]
By analogy with classical RSK, we identify the output matrix with a pair of patterns $(P,Q) \in \GT_n^{\leq m} * \GT_m^{\leq n}$, where
\[
\GT_n^{\leq m} * \GT_m^{\leq n} = \{(P,Q) \in \GT_n^{\leq m} \times \GT_m^{\leq n} \mid \sh(P) = \sh(Q)\}
\]
(recall the definitions~\eqref{eq:GT_defn} and~\eqref{eq:GT_shape}). As in the combinatorial setting, our convention is that $P$ sits in the bottom-left corner of the matrix, and the transpose of $Q$ sits in the top-right corner.

We will give two equivalent definitions of gRSK. The first definition is modeled on the row insertion definition of RSK, and is due to Noumi and Yamada~\cite{NoumiYamada}; the second is similar in spirit to the growth diagram formulation of RSK, and is due to O'Connell, Sepp\"al\"ainen, and Zygouras~\cite{OSZ}.

\subsection{Row insertion formulation of gRSK}
\label{sec:gRSK row}

Suppose $\xx = (x_a^b)_{a \in [m], b \in [n]} \in \Mat_{m \times n}(\CC^*)$. Let $\xx_a = (x_a^1, \ldots, x_a^n)$ be the $a$th row of $\xx$, and let $p = \min(m,n)$. Recall from \S \ref{sec:unipotent} the map $M \colon (X_n)^k \rightarrow \GL_n$ given by
\[
M(\xx_1, \ldots, \xx_k) = W(\xx_1) \cdots W(\xx_k),
\]
which makes $(X_n)^k$ into a $\GL_n$-unipotent crystal. It is straightforward to show by induction on $k$ that $M$ is a surjection from $(X_n)^k$ onto $(B^-)^{\leq k}$.

\begin{dfn}
\label{defn:gRSK}
Define $\gRSK(\xx) = (P,Q)$, where $P = (z_{i,j}) \in \GT_n^{\leq m}$ and $Q = (z'_{j',i'}) \in \GT_m^{\leq n}$ are given by
\[
z_{i,j} = \dfrac{\Delta_{[i,j]}\bigl( M(\xx_1, \ldots, \xx_m) \bigr)}{\Delta_{[i+1,j]}\bigl( M(\xx_1, \ldots, \xx_m) \bigr)},
\qquad\qquad z'_{j',i'} = \dfrac{\Delta_{[j',n]}\bigl( M(\xx_1, \ldots, \xx_{i'}) \bigr)}{\Delta_{[j'+1,n]}\bigl( M(\xx_1, \ldots, \xx_{i'}) \bigr)},
\]
for $1 \leq i \leq m, \, i \leq j \leq n$ and $1 \leq j' \leq n, \, j' \leq i' \leq m$. We identify $(P,Q)$ with a matrix in $\Mat_{m \times n}(\CC^*)$ by gluing $P$ and $Q$ along their common shape, as described above. We refer to $P$ and $Q$ as the \defn{$P$-pattern} and \defn{$Q$-pattern} of the matrix $\xx$.
\end{dfn}

It is clear that the shape $(z_{1,n}, \ldots, z_{p,n})$ of $P$ is equal to the shape $(z'_{1,m}, \ldots, z'_{p,m})$ of $Q$, as required. Moreover, if we define
\[
P_k = \Psi_n^{\leq k}\bigl(M(\xx_1, \ldots, \xx_k)\bigr) \in \GT_n^{\leq k},
\]
then it follows from Lemma~\ref{lem:Phi inverse} that $P = P_m$, and the $k$th diagonal $(z'_{1,k}, z'_{2,k}, \ldots, z'_{\min(k,n),k})$ of $Q$ is equal to the shape of the pattern $P_k$. We interpret this as saying that $P$ is the result of ``geometrically inserting'' the rows of $\xx$ (starting with the top row $\xx_1$), and $Q$ ``records the growth'' of $P$.

\begin{ex}
\label{ex:gRSK insertion}
We work out the geometric RSK correspondence in the case $m = 3, n=2$. We start with the matrix
$
\xx = \begin{pmatrix}
x_1^1 & x_1^2 \medskip \\
x_2^1 & x_2^2 \medskip \\
x_3^1 & x_3^2
\end{pmatrix}
$,
and compute the matrices $M_1 = M(\xx_1), M_2 = M(\xx_1, \xx_2), M_3 = M(\xx_1, \xx_2, \xx_3)$:
\[
M_1 = 
\begin{pmatrix}
x_1^1 & 0 \medskip \\
1 & x_1^2
\end{pmatrix},
\qquad
M_2 = \begin{pmatrix}
x_1^1x_2^1 & 0 \medskip \\
x_1^2 + x_2^1 & x_1^2x_2^2
\end{pmatrix},
\qquad
M_3 =
\begin{pmatrix}
x_1^1x_2^1x_3^1 & 0 \medskip \\
x_1^2 x_2^2 + x_1^2 x_3^1 + x_2^1x_3^1 & x_1^2x_2^2x_3^2
\end{pmatrix}.
\]
Using these matrices, we compute $P \in \GT_2^{\leq 3} = \GT_2$ and $Q \in \GT_3^{\leq 2}$:
\[
P = \begin{matrix}
z_{22} & \medskip\medskip \\
z_{11} & z_{12} \medskip
\end{matrix}
\; = \;
\begin{matrix}
\Delta_{2}(M_3) & \medskip \\
\Delta_{1}(M_3) & \dfrac{\Delta_{12}(M_3)}{\Delta_{2}(M_3)}
\end{matrix}
\; = \;
\begin{matrix}
x_1^2 x_2^2 + x_1^2 x_3^1 + x_2^1x_3^1 & \medskip \\
x_1^1 x_2^1 x_3^1 & \dfrac{x_1^1 x_2^1 x_3^1 x_1^2 x_2^2 x_3^2}{x_1^2 x_2^2 + x_1^2 x_3^1 + x_2^1x_3^1}
\end{matrix}
\]
\bigskip
\medskip
\[
Q = \begin{matrix}
z'_{22} & z'_{23} & \medskip\medskip \\
z'_{11} & z'_{12} & z'_{13} \medskip
\end{matrix}
\; = \;
\begin{matrix}
\Delta_{2}(M_2) & \Delta_2(M_3) & \medskip \\
\dfrac{\Delta_{12}(M_1)}{\Delta_2(M_1)} & \dfrac{\Delta_{12}(M_2)}{\Delta_{2}(M_2)} & \dfrac{\Delta_{12}(M_3)}{\Delta_{2}(M_3)}
\end{matrix}
\; = \;
\begin{matrix}
x_1^2 + x_2^1 & x_1^2 x_2^2 + x_1^2 x_3^1 + x_2^1x_3^1 & \medskip \\
x_1^1x_1^2 & \dfrac{x_1^1 x_2^1 x_1^2 x_2^2}{x_1^2 + x_2^1} & \dfrac{x_1^1 x_2^1 x_3^1 x_1^2 x_2^2 x_3^2}{x_1^2 x_2^2 + x_1^2 x_3^1 + x_2^1x_3^1}
\end{matrix}.
\]
\bigskip
Thus, we have
\[
\begin{pmatrix}
x_1^1 & x_1^2 \medskip \\
x_2^1 & x_2^2 \medskip \\
x_3^1 & x_3^2
\end{pmatrix}
\xrightarrow{\displaystyle \gRSK}
\begin{pmatrix}
x_1^2 + x_2^1 & x_1^1x_1^2 \bigskip \\
x_1^2 x_2^2 + x_1^2 x_3^1 + x_2^1x_3^1 & \dfrac{x_1^1 x_2^1 x_1^2 x_2^2}{x_1^2 + x_2^1} \bigskip \\
x_1^1 x_2^1 x_3^1 & \dfrac{x_1^1 x_2^1 x_3^1 x_1^2 x_2^2 x_3^2}{x_1^2 x_2^2 + x_1^2 x_3^1 + x_2^1x_3^1}
\end{pmatrix}.
\]
\end{ex}

\subsection{Local move formulation of gRSK}
\label{sec:gRSK local}

As above, denote an element of $\Mat_{m \times n}(\CC^*)$ by $\xx = (x_a^b)$, and set $p = \min(m,n)$. For each $a \in [m]$ and $b \in [n]$, we introduce two rational maps
\[
\eta_a^b, T_a^b \colon \Mat_{m \times n}(\CC^*) \rightarrow \Mat_{m \times n}(\CC^*),
\]
each of which changes only the entry $x_a^b$. These maps are defined by
\begin{equation}
\label{eq:eta T}
\eta_a^b \colon x_a^b \mapsto x_a^b \gMax_a^b(\xx),
\qquad\qquad
T_a^b \colon x_a^b \mapsto \dfrac{1}{x_a^b} \gMax_a^b(\xx) \gMin_a^b(\xx),
\end{equation}
where
\[
\gMax_a^b(\xx) = \begin{cases}
\dfrac{x_a^{b-1}x_{a-1}^b}{x_a^{b-1} + x_{a-1}^b} & \text{if } a,b > 1, \\
x_1^{b-1} & \text{if } a = 1, b > 1, \\
x_{a-1}^1 & \text{if } a > 1, b = 1, \\
1 & \text{if } a = b = 1,
\end{cases}
\qquad\qquad
\gMin_a^b(\xx) = \begin{cases}
x_{a+1}^b + x_a^{b+1} & \text{if } a < m, b < n, \\
x_m^{b+1} & \text{if } a = m, b < n, \\
x_{a+1}^n & \text{if } a < m, b = n, \\
1 & \text{if } a = m, b = n.
\end{cases}
\]
It is easy to see that $T_a^b$ is an involution, and $\eta_a^b$ has inverse given by $x_a^b \mapsto x_a^b \frac{1}{\gMax_a^b(\xx)}$.

\begin{remark}
The function $\gMax_a^b$ tropicalizes to $\max(x_a^{b-1}, x_{a-1}^b)$, and the function $\gMin_a^b$ tropicalizes to $\min(x_{a+1}^b, x_a^{b+1})$. (The various cases correspond to setting $x_0^b, x_a^0 = 0$ and $x_{m+1}^b = x_a^{n+1} = \infty$ in the tropicalized formulas. This rule does not work for $\gMin_m^n$, but we will never use the operator $T_m^n$ so this is not a problem.) Thus, if $(a,b) \neq (m,n)$, the map $T_a^b$ tropicalizes to the \defn{piecewise-linear toggle} at node $(a,b)$ in the coordinate-wise partial order on $[m] \times [n]$~\cite{EP21}. These toggles were first studied by Kirillov and Berenstein~\cite{KB95}.
\end{remark}

For $a \in [m]$ and $b \in [n]$, define $\tau_a^b$ to be the composition of the maps $T_i^j$ at each position of the diagonal strictly to the northwest of $(a,b)$, that is,
\[
\tau_a^b = \begin{cases}
T_1^{b-a+1} \circ \cdots \circ T_{a-2}^{b-2} \circ T_{a-1}^{b-1} & \text { if } a \leq b, \\
T_{a-b+1}^1 \circ \cdots \circ T_{a-2}^{b-2} \circ T_{a-1}^{b-1} & \text{ if } a \geq b.
\end{cases}
\]
Define
\[
\rho_a^b = \tau_a^b \circ \eta_a^b.
\]
Note that $\tau_a^b$ is the identity map if $a = 1$ or $b = 1$, and $\rho_1^1 = \eta_1^1$ is also the identity map. Finally, define
\[
\rho = (\rho_m^n \circ \cdots \circ \rho_m^1) \circ \cdots \circ (\rho_1^n \circ \cdots \circ \rho_1^1),
\]
and define $\eta$ and $\tau$ in the same way, but with $\rho_a^b$ replaced by $\eta_a^b$ and $\tau_a^b$, respectively.

\begin{lemma}
\label{lem:toggles commute}
\
\begin{enumerate}
\item The map $\rho$ (resp., $\eta$, $\tau$) can be computed by composing the maps $\rho_a^b$ (resp., $\eta_a^b$, $\tau_a^b$) along any linear extension of the coordinate-wise partial order on $[m] \times [n]$.
\item The map $\rho$ factors as $\rho = \tau \circ \eta$.
\end{enumerate}
\end{lemma}

\begin{proof}
It is clear that when $(a,b)$ and $(a',b')$ are not adjacent, $\eta_a^b$ commutes with both $\eta_{a'}^{b'}$ and $T_{a'}^{b'}$, and $T_a^b$ commutes with $T_{a'}^{b'}$. This implies that $\rho_a^b$ and $\rho_{a'}^{b'}$ commute (as do $\tau_a^b, \tau_{a'}^{b'}$ and $\eta_a^b, \eta_{a'}^{b'}$) when $(a,b)$ and $(a',b')$ are incomparable in $[m] \times [n]$, \textit{i.e.}, when one position lies strictly to the southeast of the other. This proves part (1). Part (2) follows from the observation that $\eta_a^b$ commutes with $\tau_{a'}^{b'}$ when $(a',b') \not \geq (a,b)$.
\end{proof}

\begin{thm}[O'Connell--Sepp\"al\"ainen--Zygouras~\cite{OSZ}]
\label{thm:local}
The map $\rho$ agrees with the geometric RSK correspondence of Definition~\ref{defn:gRSK}.
\end{thm}

As mentioned in~\S \ref{sec:intro outline}, our geometric RSK correspondence is different from the one studied in~\cite{NoumiYamada, OSZ}. In \S\ref{sec:max min}, we discuss the precise connection between the two versions of gRSK, and explain why Theorem~\ref{thm:local} follows from the analogous result proved in~\cite{OSZ}.

\begin{ex}
\label{ex:gRSK local}
We illustrate Theorem~\ref{thm:local} in the case $m = 3, n = 2$. By Lemma~\ref{lem:toggles commute}(2), we have $\rho = \tau \circ \eta$, where
\[
\eta =  \eta_3^2 \circ \eta_3^1 \circ \eta_2^2 \circ \eta_2^1 \circ \eta_1^2 \circ \eta_1^1, \qquad \tau = \tau_3^2 \circ \tau_2^2 = T_2^1 \circ T_1^1.
\]
We compute
\[
\begin{pmatrix}
x_1^1 & x_1^2 \medskip \\
x_2^1 & x_2^2 \medskip \\
x_3^1 & x_3^2
\end{pmatrix}
\overset{\begin{matrix} \eta \end{matrix}} {\longrightarrow}
\begin{pmatrix}
x_1^1 & x_1^1x_1^2 \bigskip \\
x_1^1x_2^1 & \dfrac{x_1^1 x_2^1 x_1^2 x_2^2}{x_1^2 + x_2^1} \bigskip \\
x_1^1 x_2^1 x_3^1 & \dfrac{x_1^1 x_2^1 x_3^1 x_1^2 x_2^2 x_3^2}{x_1^2 x_2^2 + x_1^2 x_3^1 + x_2^1x_3^1}
\end{pmatrix}
\overset{\begin{matrix} \tau \end{matrix}} {\longrightarrow}
\begin{pmatrix}
x_1^2 + x_2^1 & x_1^1x_1^2 \bigskip \\
x_1^2 x_2^2 + x_1^2 x_3^1 + x_2^1x_3^1 & \dfrac{x_1^1 x_2^1 x_1^2 x_2^2}{x_1^2 + x_2^1} \bigskip \\
x_1^1 x_2^1 x_3^1 & \dfrac{x_1^1 x_2^1 x_3^1 x_1^2 x_2^2 x_3^2}{x_1^2 x_2^2 + x_1^2 x_3^1 + x_2^1x_3^1}
\end{pmatrix},
\]

\bigskip
\noindent obtaining the same result as in Example~\ref{ex:gRSK insertion}.
\end{ex}

\begin{cor}
\label{cor:gRSK props}
\
\begin{enumerate}
\item gRSK is a birational isomorphism.
\item gRSK and its inverse are positive. This implies that gRSK restricts to a homeomorphism from $\Mat_{m \times n}(\RR_{> 0})$ to itself.
\item gRSK satisfies the symmetry property
\[
\gRSK(\xx) = (P,Q) \iff \gRSK(\xx^t) = (Q,P),
\]
where $\xx^t$ is the transpose of $\xx$. Thus, the patterns $P = (z_{i,j})$ and $Q = (z'_{j',i'})$ are given by the alternative formulas
\begin{equation}
\label{eq:gRSK transposed}
z_{i,j} = \dfrac{\Delta_{[i,m]}(M(\xx^1, \ldots, \xx^j))}{\Delta_{[i+1,m]}(M(\xx^1, \ldots, \xx^j))},
\qquad\qquad z'_{j',i'} = \dfrac{\Delta_{[j',i']}(M(\xx^1, \ldots, \xx^n))}{\Delta_{[j'+1,i']}(M(\xx^1, \ldots, \xx^n))},
\end{equation}
for $1 \leq i \leq m, \, i \leq j \leq n$ and $1 \leq j' \leq n, \, j' \leq i' \leq m$, where $\xx^j = (x_1^j, \ldots, x_m^j)$ is the $j$th column of~$\xx$.
\end{enumerate}
\end{cor}

\begin{proof}
Part (1) follows from the fact that the $\eta_a^b$ and $T_a^b$ are invertible rational maps, and part (2) follows from the fact that these maps and their inverses are positive. Part (3) follows from Lemma~\ref{lem:toggles commute}(1), which implies that the map $\rho$ can be computed column-by-column, rather than row-by-row.
\end{proof}

\begin{ex}
\label{ex:gRSK transposed}
The reader may verify that in the case $m=3, n=2$, the $P$- and $Q$-patterns computed in Example~\ref{ex:gRSK insertion} can be obtained from the matrices
\[
M(\xx^1) =
\begin{pmatrix}
x_1^1 & 0 & 0 \medskip \\
1 & x_2^1 & 0 \medskip \\
0 & 1 & x_3^1
\end{pmatrix},
\qquad
M(\xx^1, \xx^2) = \begin{pmatrix}
x_1^1x_1^2 & 0 & 0 \medskip \\
x_2^1 + x_1^2 & x_2^1x_2^2 & 0 \medskip \\
1 & x_3^1 + x_2^2 & x_3^1x_3^2
\end{pmatrix}
\]
as the ratios of determinants specified by~\eqref{eq:gRSK transposed}. In particular, the common shape $(z_{12}, z_{22}) = (z'_{13}, z'_{23})$ of $P$ and $Q$ is given by
\[
\left(\dfrac{\Delta_{123}\bigl( M(\xx^1, \xx^2) \bigr)}{\Delta_{23}\bigl( M(\xx^1, \xx^2) \bigr)}, \dfrac{\Delta_{23}\bigl( M(\xx^1, \xx^2) \bigr)}{\Delta_3\bigl( M(\xx^1, \xx^2) \bigr)}\right) = \left(\dfrac{x_1^1 x_2^1 x_3^1 x_1^2 x_2^2 x_3^2}{x_1^2 x_2^2 + x_1^2 x_3^1 + x_2^1x_3^1}, x_1^2 x_2^2 + x_1^2 x_3^1 + x_2^1x_3^1\right).
\]
\end{ex}

\begin{remark}
Results of Noumi and Yamada~\cite{NoumiYamada} imply that the geometric row insertion formulation of gRSK tropicalizes to RSK (see \S \ref{sec:max min}). There are also multiple proofs in the literature of the fact that the tropicalization of $\rho$ is a piecewise-linear description of RSK. This was known to Pak~\cite{Pak01}, who used a generalization of $\Trop(\rho)$ to give a simple proof of the hook-length formula for the number of standard Young tableaux of a given shape. An elementary proof that $\Trop(\rho) = \RSK$ can be found in~\cite{Hopkins14}. This result was recently reproved (and substantially generalized) in the context of quiver representations~\cite{GPT18}.
\end{remark}

\subsection{Geometric RSK is an isomorphism of geometric crystals}
\label{sec:gRSK isom}

We have described two $\GL_n$- (resp., $\GL_m$-) geometric crystal structures on the variety $X = \Mat_{m \times n}(\CC^*)$. The first is the \defn{basic geometric crystal structure}, which comes from identifying $X$ with $(X_n)^m$ (resp., $(X_m)^n$); we denote this geometric crystal by $X^{\Mat}$.
The second is the \defn{Gelfand--Tsetlin geometric crystal structure}, which comes from identifying $X$ with $\GT_n^{\leq m} * \GT_m^{\leq n}$, and defining $\ov{e}_j^c(P,Q) = (\ov{e}_j^c(P), Q)$ (resp., $e_i^c(P,Q) = (P,e_i^c(Q))$).
We denote this geometric crystal by $X^{\GT}$.

\begin{thm}
\label{thm:gRSK isom}
The geometric RSK correspondence
\[
\gRSK \colon X^{\Mat} \to X^{\GT}
\]
is an isomorphism of $\GL_n \times \GL_m$-geometric crystals.\footnote{A $\GL_n \times \GL_m$-geometric crystal is a variety equipped with $\GL_n$- and $\GL_m$-geometric crystal structures which commute with each other (\textit{i.e.}, $\ve_i \ov{e}^c_j = \ve_i$, $\ov{\ve}_j e_i^c = \ov{\ve}_j$, $e_i^{c_1} \ov{e}_j^{c_2} = \ov{e}_j^{c_2} e_i^{c_1}$, etc.)}
\end{thm}

\begin{proof}
Since geometric RSK is a birational isomorphism of the underlying varieties and the $\GL_n$- and $\GL_m$-geometric crystal structures on $X^{\GT}$ clearly commute, it suffices to prove that $\gRSK$ is an isomorphism of $\GL_m$-geometric crystals and $\GL_n$-geometric crystals.
By Corollary~\ref{cor:gRSK props}(3), it is enough to prove that $\gRSK$ is an isomorphism of $\GL_n$-geometric crystals. To prove this, we must show that if $\gRSK(\xx) = (P,Q)$, then
\begin{equation}
\label{eq:gRSK isom}
\ov{\gamma}(\xx) = \ov{\gamma}(P), \qquad \ov{\ve}_j(\xx) = \ov{\ve}_j(P), \qquad \ov{\vp}_j(\xx) = \ov{\vp}_j(P), \qquad \gRSK(\ov{e}_j^c(\xx)) = (\ov{e}_j^c(P), Q).
\end{equation}

By~\eqref{eq:basic functions M}, $\ov{\gamma}(\xx)$ is the diagonal of the $n \times n$ matrix $M(\xx_1, \ldots, \xx_m)$, and $\ov{\ve}_j(\xx),\ov{\vp}_j(\xx)$ are ratios of certain entries of this matrix. By Definition~\ref{defn:GT geom cryst}, $\ov{\gamma}(P)$, $\ov{\ve}_j(P)$, and $\ov{\vp}_j(P)$ are obtained in the exactly the same way from the matrix $\Phi_n^{\leq m}(P)$. The geometric RSK correspondence is defined in such a way that $P = \Psi_n^{\leq m}(M(\xx_1, \ldots, \xx_m))$, so $\Phi_n^{\leq m}(P) = M(\xx_1, \ldots, \xx_m)$ by Lemma~\ref{lem:Phi inverse}. This proves the first three identities in~\eqref{eq:gRSK isom}.

Suppose $\gRSK(\ov{e}_j^c(\xx)) = (P',Q')$. By~\eqref{eq:basic crystal operators M} and Definitions~\ref{defn:GT geom cryst} and~\ref{defn:gRSK}, we have
\begin{align*}
P' &= \Psi_n^{\leq m}\left(x_j\left((c-1)\ov{\vp}_j(\xx)\right) \cdot M(\xx_1, \ldots, \xx_m) \cdot x_j\left((c^{-1} - 1)\ov{\ve}_j(\xx)\right)\right) \\
&= \Psi_n^{\leq m}\left(x_j\left((c-1)\ov{\vp}_j(P)\right) \cdot \Phi_n^{\leq m}(P) \cdot x_j\left((c^{-1} - 1)\ov{\ve}_j(P)\right)\right) \\
&= \ov{e}_j^c(P).
\end{align*}
It remains to show that $Q' = Q$. Suppose $\ov{e}_j^c(\xx_1, \ldots, \xx_m) = (\xx_1', \ldots, \xx_m')$. The formula~\eqref{eq:basic crystal operators M} implies that for each $k \in [m]$, there are rational functions $a,a'$ in the entries of $\xx$ and $c$ such that
\[
M(\xx'_1, \ldots, \xx'_k) = x_j(a)M(\xx_1, \ldots \xx_k) x_j(a').
\]
As observed in the proof of Lemma~\ref{lem:shape preserved}, this means that $M(\xx_1, \ldots \xx_k)$ and $M(\xx'_1, \ldots, \xx'_k)$ have the same bottom-left justified minors. By definition, the entries of $Q$ and $Q'$ are ratios of bottom-left justified minors of these matrices, so we are done.
\end{proof}

Theorem~\ref{thm:gRSK isom} gives a new proof of the following result from~\cite{LP13II}.

\begin{cor}[{\cite{LP13II}}]
The basic $\GL_n$- and $\GL_m$-geometric crystal structures on $\Mat_{m \times n}(\CC^*)$ commute.
\end{cor}

\subsection{Connection to Noumi and Yamada's geometric RSK}
\label{sec:max min}

Let $f \colon (\CC^*)^d \rightarrow (\CC^*)^k$ be a positive rational map. Define a piecewise-linear map $\Trop^{\dagger}(f)$ in the same way as $\Trop(f)$ was defined in \S \ref{sec:trop}, but using $\max$ instead of $\min$. By definition, we have $\Trop(x + y) = \min(x,y)$ and $\Trop^{\dagger}(x+y) = \max(x,y)$. In addition, the identity $\max(x,y) = x + y - \min(x,y)$ implies that
\[
\max(x,y) = \Trop\left(\dfrac{xy}{x+y}\right), \qquad \min(x,y) = \Trop^{\dagger}\left(\dfrac{xy}{x+y}\right).
\]

Following~\cite[\S 1.3]{NoumiYamada}, define an involution $f \mapsto f^{\dagger}$ of positive rational functions by
\[
f^{\dagger}(z_1, \ldots, z_n) = \frac{1}{f\left(\frac{1}{z_1}, \ldots, \frac{1}{z_n}\right)}.
\]
This involution interchanges the two tropicalizations; that is, $\Trop(f^{\dagger}) = \Trop^{\dagger}(f)$. This follows from the simple observation that
\[
\left(\dfrac{1}{x} + \dfrac{1}{y}\right)^{-1} = \dfrac{xy}{x+y}.
\]

The version of geometric RSK studied in~\cite{Kir01,NoumiYamada,OSZ} tropicalizes to RSK using $\Trop^{\dagger}$. We will now show that our definition of gRSK is obtained from Noumi and Yamada's by the transformation $f \mapsto f^{\dagger}$. Since the $Q$-pattern in both definitions consists of the shapes of the $P$-pattern after inserting the first several rows, it suffices to consider the $P$-patterns. Let $y_{i,j}(\xx)$ be the entries of the $P$-pattern associated to $\xx$ by the Noumi--Yamada version of gRSK. Noumi and Yamada define the $P$-pattern by a recursive ``geometric row insertion'' procedure that tropicalizes (using $\Trop^{\dagger}$) to combinatorial row insertion, and then they prove~\cite[Thm.~2.4]{NoumiYamada} that the $y_{i,j}$ are given by the formula\footnote{Thm.~2.4 in \cite{NoumiYamada} is stated in terms of coordinates $p_i^j$, which are related to our $y_{i,j}$ coordinates by $y_{i,j} = p^i_i p^i_{i+1} \cdots p^i_j$.}
\[
y_{i,j}(\xx) = \dfrac{\Delta_{[1,i], [j-i+1,j]}(H(\xx_1) \cdots H(\xx_m))}{\Delta_{[1,i-1], [j-i+2,j]}(H(\xx_1) \cdots H(\xx_m))},
\]
where $H(a^1, \ldots, a^n)$ is the upper-triangular $n \times n$ matrix defined by
\[
H(a^1, \ldots, a^n) = \sum_{i \leq j} a^i a^{i+1} \cdots a^j E_{ij}.
\]
The following result shows that our definition of the $P$-pattern is related to Noumi and Yamada's definition by the transformation $f \mapsto f^{\dagger}$, and therefore our definition of geometric RSK tropicalizes to combinatorial RSK using $\Trop$.

\begin{lemma}
\label{lem:H M identity}
Let $\xx_i^{-1} = (\frac{1}{x_i^1}, \ldots, \frac{1}{x_i^n})$. For $1 \leq i \leq m$ and $i \leq j \leq n$, we have
\begin{equation}
\label{eq:H M identity}
\dfrac{\Delta_{[1,i-1], [j-i+2,j]}(H(\xx_1^{-1}) \cdots H(\xx_n^{-1}))}{\Delta_{[1,i], [j-i+1,j]}(H(\xx_1^{-1}) \cdots H(\xx_n^{-1}))} = \dfrac{\Delta_{[i,j],[1,j-i+1]}(W(\xx_1) \cdots W(\xx_m))}{\Delta_{[i+1,j],[1,j-i]}(W(\xx_1) \cdots W(\xx_m))}.
\end{equation}
\end{lemma}

\begin{proof}
For an invertible $n \times n$ matrix $A$, define $A^{\dagger}$ to be the matrix obtained from the transposed inverse $(A^{-1})^t$ by multiplying the $i$th row and column by $(-1)^{i-1}$. It is well-known that the minors of $A$ and $A^{\dagger}$ are related by
\begin{equation}
\label{eq:Jacobi}
\Delta_{I,J}(A^{\dagger}) = \dfrac{\Delta_{[n] \setminus I, [n] \setminus J}(A)}{\det(A)}.
\end{equation}

Set $H = H(\xx_1^{-1}) \cdots H(\xx_m^{-1})$ and $M = W(\xx_1) \cdots W(\xx_m)$. It is easy to see that $H(\xx_i^{-1}) = W(\xx_i)^{\dagger}$, so $H = M^{\dagger}$, and~\eqref{eq:H M identity} asserts that
\[
\dfrac{\Delta_{[1,i-1], [j-i+2,j]}(M^{\dagger})}{\Delta_{[1,i], [j-i+1,j]}(M^{\dagger})} = \dfrac{\Delta_{[i,j],[1,j-i+1]}(M)}{\Delta_{[i+1,j],[1,j-i]}(M)}.
\]
Using~\eqref{eq:Jacobi}, this is equivalent to the identity
\begin{equation}
\label{eq:H M to prove}
\dfrac{\Delta_{[i,n], [1,j-i+1] \cup [j+1,n]}(M)}{\Delta_{[i+1,n], [1,j-i] \cup [j+1,n]}(M)} = \dfrac{\Delta_{[i,j],[1,j-i+1]}(M)}{\Delta_{[i+1,j],[1,j-i]}(M)}.
\end{equation}
Since $M$ is lower-triangular, the submatrices $M_{[i,j],[j+1,n]}$ and $M_{[i+1,j],[j+1,n]}$ consist entirely of zeroes, so we have
\[
\dfrac{\Delta_{[i,n], [1,j-i+1] \cup [j+1,n]}(M)}{\Delta_{[i+1,n], [1,j-i] \cup [j+1,n]}(M)} = \dfrac{\Delta_{[i,j],[1,j-i+1]}(M) \Delta_{[j+1,n],[j+1,n]}(M)}{\Delta_{[i+1,j],[1,j-i]}(M) \Delta_{[j+1,n],[j+1,n]}(M)},
\]
proving~\eqref{eq:H M to prove}.
\end{proof}

Next, we consider the local move formulation. In~\cite{OSZ}, geometric RSK is written as a composition
\[
\rho' = ({\rho'}_n^m \circ \cdots \circ {\rho'}_1^m) \circ \cdots \circ ({\rho'}_n^1 \circ \cdots \circ {\rho'}_1^1)
\]
(the roles of $m$ and $n$ are reversed in that paper). It is clear from the discussion following~\cite[Eq.~(3.4)]{OSZ} that ${\rho'}_b^a$ differs from the map $\rho_a^b$ defined in \S\ref{sec:gRSK local} by interchanging the functions $\gMax$ and $\gMin$ in~\eqref{eq:eta T}. The $\dagger$-involution interchanges $\gMax$ and $\gMin$, so we have ${\rho'}_b^a = (\rho_a^b)^{\dagger}$; since the $\dagger$-involution commutes with composition of rational functions, we have $\rho = ({\rho'})^{\dagger}$. It is shown in~\cite[\S 4]{OSZ} that the map $\rho'$ agrees with Noumi and Yamada's definition of geometric RSK in terms of geometric row insertion; together with the preceding discussion, this proves Theorem~\ref{thm:local}.

\subsection{Decorations and gRSK}
\label{sec:dec again}

Recall that for $\xx \in \Mat_{m \times n}(\CC^*)$, the decoration $F(\xx)$ is defined to be the sum of the $x_i^j$ (whether we view $\xx$ as a point of $(X_n)^m$ or of $(X_m)^n$). We now show that $F(\xx)$ is (almost) the sum of the decorations of the $P$- and $Q$-patterns associated to $\xx$ by $\gRSK$. This result is essentially a special case of~\cite[Thm.~3.2]{OSZ}, but since they use a different definition of $\gRSK$, we provide a proof below.

\begin{thm}[\cite{OSZ}]
\label{thm:dec}
Suppose $\xx \in \Mat_{m \times n}(\CC^*)$. If $\gRSK(\xx) = (P,Q)$, then
\[
F(\xx) = F(P) + F(Q) + \delta_{m,n} z_{n,n},
\]
where $\delta_{m,n}$ is the Kronecker delta, and $z_{n,n}$ is the last part of the shape of $P$ and $Q$ (when $m = n$).
\end{thm}

\begin{proof}
Let $\yy = (y_i^j)$ be the $m \times n$ matrix formed by gluing $P$ and $Q$ along their common shape. By our convention, $P$ occupies the lower left corner of $\yy$ and the transpose of $Q$ occupies the upper right corner, with the common shape $(z_1, \ldots, z_p)$ given by the diagonal $(y_m^n, y_{m-1}^{n-1}, \ldots, y_{m-p+1}^{n-p+1})$. Define
\[
G(\xx) = x_1^1 + \sum_{1 \leq i \leq m, 2 \leq j \leq n} \dfrac{x_i^j}{x_i^{j-1}} + \sum_{2 \leq i \leq m, 1 \leq j \leq n} \dfrac{x_i^j}{x_{i-1}^j},
\]
so that
\[
G(\yy) = F(P) + F(Q) + \delta_{m,n} z_{n,n}.
\]
By Theorem~\ref{thm:local} and Lemma~\ref{lem:toggles commute}, we have $\yy = (\tau \circ \eta)(\xx)$, so it suffices to show that $G((\tau \circ \eta)(\xx)) = F(\xx)$. In fact, we will show that for any $m \times n$ matrix $\xx$, 
\begin{subequations}
\label{eq:dec thm to prove}
\begin{align}
G(\eta(\xx)) & = F(\xx),  \label{eq:dec_first}
\\
G(\tau(\xx)) & = G(\xx).  \label{eq:dec_second}
\end{align}
\end{subequations}

We first prove~\eqref{eq:dec_second}. The map $\tau$ is a composition of the local maps $T_a^b$ with $a < m$ and $b < n$; we will show that each of these maps is $f$-invariant. By definition, $T_a^b$ replaces $x_a^b$ with
\[
\wh{x}_a^b = \dfrac{1}{x_a^b} \gMax_a^b(\xx) \gMin_a^b(\xx).
\]
If $(a,b) \neq (m,n)$, then the terms of $G(\xx)$ which depend on $x_a^b$ can be written as
\[
\dfrac{x_a^b}{\gMax_a^b(\xx)} + \dfrac{\gMin_a^b(\xx)}{x_a^b},
\]
and it is clear that replacing $x_a^b$ with $\wh{x}_a^b$ does not change this sum. (Note that this argument would fail for $a = b = 1$ if $G(\xx)$ did not contain the term $x_1^1$.)

Now we prove~\eqref{eq:dec_first}. For an order ideal $S$ of the coordinate-wise partial order on $[m] \times [n]$, define
\[
G_S(\xx) = \sum_{(a,b) \in S} \dfrac{x_a^b}{\gMax_a^b(\xx)} + \sum_{(a,b) \not \in S} x_a^b.
\]
It is easy to see that $G_\emptyset(\xx) = F(\xx)$ and $G_{[m] \times [n]}(\xx) = G(\xx)$. Suppose $S'$ is an order ideal obtained by adding a single element $(a,b)$ to $S$. The map $\eta_a^b$ replaces $x_a^b$ with $x_a^b \gMax_a^b(\xx)$. The expressions $\gMax_i^j(\xx)$ for $(i,j) \in S$ do not depend on $x_a^b$, so we have
\[
G_{S'}(\eta_a^b(\xx)) = G_S(\xx).
\]
Since $\eta$ is the composition of the maps $\eta_a^b$ along a linear extension of the lattice $[m] \times [n]$, we conclude that $G_{[m] \times [n]}(\eta(\xx)) = G_\emptyset(\xx)$.
\end{proof}

\subsection{Central charge}
\label{sec:central charge}

Berenstein and Kazhdan defined a function called the \defn{central charge} on products of decorated unipotent crystals\footnote{Technically, the central charge is only defined on products of unipotent crystals that come from \defn{unipotent bicrystals}.}~\cite{BK07}. This function is invariant under the geometric crystal operators, so if it is positive, its tropicalization is constant on each connected component of the corresponding tensor product of combinatorial crystals, and thus defines a $q$-analogue of tensor product multiplicity. It is unknown in general whether the central charge is positive, and even in the cases where it is known to be positive (such as the product of ``full'' Gelfand--Tsetlin geometric crystals $\GT_n$), this $q$-analogue does not seem to have received much attention.

For the product $(X_n)^m$ of the $\GL_n$-unipotent crystal $X_n$, the central charge is given by
\[
\Delta(\xx) = F(\xx) - F(P),
\]
where $\gRSK(\xx) = (P,Q)$. Theorem~\ref{thm:dec} and Corollary~\ref{cor:gRSK props}(2) yield the following.

\begin{prop}
\label{prop:cc}
The central charge of the $\GL_n$-geometric crystal $(X_n)^m$ is given by
\begin{equation}
\label{eq:Delta}
\Delta(\xx) = F(Q) + \delta_{m,n} z_{n,n},
\end{equation}
where $z_{n,n}$ is the last part of the shape of $P$ and $Q$ (when $m=n$). In particular, the central charge is positive.
\end{prop}

The positivity of the central charge in this case was not previously known.

\medskip

The components of the tensor product of one-row $\GL_n$-crystals are labeled by the RSK recording tableau $Q$, and~\eqref{eq:Delta} gives a simple description of the $q$-weight assigned to $Q$ by the tropicalization of the central charge. Interestingly, this $q$-weight is completely different from---and much simpler than---the charge statistic (or closely related energy function), which appears in the ``usual'' $q$-analogue of tensor product multiplicity for the product of one-row tableaux~\cite{LS78, NY97}.

\begin{ex}
\label{ex:central charge}
Let $m=n=2$, and consider the $\GL_2$-representation $\Sym^{\mu_1}(\CC^2) \otimes \Sym^{\mu_2}(\CC^2)$, with $\mu_1 \geq \mu_2$. The irreducible components of this representation are labeled by the semistandard Young tableaux of content $(\mu_1, \mu_2)$, or equivalently, by the Gelfand--Tsetlin patterns
\[
\begin{array}{ccccccc}
& \mu_1 \\
\mu_1 + k && \mu_2-k
\end{array},
\qquad 0 \leq k \leq \mu_2.
\]
By Proposition~\ref{prop:cc}, the tropicalization of the central charge evaluates to
\[
\min(k, \mu_1-\mu_2+k,\mu_2-k) = \min(k,\mu_2-k)
\]
on these Gelfand--Tsetlin patterns. Thus, the $q$-analogue of the character of $\Sym^{\mu_1}(\CC^2) \otimes \Sym^{\mu_2}(\CC^2)$ coming from the tropicalization of the central charge is
\[
s_{\mu_1,\mu_2} + q s_{\mu_1+1, \mu_2-1} + q^2 s_{\mu_1+2, \mu_2-2} + \ldots + q^2 s_{\mu_1+\mu_2-2,2} + q s_{\mu_1+\mu_2-1,1} + s_{\mu_1+\mu_2},
\]
where $s_\la$ denotes a Schur polynomial in two variables. By contrast, the charge-graded analogue of the character of $\Sym^{\mu_1}(\CC^2) \otimes \Sym^{\mu_2}(\CC^2)$ is the modified Hall--Littlewood polynomial
\[
Q'_{(\mu_1,\mu_2)}(x_1,x_2;q) = \sum_{k=0}^{\mu_2} q^k s_{\mu_1+k,\mu_2-k}.
\]
\end{ex}

\begin{ex}
For $m > 2$, the difference between the central charge-graded $q$-analogue and the modified Hall--Littlewood polynomial is more pronounced. For example, when $m=3,n=2,$ and $(\mu_1,\mu_2,\mu_3) = (4,3,2)$, the central charge-graded $q$-analogue of the character of $\Sym^{\mu_1}(\CC^2) \otimes \Sym^{\mu_2}(\CC^2) \otimes \Sym^{\mu_3}(\CC^2)$ is
\[
s_9 + (1+1)s_{81} + (1+1+q)s_{72} + (1+1+q)s_{63} + (1+1)s_{54},
\]
whereas
\[
Q'_{432}(x_1,x_2;q) = q^7 s_9 + (q^5 + q^6) s_{81} + (q^3+q^4+q^5) s_{72} + (q^2+q^3+q^4) s_{63} + (q^2+q^3) s_{54}.
\]
\end{ex}

\section{Invariants of the geometric crystal operators}
\label{sec:e}

For the remainder of the paper, we work in the field $\CC(X)$ of rational functions in the $mn$ indeterminates $X = \{x_i^j \mid i \in [m], j \in [n]\}$. The actions of the $\GL_m$-geometric crystal operators $e_i^c$ (resp., the $\GL_n$-geometric crystal operators $\ov{e}_j^c$) on $(X_m)^n$ (resp., $(X_n)^m$) induce birational actions on $\CC(X)$, which we denote by the same symbols. We write $\Inv_e$ (resp., $\Inv_{\ov{e}}$) for the subfield of $\CC(X)$ consisting of fixed points of the $e_i^c$ (resp., $\ov{e}_j^c$), and we refer to these fixed points as \defn{$e$-invariants} (resp., \defn{$\ov{e}$-invariants}). We also write $\Inv_{e \ov{e}}$ for the intersection $\Inv_e \cap \Inv_{\ov{e}}$, whose elements we call \defn{$e \ov{e}$-invariants}.

Set $Z = \{z_{i,j}, z'_{j,i}\} \subset \CC(X)$, where $z_{i,j}$ and $z'_{j,i}$ denote the rational functions in the $x_i^j$ which give the entries of the $P$-pattern and $Q$-pattern, respectively (note that $Z$ contains only one copy of the entries $z_{k,n} = z'_{k,m}$ of the common shape). Since geometric RSK is a birational isomorphism, it induces an automorphism $\gRSK \colon \CC(X) \rightarrow \CC(X)$ which sends the elements of $X$ to the elements of $Z$. It follows that the elements of $Z$ are algebraically independent generators of $\CC(X)$, so we may view this field alternatively as the field $\CC(Z)$ of rational functions in indeterminates $Z$.

In \S \ref{sec:generators}, we describe two algebraically independent generating sets for each of the fields $\Inv_e$, $\Inv_{\ov{e}}$, and $\Inv_{e \ov{e}}$, one consisting of polynomials in the $X$ variables, and one consisting of polynomials in the $Z$ variables. We reduce the proof that these are indeed generating sets to Theorem~\ref{thm:main Z}, which is proved in \S \ref{sec:connectedness}.

\subsection{Generating sets}
\label{sec:generators}
Since the $\GL_m$-geometric crystal operators act only on the $Q$-pattern and geometric RSK commutes with the geometric crystal operators, one see immediately that the $z_{i,j}$ are $e$-invariants. The crucial step in proving Theorem~\ref{thm:main}(a) is the following result.

\begin{thm}
\label{thm:main Z}
The subfield $\Inv_e \subset \CC(Z)$ is equal to $\CC(z_{i,j})$.
\end{thm}

We will prove Theorem~\ref{thm:main Z} in \S\ref{sec:connectedness}.

\begin{cor}
\label{cor:Z generators}
$\Inv_{\ov{e}}$ is generated by the $z'_{j,i}$, and $\Inv_{e \ov{e}}$ is generated by the entries $z_{1,n}, \ldots, z_{\min(m,n),n}$ of the common shape.
\end{cor}

\begin{proof}
The first assertion follows immediately from Theorem~\ref{thm:main Z} and the symmetry of $\gRSK$ (Corollary~\ref{cor:gRSK props}(3)). Since the set $Z$ is algebraically independent, we have
\[
\Inv_e \cap \Inv_{\ov{e}} = \CC(z_{i,j}) \cap \CC(z'_{j,i}) = \CC(\{z_{i,j}\} \cap \{z'_{j,i}\}) = \CC(z_{k,n}),
\]
which proves the second assertion.
\end{proof}

Before giving our generating sets consisting of polynomials in the $X$ variables, we explicitly describe the entries of the $n \times n$ matrix $M(\xx) = W(\xx_1) \cdots W(\xx_m)$ defined in \S \ref{sec:unipotent}. For $k \in [m]$ and $r \in [n]$ such that $k+r \geq m+1$, define
\[
\lE{k}{r}(\xx) = \sum_{1 \leq i_1 < i_2 < \cdots < i_k \leq m} x_{i_1}^{r+1-i_1} x_{i_2}^{r+2-i_2} \cdots x_{i_k}^{r+k-i_k}.
\]
The assumptions on $k$ and $r$ guarantee that the superscripts lie in $[n]$. Note that if we ignore superscripts, each $\lE{k}{r}(\xx)$ is just the $k$th elementary symmetric polynomial in $m$ variables. We call $\lE{k}{r}$ a \defn{$P$-type loop elementary symmetric function} (this terminology will be explained in \S \ref{sec:R}).

It is straightforward to prove by induction on $m$ that
\[
M(\xx)_{ij} = \begin{cases}
\lE{m+j-i}{i}(\xx) & \text{ if } i \geq j \text{ and } m+j-i > 0 \\
1 & \text{ if } m+j-i = 0 \\
0 & \text{ otherwise}
\end{cases}.
\]
This shows that $M(\xx)$ lies in $(B^-)^{\leq m}$, and its non-trivial entries (\textit{i.e.}, entries not equal to $0$ or $1$) are the $P$-type loop elementary symmetric functions. Interchanging the roles of $m$ and $n$, we see that the $m \times m$ matrix $M(\xx^t) = W(\xx^1) \cdots W(\xx^n)$ lies in $(B^-)^{\leq n}$, and has non-trivial entries
\[
\lE{k'}{r'}(\xx^t) = \sum_{1 \leq j_1 < j_2 < \cdots < j_{k'} \leq n} x_{r'+1-j_1}^{j_1} x_{r'+2-j_2}^{j_2} \cdots x_{r'+k'-j_{k'}}^{j_{k'}}
\]
for $k' \in [n], r' \in [m]$, and $k'+r' \geq n+1$.
We call $\lE{k'}{r'}(\xx^t)$ a \defn{$Q$-type barred loop elementary symmetric function} (these polynomials will be further examined in the sequel to this paper~\cite{BFPSII}).

For $k = 1, \ldots, p = \min(m,n)$, define the \defn{shape invariant}
\[
S_k(\xx) = \Delta_{[k,n],[1,n-k+1]}(M(\xx)) = \det\left(\bigl(\lE{m-k+1+j-i}{k+i-1}(\xx)\bigr)_{i,j \in [n-k+1]}\right),
\]
where in the last expression we set $\lE{0}{r} = 1$, and $\lE{k}{r} = 0$ if $k < 0$ or $k > m$.

\begin{remark}
\label{rem:Schur}
A reader familiar with symmetric functions will recognize that if we replace $\lE{k}{r}(\xx)$ with the elementary symmetric polynomial $e_k(x_1, \ldots, x_m)$, the last expression becomes the (dual) Jacobi--Trudi formula for the Schur function in $m$ variables associated to the rectangular partition $\bigl((n-k+1)^{m-k+1}\bigr)$ (see, \textit{e.g.},~\cite[Ch.~7]{EC2}). In fact, we will see in \S \ref{sec:Schur} that $S_k(\xx)$ is a generating function for semistandard tableaux of this shape.
\end{remark}

\begin{cor}
\label{cor:X generators}
\
\begin{enumerate}
\item The $P$-type loop elementary symmetric functions are algebraically independent generators of the field $\Inv_e$ (this is the content of Theorem~\ref{thm:main}(1)).
\item The $Q$-type barred loop elementary symmetric functions are algebraically independent generators of the field $\Inv_{\ov{e}}$.
\item The shape invariants are algebraically independent generators of the field $\Inv_{e\ov{e}}$.
\end{enumerate}
\end{cor}

\begin{proof}
By definition, the $P$-pattern $(z_{i,j})$ associated to $\xx$ is the image of $M(\xx)$ under $\Psi_n^{\leq m} \colon (B^-)^{\leq m} \rightarrow \GT_n^{\leq m}$. This means that each $z_{i,j}$ is a ratio of minors of the matrix $M(\xx)$, so $\CC(z_{i,j}) \subseteq \CC\bigl(\lE{k}{r}(\xx)\bigr)$. On the other hand, we have $M(\xx) = \Phi_n^{\leq m} (z_{i,j})$ by Lemma \ref{lem:Phi inverse}, so each entry of $M(\xx)$ is a Laurent polynomial in the $z_{i,j}$, and we have the opposite inclusion $\CC\bigl(\lE{k}{r}(\xx)\bigr) \subseteq \CC(z_{i,j})$. Thus, $\Inv_e = \CC\bigl(\lE{k}{r}(\xx)\bigr)$ by Theorem~\ref{thm:main Z}. The sets $\{\lE{k}{r}(\xx)\}$ and $\{z_{i,j}\}$ have the same cardinality (the dimension of $\GT_n^{\leq m}$), so the algebraic independence of the $z_{i,j}$ implies the algebraic independence of the $\lE{k}{r}(\xx)$. This proves part (1), and part (2) follows by symmetry.

The common shape of the $P$- and $Q$-patterns is, by definition, given by
\[
(z_{1,n}, \ldots, z_{p-1,n}, z_{p,n}) = \left(\dfrac{S_1}{S_2}, \ldots, \dfrac{S_{p-1}}{S_p}, S_p\right).
\]
This means the shape invariants generate the same field as the $z_{k,n}$, so part (3) follows from Corollary~\ref{cor:Z generators}.
\end{proof}

\begin{remark}
The symmetry of gRSK implies that the shape invariant $S_k$ is also equal to the $(m-k+1) \times (m-k+1)$ minor in the bottom-left corner of $M(\xx^t)$, which can be viewed as a generalization of the Schur function in $n$ variables associated to the partition $\bigl((m-k+1)^{n-k+1}\bigr)$.
\end{remark}

\subsection{Connectedness and proof of Theorem~\ref{thm:main Z}}
\label{sec:connectedness}

A crystal is \defn{connected} if one can get between any two elements by a sequence of crystal raising and lowering operators. An important property of crystals coming from representations is that a representation is irreducible if and only if its crystal is connected; thus, one can identify the irreducible components of a representation simply by looking at connected components of the crystal. As an example, the set of Gelfand--Tsetlin patterns with fixed shape $\la$ form a connected crystal, corresponding to the Schur module indexed by $\la$.

We now address the analogous question for Gelfand--Tsetlin geometric crystals. The following result is based on combining the results of Kashiwara--Nakashima--Okado~\cite{KNO10} and Kanakubo--Nakashima~\cite{KanNak}.

\begin{thm}
\label{thm:connected}
Fix $\omega = (z_{1,n},\ldots,z_{p,n}) \in (\CC^*)^p$, where $p = \min(m,n)$, and let $\GT_n^{\leq m}(\omega)$ be the subvariety of $\GT_n^{\leq m}$ consisting of patterns of shape $\omega$. The action of the geometric crystal operators $\ov{e}_1^c, \ldots, \ov{e}_{n-1}^c$ on $\GT_n^{\leq m}(\omega)$ has a dense orbit.
\end{thm}

\begin{proof}
We first show that it suffices to prove the result for $\GT_n^{\leq m}(1, \ldots, 1)$. Given $\omega \in (\CC^*)^p$ and $\zz \in \GT_n^{\leq m}$, let $\omega \cdot \zz$ be the pattern obtained by replacing each $z_{i,j}$ with $\omega_i z_{i,j}$. We claim that
\begin{equation}
\label{eq:omega zz}
\ov{e}_j^c(\omega \cdot \zz) = \omega \cdot \ov{e}_j^c(\zz).
\end{equation}
Multiplication by $(\omega_1^{-1}, \ldots, \omega_p^{-1})$ gives a bijection from $\GT_n^{\leq m}(\omega)$ to $\GT_n^{\leq m}(1, \ldots, 1)$, so~\eqref{eq:omega zz} implies that the orbit structure under the geometric crystal operators is independent of $\omega$.

To prove~\eqref{eq:omega zz}, let $T_\omega$ be the diagonal $n \times n$ matrix with entries $(\omega_1, \ldots, \omega_p, 1, \ldots, 1)$. It is easy to see that $\Phi(\omega \cdot \zz) = T_\omega \Phi(\zz)$ and $\Psi(T_\omega M) = \omega \cdot \Psi(M)$ (here we write $\Phi,\Psi$ instead of $\Phi_n^{\leq m}, \Psi_n^{\leq m}$). Set $M = \Phi(\zz)$. Working through the various definitions, we compute
\begin{align*}
\ov{e}_j^c(\omega \cdot \zz) &= \Psi\left(x_j\left((c-1)\ov{\varphi}_j(T_\omega M)\right) \cdot T_\omega M \cdot x_j\left((c^{-1} - 1)\ov{\varepsilon}_j(T_\omega M)\right)\right) \\
&= \Psi\left(x_j\left((c-1)\frac{\omega_j}{\omega_{j+1}}\ov{\varphi}_j(M)\right) \cdot T_\omega M \cdot x_j\left((c^{-1} - 1)\ov{\varepsilon}_j(M)\right)\right) \\
&= \Psi\left(T_\omega x_j\left(c-1)\ov{\varphi}_j(M)\right) \cdot M \cdot x_j\left((c^{-1} - 1)\ov{\varepsilon}_j(M)\right)\right) \\
&= \omega \cdot \ov{e}_j^c(\zz).
\end{align*}

Let $\mb{i} = (i_1, \ldots, i_\ell) \in [n-1]^\ell$ be a reduced word for an element of the Weyl group $S_n$, and consider the map $\Phi_{\mb{i}} \colon (\CC^*)^\ell \rightarrow B^- \subset \GL_n$ defined by
\[
\Phi_{\mb{i}} \colon (z_1, \ldots, z_\ell) \mapsto Y_{i_1}(z_1) \cdots Y_{i_\ell}(z_\ell), \quad \text{ where } \quad Y_j(z) = I + (z^{-1}-1)E_{j,j} + (z-1)E_{j+1,j+1} + E_{j+1,j}.
\]
One easily verifies that
\[
Y_i(z_i) Y_{i+1}(z_{i+1}) \cdots Y_{n-1}(z_{n-1}) = W^i\left(\dfrac{1}{z_i}, \dfrac{z_i}{z_{i+1}}, \ldots, \dfrac{z_{n-2}}{z_{n-1}}, z_{n-1}\right),
\]
where $W^i$ is defined by~\eqref{eq:row matrix}. This implies that the restriction of the parametrization $\Phi_n^{\leq m}$ to $\GT_n^{\leq m}(1, \ldots, 1)$ is the same (up to a simple change of variables) as the parametrization $\Phi_{\mb{i}(m)}$, where
\[
\mb{i}(m) = (p',p'+1,\ldots, n-1, p'-1,p',p'+1,\ldots,n-1, \ldots, 1, 2, \ldots, n-1), \qquad\quad p' = \min(m,n-1).
\]
(Note that $\mb{i}(m)$ is a reduced word for the longest element $w_0$ if $m \geq n-1$.)

One of the first results in the theory of geometric crystals is that for any reduced word ${\mb{i}}$, the tropicalization, with respect to $\Phi_{\mb{i}}$, of the geometric crystal induced from the unipotent crystal structure on $B^-$ is the free combinatorial crystal $B_{\mb{i}} = B_{i_1} \otimes \cdots \otimes B_{i_\ell}$~\cite[\S 4.4, Rem.~1]{BK00}. Kanakubo and Nakashima proved that for any reduced word $\mb{i}$, the crystal $B_{\mb{i}}$ is connected~\cite[Thm.~8.7]{KanNak}. Kashiwara, Nakashima, and Okado proved that if the tropicalization of a geometric crystal is connected, then the geometric crystal has a dense orbit under the action of the geometric crystal operators~\cite[Thm.~3.3]{KNO10}.
\end{proof}

The proof of Theorem~\ref{thm:main Z} requires two additional general lemmas. For the first lemma, let $X$ be an irreducible algebraic variety of dimension $N$ over a field $\fieldalt$, and let $G = \{G_i\}_{i \in I}$ be a collection of algebraic groups over $\fieldalt$, each acting rationally on $X$. Define a \defn{$G$-orbit} in $\fieldalt$ to be a subset obtained by starting with a point $x_0 \in X$, and repeatedly acting by elements of the groups $G_i$. Let $\fieldalt(X)^G = \bigcap_{i \in I} \fieldalt(X)^{G_i}$ denote the subfield of the fraction field of $X$ consisting of invariants for all the $G_i$.

\begin{lemma}
\label{lem:orbits}
If every point of $X$ is contained in a $G$-orbit closure of dimension at least $N-d$, then $\fieldalt(X)^G$ has transcendence degree at most $d$ over $\fieldalt$.
\end{lemma}

\begin{proof}
Suppose $f_1, \ldots, f_k \in \fieldalt(X)^G$ are algebraically independent over $K$. Let $V$ be the open subset of $X$ where all the $f_j$ are defined. Consider the morphism
\[
f = (f_1, \ldots, f_k) \colon V \rightarrow \bb{A}^k.
\]
Since the $f_j$ are algebraically independent, the pullback $f^*$ is injective, so $f$ is dominant. A standard result in algebraic geometry (\textit{e.g.},~\cite[Ex. II.3.22]{Hart}) says that the general fiber of a dominant morphism $f \colon X \rightarrow Y$, where $X$ and $Y$ are irreducible, has dimension $\dim X - \dim Y$. In our setting, this means that there is a non-empty open subset $V' \subseteq V$ such that for all $z \in f(V')$, the fiber $f^{-1}(z)$ has dimension $N-k$.

Choose $z \in f(V')$, and let $\ov{\mc{O}}$ be a $G$-orbit closure of dimension at least $N-d$ which intersects $f^{-1}(z)$. The irreducibility of $V$ ensures that $\ov{\mc{O}} \cap U$ has the same dimension as $\ov{\mc{O}}$. The $f_j$ are constant on each $G$-orbit, so we must have $\ov{\mc{O}} \cap V \subseteq f^{-1}(z)$, and thus $k \leq d$.
\end{proof}

The second lemma is a straightforward exercise in field theory.

\begin{lemma}
\label{lem:no finite extensions}
If $x_1, \ldots, x_\ell$ are independent transcendentals over a field $\fieldalt$, then $\fieldalt$ has no proper algebraic extensions inside the field $\fieldalt(x_1, \ldots, x_\ell)$.
\end{lemma}


\begin{proof}[Proof of Theorem~\ref{thm:main Z}]
Let $\abs{P} = \dim \GT_n^{\leq m}$ and $\abs{Q} = \dim \GT_m^{\leq n}$ be the number of entries in the $P$- and $Q$-patterns, respectively. We observed at the beginning of \S \ref{sec:generators} that $\Inv_e$ contains the $z_{i,j}$, so it has transcendence degree at least $\abs{P}$ over $\CC$. On the other hand, Theorem~\ref{thm:connected} implies that every point of $\GT_n^{\leq m} * \GT_m^{\leq n}$ is contained in a $G$-orbit closure of dimension $\abs{Q} - \min(m,n) = mn - \abs{P}$, where $G = \{G_i\}_{i \in [m-1]}$, with $G_i = \CC^*$ acting by the geometric crystal operator $e_i$. Thus, $\Inv_e$ has transcendence degree at most $|P|$ by Lemma~\ref{lem:orbits}, so it is an algebraic extension of $\CC(z_{i,j})$. By Lemma~\ref{lem:no finite extensions}, $\CC(z_{i,j})$ has no proper algebraic extensions inside $\CC(Z)$, so we must have $\Inv_e = \CC(z_{i,j})$.
\end{proof}

\section{Invariants of the Weyl group action}
\label{sec:R}

\subsection{Weyl group action and geometric \texorpdfstring{$R$}{R}-matrix}
\label{sec:geometric R}

Let $X$ be a $\GL_m$-geometric crystal. For $i \in [m-1]$, define a rational map $s_i \colon X \to X$ by
\[
s_i(x) = e_i^{\frac{1}{\alpha_i(\gamma(x))}}(x) = e_i^{\frac{\ve_i(x)}{\vp_i(x)}}(x).
\]
Berenstein and Kazhdan proved that the maps $s_i$ generate a birational action of the Weyl group $S_m$ on $X$~\cite[Prop. 2.3]{BK00} (in fact, the purpose of axiom (3) in Definition~\ref{defn:geom crystal} is to ensure that the $s_i$ satisfy the braid relations). We will show that the birational Weyl group action on $(X_m)^n$ agrees with the action of the geometric $R$-matrix, a map coming from the representation theory of quantum groups of type $A_{n-1}^{(1)}$.

The \defn{geometric $R$-matrix (of type $A_{n-1}^{(1)}$)} is the rational map
\begin{align*}
R \colon (\CC^*)^n \times (\CC^*)^n &\rightarrow (\CC^*)^n \times (\CC^*)^n \\
\bigl((x_1, \ldots, x_n), (y_1, \ldots, y_n)\bigr) &\mapsto \bigl((y'_1, \ldots, y'_n), (x'_1, \ldots, x'_n)\bigr)
\end{align*}
given by
\[
y'_j = y_j \dfrac{\kappa_{j+1}}{\kappa_j}, \quad x'_j = x_j\dfrac{\kappa_j}{\kappa_{j+1}}, \qquad \kappa_r = \kappa_r(\xx,\mb{y}) = \sum_{k=0}^{n-1} y_r \cdots y_{r+k-1} x_{r+k+1} \cdots x_{r+n-1},
\]
with subscripts interpreted modulo $n$. For $i \in [m-1]$, define $R_i \colon (X_n)^m \rightarrow (X_n)^m$ by
\[
R_i(\xx_1, \ldots, \xx_i, \xx_{i+1}, \ldots, \xx_m) = (\xx_1, \ldots, \xx'_{i+1}, \xx'_i, \ldots, \xx_m),
\]
where $(\xx'_{i+1}, \xx'_i) = R(\xx_i, \xx_{i+1})$.

The following result is a variant of~\cite[Lem.~3.2]{KajNouYam}.

\begin{prop}
\label{prop:R=Weyl}
If $(\xx^1, \ldots, \xx^n) \in (X_m)^n$ are the columns of a matrix $\xx \in \Mat_{m \times n}(\CC^*)$, and $(\xx_1, \ldots, \xx_m) \in (X_n)^m$ are the rows of $\xx$, then for $i \in [m-1]$ we have
\[
s_i(\xx^1, \ldots, \xx^n) = R_i(\xx_1, \ldots, \xx_m).
\]
\end{prop}

\begin{proof}
By Lemma~\ref{lem:explicit formulas basic}, $s_i$ fixes row $\xx_k$ for $k \neq i,i+1$, and acts on rows $\xx_i, \xx_{i+1}$ by
\[
x_i^j \mapsto x_i^j \dfrac{\sigma^j\left(\xx_i, \xx_{i+1}; \frac{\pi_{i+1}}{\pi_i}\right)}{\sigma^{j-1}\left(\xx_i, \xx_{i+1}; \frac{\pi_{i+1}}{\pi_i}\right)}, \qquad x_{i+1}^j \mapsto x_{i+1}^j \dfrac{\sigma^{j-1}\left(\xx_i, \xx_{i+1}; \frac{\pi_{i+1}}{\pi_i}\right)}{\sigma^j\left(\xx_i, \xx_{i+1}; \frac{\pi_{i+1}}{\pi_i}\right)},
\]
where $\pi_i = x_i^1 \cdots x_i^n$, and $\sigma^j(\xx_i,\xx_{i+1};c)$ is defined by~\eqref{eq:sigma^j}. For $j = 0, \ldots, n$, we compute
\begin{align*}
\sigma^j\left(\xx_i, \xx_{i+1}; \frac{\pi_{i+1}}{\pi_i}\right) &= \dfrac{\pi_{i+1}}{\pi_i} \sum_{k=1}^j x_{i+1}^1 \cdots x_{i+1}^{k-1} x_i^{k+1} \cdots x_i^n + \sum_{k=j+1}^n x_{i+1}^1 \cdots x_{i+1}^{k-1} x_i^{k+1} \cdots x_i^n \\
&= \dfrac{x_{i+1}^1 \cdots x_{i+1}^j}{x_i^1 \cdots x_i^j} \left(\sum_{k=1}^j x_{i+1}^{j+1} \cdots x_{i+1}^n x_{i+1}^1 \cdots x_{i+1}^{k-1} x_i^{k+1} \cdots x_i^j\right. \\
& \hspace{80pt} + \left.\sum_{k=j+1}^n x_{i+1}^{j+1} \cdots x_{i+1}^{k-1} x_i^{k+1} \cdots x_i^n x_i^1 \cdots x_i^j\right) \\
&= \dfrac{x_{i+1}^1 \cdots x_{i+1}^j}{x_i^1 \cdots x_i^j} \kappa_{j+1}(\xx_i, \xx_{i+1}).
\end{align*}
The proposition follows.
\end{proof}

\begin{remark}
Yamada~\cite{Yamada01} defined the geometric (or birational) $R$-matrix as a geometric lift of the \defn{combinatorial $R$-matrix}, which is the unique crystal isomorphism that permutes adjacent factors in a tensor product of one-row affine crystals (see, \textit{e.g.},~\cite{HKOTY99,HKOTT02}). Several different proofs of the fact that the maps $R_i$ generate a birational $S_m$-action have appeared in the literature~\cite{Yamada01,KajNouYam,Etingof,LP12,LP13II}.
\end{remark}

\subsection{Loop symmetric functions}
\label{sec:LSym}

The geometric $R$-matrices $R_1, \ldots, R_{m-1}$ generate a group of automorphisms of the field $\CC(x_i^j)$ isomorphic to $S_m$. The fixed field $\Inv_R$ of this group of automorphisms has been studied in several papers, starting with~\cite{LP12}, under the name of loop symmetric functions. It is more convenient to describe loop symmetric functions in terms of variables $\lx{i}{r}$, where the superscript $r$ is considered modulo $n$, and $x_i^j = \lx{i}{j+i-1}$. In other words, we make the identification
\[
\xx_i = (x_i^1, x_i^2, \ldots, x_i^n) = (\lx{i}{i}, \lx{i}{i+1}, \ldots \lx{i}{i+n-1}).
\]

\begin{dfn}
For $k \in [m]$ and $r \in [n]$, define the \defn{loop elementary symmetric function}
\[
\lE{k}{r} = \lE{k}{r}(\xx_1, \ldots, \xx_m) = \sum_{1 \leq i_1 < i_2 < \cdots < i_k \leq m} \lx{i_1}{r} \lx{i_2}{r+1} \cdots \lx{i_k}{r+k-1}.
\]
It is convenient to allow arbitrary integers $r$ in the superscript by setting $\lE{k}{r} = \lE{k}{r \mod n}$, and to define $\lE{0}{r} = 1$ and $\lE{k}{r} = 0$ if $k < 0$ or $k > m$.

Define $\LSym = \LSym_m(n)$ to be the ring generated (over $\CC$) by the $\lE{k}{r}$. We call this the \defn{ring of loop symmetric functions in $m$ variables and $n$ colors}.
\end{dfn}

When $k+r \geq m+1$, $\lE{k}{r}$ is the $P$-type loop elementary symmetric function that appeared in \S \ref{sec:generators} as an entry of the $n \times n$ matrix $M(\xx)$. We now describe an infinite ``cyclic extension'' of the matrix $M(\xx)$ that has all the loop elementary symmetric functions as entries.

Define an \defn{$n$-periodic matrix} to be a $\mathbb{Z} \times \mathbb{Z}$ array $(A_{ij})_{i,j \in \mathbb{Z}}$ such that $A_{ij} = A_{i+n,j+n}$ for all $i,j$, and $A_{ij} = 0$ if $j-i$ is sufficiently large. The second hypothesis ensures that multiplication of these matrices is well-defined. Given an $n$-tuple of variables $(x^1, \ldots, x^n)$, let $\widetilde{W} = \widetilde{W}(x^1, \ldots, x^n)$ be the $n$-periodic matrix with $\widetilde{W}_{j,j-1} = 1$, $\widetilde{W}_{jj} = x^j$ (the superscript of $x^j$ is interpreted mod $n$), and all other entries zero. This matrix is called a \defn{whirl} in~\cite{LP12}. For example, when $n=3$,
\begin{equation*}
\widetilde{W}(x^1, x^2, x^3) =
\begin{pmatrix}
x^1 & 0 & 0 & 0 \\
1 & x^2 & 0 & 0  \\
0 & 1 & x^3 & 0 & \hdots \\
0 & 0 & 1 & x^1  \\
0 & 0 & 0 & 1 \\
&\vdots &&& \ddots
\end{pmatrix}.
\end{equation*}
Note that we are depicting only the quadrant of the matrix with $i,j \geq 1$.

For $\xx_1, \ldots, \xx_m \in (\CC^*)^n$, define $\widetilde{M}(\xx_1, \ldots, \xx_m) = \widetilde{W}(\xx_1) \cdots \widetilde{W}(\xx_m)$. It is straightforward to prove by induction that the entries of this matrix are the loop elementary symmetric functions:
\begin{equation}
\label{eq:tilde M entries}
\widetilde{M}(\xx_1, \ldots, \xx_m)_{ij} = \lE{m+j-i}{i}(\xx_1, \ldots, \xx_m).
\end{equation}
An example of the matrix $\widetilde{M}(\xx_1, \ldots, \xx_m)$ appears in Figure \ref{fig:M unfolded}. As one can see in this example, $\widetilde{M}(\xx_1, \ldots, \xx_m)$ consists of $n \times n$ blocks that repeat along (block) diagonals, and the top nonzero block is the matrix $M(\xx)$ that has appeared throughout the paper. The $P$-type loop elementary symmetric functions are exactly the loop elementary symmetric functions appearing in this block.

\begin{figure}
\[
\left( \begin{array}{cc|cc|cc|c}
\lx{1}{1}\lx{2}{2}\lx{3}{1}  = \lE{3}{1} & 0 & 0 & 0 & 0 & 0 & \hdots \\
\lx{1}{2}\lx{2}{1} + \lx{1}{2}\lx{3}{1} + \lx{2}{2}\lx{3}{1} = \lE{2}{2} & \lx{1}{2}\lx{2}{1}\lx{3}{2} = \lE{3}{2} & 0 & 0 & 0 & 0 \\ \hline
\lx{1}{1} + \lx{2}{1} + \lx{3}{1} = \lE{1}{1} & \lx{1}{1}\lx{2}{2} + \lx{1}{1}\lx{3}{2} + \lx{2}{1}\lx{3}{2} = \lE{2}{1} & \lE{3}{1} & 0 & 0 & 0 \\
1 & \lx{1}{2} + \lx{2}{2} + \lx{3}{2} = \lE{1}{2} & \lE{2}{2} & \lE{3}{2} & 0 & 0 \\ \hline
0 & 1 & \lE{1}{1} & \lE{2}{1} & \lE{3}{1} & 0 \\
0 & 0 & 1 & \lE{1}{2} & \lE{2}{2} & \lE{3}{2} \\ \hline
0 & 0 & 0 & 1 & \lE{1}{1} & \lE{2}{1} \\
0 & 0 & 0 & 0 & 1 & \lE{1}{2} \\ \hline
\vdots &  &  & & & & \ddots
\end{array} \right)
\]
\caption{The $2$-periodic matrix $\widetilde{M}(\xx_1, \xx_2, \xx_3)$, with $2 \times 2$ blocks indicated. We depict only the quadrant of the matrix with coordinates $i,j \geq 1$.}
\label{fig:M unfolded}
\end{figure}

\begin{lemma}[{\cite[Thm.~6.2]{LP12}}]
\label{lem:E is R invariant}
The entries of $\widetilde{M}(\xx_1, \ldots, \xx_m)$ are $R_i$-invariant for $i = 1, \ldots, m-1$. Thus, $\LSym$ is a subring of the invariant field $\Inv_R$.
\end{lemma}

\begin{thm}[{\cite[Thm.~4.1]{LPaff}}]
\label{thm:LSym}
The loop elementary symmetric functions are algebraically independent, and they generate $\Inv_R$. That is, we have
\[
\Inv_R = \Frac(\LSym) = \CC(\lE{k}{r}).
\]
\end{thm}

\begin{thm}[Fundamental theorem of loop symmetric functions]
The loop elementary symmetric functions generate the ring $\Inv_R \cap \CC[x_i^j]$ of polynomial $R$-invariants.
\label{thm:FTLSF}
\end{thm}

\begin{ex}
\
\begin{enumerate}
\item[(a)] When $n = 1$, $\lE{i}{1}$ is the elementary symmetric polynomial in the $m$ variables $\lx{1}{1}, \ldots, \lx{m}{1}$, and $R_i$ simply swaps the variables $\lx{i}{1}$ and $\lx{i+1}{1}$. Theorem~\ref{thm:FTLSF} reduces in this case to Newton's fundamental theorem of symmetric functions, which says that the ring of symmetric polynomials in $m$ variables is the polynomial ring in the elementary symmetric polynomials $e_1, \ldots, e_m$.

\item[(b)] When $m = 1$, there are no maps $R_i$, and $\LSym$ is the polynomial ring in the $n$ variables $\lx{1}{r} = \lE{1}{r}$, $r \in [n]$, so Theorem~\ref{thm:FTLSF} is trivial in this case.
\end{enumerate}
\end{ex}

Theorem~\ref{thm:FTLSF} is proved in Appendix~\ref{app:FTLSF}; the proof is elementary, but considerably more involved than in the classical $(n=1)$ case. The algebraic independence of the $\lE{k}{r}$ is also proved in the course of the argument.

\begin{cor}[Theorem~\ref{thm:main}(2)]
\label{cor:polynomial e invariants}
The $P$-type loop elementary symmetric functions generate the ring $\Inv_e \cap \CC[x_i^j]$ of polynomial $e$-invariants.
\end{cor}

\begin{proof}
It follows from Proposition~\ref{prop:R=Weyl} that $\Inv_e \subset \Inv_R$, so Theorem~\ref{thm:FTLSF} implies that each polynomial $e$-invariant is a polynomial in the loop elementary symmetric functions. On the other hand, Theorem~\ref{thm:main} says that each polynomial $e$-invariant is a ratio of polynomials in the $P$-type loop elementary symmetric functions. The result follows.
\end{proof}

\subsection{Loop Schur functions}
\label{sec:Schur}

For our next definition, we need to recall some notions from tableau combinatorics. We identity a partition with its Young diagram, which we view, following the English convention, as a northwest-justified collection of unit cells in the plane. For partitions $\mu$ and $\la$, we write $\mu \subseteq \la$ if the Young diagram of $\mu$ is contained in that of $\la$, and if $\mu \subseteq \la$, we define the skew diagram $\la/\mu$ to be the cells in $\la$ which are not in $\mu$. For a cell $s = (i, j)$ in the $i$th row and $j$th column (using matrix coordinates), the \defn{content} of $s$ is $c(s) = i - j$ (this is the opposite of the usual definition). Let $\lambda'$ denote the conjugate or transpose partition of $\lambda$. A \defn{semistandard Young tableau of shape $\la/\mu$} is a filling $T \colon \lambda/\mu \to \ZZ_{>0}$ of the cells of $\la/\mu$ with positive integers, such that rows are weakly increasing, and columns are strictly increasing. Let $\SSYT_{\leq m}(\la/\mu)$ be the set of semistandard Young tableaux of shape $\la/\mu$ with entries at most $m$.

For partitions $\mu \subseteq \la$ and a color $r \in [n]$, define the \defn{loop skew Schur function}
\[
\ls{\lambda/\mu}{r} = \ls{\la/\mu}{r}(\xx_1, \ldots, \xx_m) = \sum_{T \in \SSYT_{\leq m}(\la)} \xx^T,
\]
where $\xx^T = \prod_{s \in \lambda/\mu} \lx{T(s)}{c(s)+r}$. If $\mu = \emptyset$, then $\ls{\la/\emptyset}{r} = \ls{\la}{r}$ is a \defn{loop Schur function}. We refer to $c(s)+r$ (considered modulo $n$) as the \defn{color} of the cell $s$.

It is clear that $\ls{(1^k)}{r} = \lE{k}{r}$. For $k > 0$ and $r \in [n]$, define the \defn{loop homogeneous symmetric function}
\[
\lH{k}{r} = \ls{(k)}{r} = \sum_{1 \leq i_1 \leq \cdots \leq i_k \leq m} \lx{i_1}{r} \lx{i_2}{r-1} \cdots \lx{i_k}{r-k+1}.
\]
It is convenient to set $\lH{0}{r} = 1$ and $\lH{k}{r} = 0$ if $k < 0$.

The following analogue of the Jacobi--Trudi formula shows that the loop skew Schur functions lie in $\LSym$.
\begin{prop}[{\cite[Thm.~7.6]{LP12}}]
\label{prop:JT}
Suppose $\mu \subseteq \la$, and the conjugate partition $\la'$ has length at most $\ell$. Then
\[
\ls{\lambda/\mu}{r} = \det\left( \lE{\lambda'_i - \mu'_j + j - i}{r + \mu'_j - j + 1} \right)_{i,j=1}^{\ell}.
\]
\end{prop}

\begin{ex}
\label{ex:Schur}
We compute the loop Schur function $\ls{(4,2)}{1}(\xx_1, \xx_2)$ in the case $n=4$. The colors $c(s) + 1$ of the cells of $(4,2)$ are shown below, where Orange is $1$, Blue is $2$, Red is $3$, and Green is $4$:

\begin{center}
\begin{tikzpicture}[baseline=-40,scale=0.5]
\foreach \x/\y in {3/1} {
  \fill[green!40] (\x,-\y) rectangle (\x+1,-\y-1);
  \draw (\x+0.5,-\y-0.5) node {$G$};
}
\foreach \x/\y in {2/1,3/2} {
  \fill[orange!40] (\x,-\y) rectangle (\x+1,-\y-1);
  \draw (\x+0.5,-\y-0.5) node {$O$};
}
\foreach \x/\y in {4/1} {
  \fill[red!40] (\x,-\y) rectangle (\x+1,-\y-1);
  \draw (\x+0.5,-\y-0.5) node {$R$};
}
\foreach \x/\y in {2/2,5/1} {
  \fill[blue!40] (\x,-\y) rectangle (\x+1,-\y-1);
  \draw (\x+0.5,-\y-0.5) node {$B$};
}
\draw[thick] (2,-1) -- (6,-1) -- (6,-2) -- (4,-2) -- (4,-3) -- (2,-3) -- (2,-1);
\draw (3,-1) -- (3,-3);
\draw (4,-1) -- (4,-2);
\draw (5,-1) -- (5,-2);
\draw (2,-2) -- (6,-2);
\end{tikzpicture}\ .
\end{center}

\noindent There are three semistandard tableaux of shape $\lambda$ with entries in $\{1,2\}$:

\begin{center}
\begin{tikzpicture}[baseline=-40,scale=0.5]
\foreach \x/\y in {3/1} {
  \fill[green!40] (\x,-\y) rectangle (\x+1,-\y-1);
  \draw (\x+0.5,-\y-0.5) node {$1$};
}
\foreach \x/\y in {2/1} {
  \fill[orange!40] (\x,-\y) rectangle (\x+1,-\y-1);
  \draw (\x+0.5,-\y-0.5) node {$1$};
}
\foreach \x/\y in {3/2} {
  \fill[orange!40] (\x,-\y) rectangle (\x+1,-\y-1);
  \draw (\x+0.5,-\y-0.5) node {$2$};
}
\foreach \x/\y in {4/1} {
  \fill[red!40] (\x,-\y) rectangle (\x+1,-\y-1);
  \draw (\x+0.5,-\y-0.5) node {$$1};
}
\foreach \x/\y in {2/2} {
  \fill[blue!40] (\x,-\y) rectangle (\x+1,-\y-1);
  \draw (\x+0.5,-\y-0.5) node {$2$};
}
\foreach \x/\y in {5/1} {
  \fill[blue!40] (\x,-\y) rectangle (\x+1,-\y-1);
  \draw (\x+0.5,-\y-0.5) node {$1$};
}
\draw[thick] (2,-1) -- (6,-1) -- (6,-2) -- (4,-2) -- (4,-3) -- (2,-3) -- (2,-1);
\draw (3,-1) -- (3,-3);
\draw (4,-1) -- (4,-2);
\draw (5,-1) -- (5,-2);
\draw (2,-2) -- (6,-2);

\begin{scope}[xshift=8cm]
\foreach \x/\y in {3/1} {
  \fill[green!40] (\x,-\y) rectangle (\x+1,-\y-1);
  \draw (\x+0.5,-\y-0.5) node {$1$};
}
\foreach \x/\y in {2/1} {
  \fill[orange!40] (\x,-\y) rectangle (\x+1,-\y-1);
  \draw (\x+0.5,-\y-0.5) node {$1$};
}
\foreach \x/\y in {3/2} {
  \fill[orange!40] (\x,-\y) rectangle (\x+1,-\y-1);
  \draw (\x+0.5,-\y-0.5) node {$2$};
}
\foreach \x/\y in {4/1} {
  \fill[red!40] (\x,-\y) rectangle (\x+1,-\y-1);
  \draw (\x+0.5,-\y-0.5) node {$1$};
}
\foreach \x/\y in {2/2} {
  \fill[blue!40] (\x,-\y) rectangle (\x+1,-\y-1);
  \draw (\x+0.5,-\y-0.5) node {$2$};
}
\foreach \x/\y in {5/1} {
  \fill[blue!40] (\x,-\y) rectangle (\x+1,-\y-1);
  \draw (\x+0.5,-\y-0.5) node {$2$};
}
\draw[thick] (2,-1) -- (6,-1) -- (6,-2) -- (4,-2) -- (4,-3) -- (2,-3) -- (2,-1);
\draw (3,-1) -- (3,-3);
\draw (4,-1) -- (4,-2);
\draw (5,-1) -- (5,-2);
\draw (2,-2) -- (6,-2);
\end{scope}

\begin{scope}[xshift=16cm]
\foreach \x/\y in {3/1} {
  \fill[green!40] (\x,-\y) rectangle (\x+1,-\y-1);
  \draw (\x+0.5,-\y-0.5) node {$1$};
}
\foreach \x/\y in {2/1} {
  \fill[orange!40] (\x,-\y) rectangle (\x+1,-\y-1);
  \draw (\x+0.5,-\y-0.5) node {$1$};
}
\foreach \x/\y in {3/2} {
  \fill[orange!40] (\x,-\y) rectangle (\x+1,-\y-1);
  \draw (\x+0.5,-\y-0.5) node {$2$};
}
\foreach \x/\y in {4/1} {
  \fill[red!40] (\x,-\y) rectangle (\x+1,-\y-1);
  \draw (\x+0.5,-\y-0.5) node {$2$};
}
\foreach \x/\y in {2/2} {
  \fill[blue!40] (\x,-\y) rectangle (\x+1,-\y-1);
  \draw (\x+0.5,-\y-0.5) node {$2$};
}
\foreach \x/\y in {5/1} {
  \fill[blue!40] (\x,-\y) rectangle (\x+1,-\y-1);
  \draw (\x+0.5,-\y-0.5) node {$2$};
}
\draw[thick] (2,-1) -- (6,-1) -- (6,-2) -- (4,-2) -- (4,-3) -- (2,-3) -- (2,-1);
\draw (3,-1) -- (3,-3);
\draw (4,-1) -- (4,-2);
\draw (5,-1) -- (5,-2);
\draw (2,-2) -- (6,-2);
\end{scope}

\end{tikzpicture}\ .
\end{center}

\noindent Thus, we have
\[
\ls{(4,2)}{1}(\xx_1, \xx_2) = \lx{1}{1}\lx{1}{2}\lx{1}{3}\lx{1}{4}\lx{2}{1}\lx{2}{2} + \lx{1}{1}\lx{1}{3}\lx{1}{4}\lx{2}{1}(\lx{2}{2})^2 + \lx{1}{1}\lx{1}{4}\lx{2}{1}(\lx{2}{2})^2\lx{2}{3}.
\]
The reader may verify that
\[
\ls{(4,2)}{1}(\xx_1, \xx_2) = \det \begin{pmatrix}
\lE{2}{1} & 0 & 0 & 0 \\
\lE{1}{1} & \lE{2}{4} & 0 & 0\\
0 & 1 & \lE{1}{3} & \lE{2}{2} \\
0 & 0 & 1 & \lE{1}{2}
\end{pmatrix}.
\]
\end{ex}

The Jacobi--Trudi formula implies that the minors of the $n$-periodic matrix $\widetilde{M}(\xx_1, \ldots, \xx_m)$ are precisely the loop skew Schur functions (one must reflect a submatrix of $\widetilde{M}(\xx_1, \ldots, \xx_m)$ over the anti-diagonal to get a matrix in ``Jacobi--Trudi form,'' but this does not change the determinant). In particular, the entries of the $\gRSK$ $P$-pattern are ratios of loop Schur functions of rectangular shape. For $1 \leq i \leq m$ and $i \leq j \leq n$, let
\[
\Box(i, j) = \ls{(\underbrace{j-i+1, \ldots, j-i+1}_{m-i+1 \text{ times}})}{j}(\xx_1, \ldots, \xx_m)
\]
be the loop Schur function associated with an $(m-i+1) \times (j-i+1)$ rectangle, where the unique northwest corner has color $j$. Set $\Box(i+1,j) = 1$ if $i = j$ or $i = m$. Since $M(\xx)$ is the block of $\widetilde{M}(\xx_1, \ldots, \xx_m)$ consisting of rows and columns $1, \ldots, n$, the Jacobi--Trudi formula implies that
\begin{equation}
\label{eq:P loop formula}
z_{i,j} = \dfrac{\Delta_{[i,j],[1,j-i+1]}(M(\xx))}{\Delta_{[i+1,j],[1,j-i]}(M(\xx))} = \frac{\Box(i, j)}{\Box(i+1, j)}.
\end{equation}
The shape invariants are given by
\[
S_k(\xx) = \Box(k,n) = \ls{(\underbrace{n-k+1, \ldots, n-k+1}_{m-k+1 \text{ times}})}{n}(\xx_1, \ldots, \xx_m)
\]
for $k = 1, \ldots, \min(m,n)$, as promised in Remark~\ref{rem:Schur}.

\begin{conj}
The shape invariants generate the ring $\Inv_{e\ov{e}} \cap \CC[x_i^j]$.
\end{conj}

\appendix

\section{Proofs of technical lemmas}
\label{app:proofs}

\begin{proof}[Proof of Lemma~\ref{lem:explicit formulas basic}]
The formula for $\gamma(\xx)$ is immediate from~\eqref{eq:basic geometric crystal} and~\eqref{eq:functions prod}. For the other assertions, suppose $Y_1, \ldots, Y_n$ are $\GL_m$-geometric crystals, and $y_j \in Y_j$. Set
\[
A_k = \prod_{r=1}^{k-1} \ve_i(y_r) \prod_{r=k+1}^n \vp_i(y_r), \qquad A'_k = \prod_{r=1}^{k-1} \ve_i(y_r) \prod_{r=k+1}^{n-1} \vp_i(y_r)
\]
for $k \in [n]$ and $k \in [n-1]$, respectively. Using~\eqref{eq:functions prod} and the identity
\begin{equation}
\label{eq:rec id}
\prod_{r=1}^{n-1} \ve_i(y_r) + \sum_{k=1}^{n-1} A'_k \vp_i(y_n) = \sum_{k=1}^n A_k,
\end{equation}
one shows by induction on $n$ that
\begin{equation}
\label{eq:functions multi prod}
\ve_i(y_1, \ldots, y_n) = \dfrac{\prod_{r=1}^n \ve_i(y_r)}{\sum_{k=1}^n A_k}, \qquad \vp_i(y_1, \ldots, y_n) =  \dfrac{\prod_{r=1}^n \vp_i(y_r)}{\sum_{k=1}^n A_k}.
\end{equation}
The formulas for $\ve_i(\xx)$ and $\vp_i(\xx)$ are obtained by taking $Y_j = X_m$ for all $j$, and using~\eqref{eq:basic geometric crystal}.

Now we prove by induction that
\begin{equation}
\label{eq:e multi prod}
e_i^c(y_1, \ldots, y_n) = (e_i^{c_1}(y_1), \ldots, e_i^{c_n}(y_n)), \quad \text{where} \quad c_j = \dfrac{\ds c\sum_{k = 1}^j A_k + \sum_{k=j+1}^n A_k}{\ds c\sum_{k=1}^{j-1} A_k + \sum_{k=j}^n A_k},
\end{equation}
from which the stated formula for $e_i^c(\xx)$ follows. For $n=2$, this is just the definition~\eqref{eq:e prod}. For $n > 2$, we have $e_i^c(y_1, \ldots, y_n) = (e_i^{c^+}(y_1, \ldots, y_{n-1}), e_i^{c/c^+}(y_n))$, where
\[
c^+ = \dfrac{\ds c \vp_i(y_n) + \frac{\prod_{r=1}^{n-1} \ve_i(y_r)}{\sum_{k=1}^{n-1} A'_k}}{\ds \vp_i(y_n) + \frac{\prod_{r=1}^{n-1} \ve_i(y_r)}{\sum_{k=1}^{n-1} A'_k}} = \dfrac{\displaystyle c \sum_{k=1}^{n-1}A_k + A_n}{\displaystyle \sum_{k=1}^n A_k}
\]
by~\eqref{eq:rec id} and~\eqref{eq:functions multi prod}. This shows that $c/c^+ = c_n$. For the other $c_j$, induction gives $e_i^{c^+}(y_1, \ldots, y_{n-1}) = (e_i^{c^+_1}(y_1), \ldots, e_i^{c^+_{n-1}}(y_{n-1}))$, where
\begin{equation}
\label{eq:e proof}
c^+_j = \dfrac{\ds c^+\sum_{k = 1}^j A'_k + \sum_{k=j+1}^{n-1} A'_k}{\ds c^+\sum_{k=1}^{j-1} A'_k + \sum_{k=j}^{n-1} A'_k} = \dfrac{\ds \dfrac{c \sum_{k=1}^{n-1}A_k + A_n}{\sum_{k=1}^n A_k}\sum_{k = 1}^j A'_k + \sum_{k=j+1}^{n-1} A'_k}{\ds \dfrac{c \sum_{k=1}^{n-1}A_k + A_n}{\sum_{k=1}^n A_k}\sum_{k=1}^{j-1} A'_k + \sum_{k=j}^{n-1} A'_k}.
\end{equation}
The numerator of~\eqref{eq:e proof} can be rewritten as
\[
\dfrac{\ds c \sum_{k=1}^{n-1}A_k \sum_{k = 1}^j A'_k + \sum_{k=1}^{n-1} A_k \sum_{k=j+1}^{n-1} A'_k + A_n \sum_{k=1}^{n-1}A_k'}{\ds \sum_{k=1}^n A_k} = \dfrac{\ds \left(c\sum_{k = 1}^j A_k + \sum_{k=j+1}^n A_k\right) \sum_{k = 1}^{n-1} A'_k}{\ds \sum_{k=1}^n A_k},
\]
where the second expression is obtained by using $A_k = A'_k \vp_i(y_n)$ for $k \leq n-1$. The denominator of~\eqref{eq:e proof} can be rewritten in the same way with $j$ replaced by $j-1$; this proves~\eqref{eq:e multi prod}.
\end{proof}

\begin{proof}[Proof of Lemma~\ref{lem:dec formula}]
We first rewrite the decoration as
\[
F(\zz) = \sum_{j = 1}^{n-1} \sum_{i = 1}^{\min(m,n-j)} \frac{z_{i,i+j}}{z_{i,i+j-1}} \quad + \quad \sum_{i = 1}^{\min(m-1,n-1)} \sum_{j = 1}^{n-i} \frac{z_{i,i+j-1}}{z_{i+1,i+j}} \quad + \quad \mathbbm{1}_{m < n}z_{m,m}.
\]
It suffices to show that
\begin{equation}
\label{eq:first sum}
\sum_{i=1}^{\min(m,n-j)} \dfrac{z_{i,i+j}}{z_{i,i+j-1}} = \begin{cases}
\dfrac{\Delta_{[m+1,m+j],[1,j-1] \cup \{j+1\}}(M)}{\Delta_{[m+1,m+j],[1,j]}(M)} & \text{ if } j \leq n-m, \bigskip \\
\dfrac{\Delta_{[n-j+1,n],[1,j-1] \cup \{j+1\}}(M)}{\Delta_{[n-j+1,n],[1,j]}(M)} & \text{ if } j \geq n-m+1,
\end{cases}
\end{equation}
for $j = 1, \ldots, n-1$,
\begin{equation}
\label{eq:second sum}
\sum_{j = 1}^{n-i} \frac{z_{i,i+j-1}}{z_{i+1,i+j}} = \frac{\Delta_{\{i\} \cup [i+2,n],[1,n-i]}(M)}{\Delta_{[i+1,n],[1,n-i]}(M)}
\end{equation}
for $i = 1, \ldots, \min(m-1,n-1)$, and
\begin{equation}
\label{eq:i=m}
z_{m,m} = \frac{\Delta_{\{m\} \cup [m+2,n],[1,n-m]}(M)}{\Delta_{[m+1,n],[1,n-m]}(M)}
\end{equation}
if $m < n$.

Each of these equations is verified by applying the Lindstr\"om/Gessel--Viennot Lemma to the network $\Gamma_n^{\leq m}$ introduced in the proof of Lemma~\ref{lem:Phi inverse}. We explain only~\eqref{eq:second sum}; the arguments for the other two equations are a bit easier. The denominator of the right-hand side of~\eqref{eq:second sum} is given by
\begin{equation}
\label{eq:denominator}
\Delta_{[i+1,n],[1,n-i]}(M) = \prod_{k=i+1}^{\min(m,n)} z_{k,n},
\end{equation}
since this is the weight of the unique non-intersecting collection of paths from $[i+1,n]$ to $[1,n-i]$.

\begin{figure}
\begin{center}
\begin{tikzpicture}

\foreach \a/\b/\c/\d in {0/1/1/0, 0/2/2/0, 0/3/3/0, 0/4/4/0, 0/5/5/0, 1/5/6/0, 2/5/7/0, 1/5/1/0, 2/5/2/0, 3/4/3/0, 4/3/4/0, 5/2/5/0, 6/1/6/0} {\draw (\a,\b) -- (\c,\d);}
\foreach \a/\b in {0/1,0/2,0/3,0/4,0/5} {\filldraw (\a,\b) circle[radius=.04cm] node[left]{$\b$}; \filldraw (\b,\a) circle[radius=.04cm] node[below]{$\b'$};}
\foreach \a/\b/\c in {1/5/6, 2/5/7} {\filldraw (\a,\b) circle[radius=.04cm] node[above]{$\c$};}
\foreach \a/\b in {0/6, 0/7} {\filldraw (\b,\a) circle[radius=.04cm] node[below]{$\b'$};}
\foreach \a in {1,2,3,4,5} {\draw (0.6,\a-0.2) node{$z_{\a\a}$};}
\foreach \a/\b in {1/2,2/3,3/4,4/5,5/6} {\draw (1.6,\a-0.2) node{$\frac{z_{\a\b}}{z_{\a\a}}$};}
\foreach \a/\b/\c in {1/2/3,2/3/4,3/4/5,4/5/6,5/6/7} {\draw (2.6,\a-0.2) node{$\frac{z_{\a\c}}{z_{\a\b}}$};}
\foreach \a/\b/\c in {1/3/4,2/4/5,3/5/6,4/6/7} {\draw (3.6,\a-0.2) node{$\frac{z_{\a\c}}{z_{\a\b}}$};}
\foreach \a/\b/\c in {1/4/5,2/5/6,3/6/7} {\draw (4.6,\a-0.2) node{$\frac{z_{\a\c}}{z_{\a\b}}$};}
\foreach \a/\b/\c in {1/5/6,2/6/7} {\draw (5.6,\a-0.2) node{$\frac{z_{\a\c}}{z_{\a\b}}$};}
\draw (6.6,0.8) node{$\frac{z_{17}}{z_{16}}$};
\foreach \a/\b/\c/\d in {0/2/1/1, 0/4/1/3, 0/5/2/3, 1/5/3/3, 2/5/4/3, 1/1/1/0, 2/1/2/0, 3/1/3/0, 4/1/4/0, 5/1/5/0} {\draw[blue,ultra thick] (\a,\b) -- (\c,\d);}
\foreach \a/\b/\c/\d in {1/3/1/2, 2/3/2/2, 3/3/4/2, 4/3/5/2, 1/2/2/1, 2/2/3/1, 4/2/4/1, 5/2/5/1} {\draw[red,ultra thick] (\a,\b) -- (\c,\d);}

\end{tikzpicture}
\end{center}
\caption{A non-intersecting collection of paths in the network $\Gamma_7^{\leq 5}$ which contributes to the minor $\Delta_{24567, 12345}(M)$. The blue edges appear in all non-intersecting collections of paths from sources $\{2,4,5,6,7\}$ to sinks $\{1',2',3',4',5'\}$; the red edges appear in the collection corresponding to the choice $j = 3$ in the proof of Lemma~\ref{lem:dec formula}. The weight of this collection is $z_{2,4} \frac{z_{3,7}}{z_{3,5}} z_{4,7} z_{5,7}$.}
\label{fig:network paths}
\end{figure}
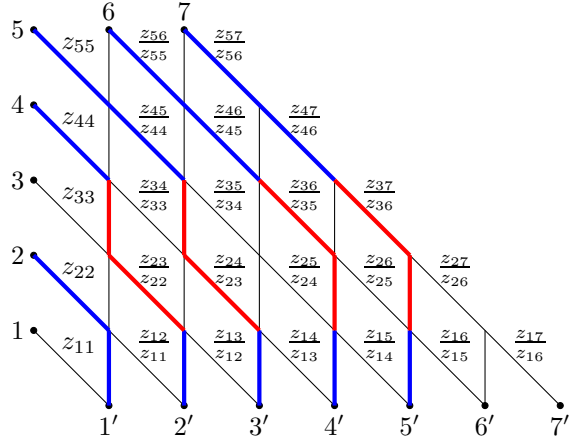

Now consider non-intersecting collections of paths from $\{i\} \cup [i+2,n]$ to $[1,n-i]$. Since $i \leq m-1$, there is a unique path from $i$ to $1'$. For $r = 2, \ldots, n-i$, the path from source $i+r-1$ to sink $r'$ must start by taking diagonal steps down to level $i+1$ (the height of source $i+1$), and it must end with vertical steps from level $i-1$ down to the bottom of the network. This leaves two choices for each path: to get from level $i+1$ to level $i-1$, it can travel diagonally and then vertically (DV), or vertically and then diagonally (VD). If the path ending in sink $r'$ chooses DV, then the path ending in sink $(r+1)'$ must also choose DV to avoid a collision. Thus, there must be some $j \in [1,n-i]$ such that the paths ending at $2', \ldots, j'$ choose VD, and the paths ending at $(j+1)', \ldots, (n-i)'$ choose DV (see Figure~\ref{fig:network paths} for an illustration). We conclude that
\[
\Delta_{\{i\} \cup [i+2,n],[1,n-i]}(M) = \prod_{k = i+2}^{\min(m,n)} z_{k,n} \; \times \; \sum_{j=1}^{n-i} z_{i,i+j-1} \frac{z_{i+1,n}}{z_{i+1,i+j}}.
\]
Dividing by~\eqref{eq:denominator}, we obtain~\eqref{eq:second sum}.
\end{proof}

\begin{proof}[Proof of Lemma~\ref{lem:explicit crystal formulas}]
Let $M = \Phi_n(\mb{z})$. By considering paths in the network $\Gamma_n$ from source $j$ (resp., $j+1$) to sink $j'$, one sees that
\[
M_{j,j} = \dfrac{z_{1,j} \cdots z_{j,j}}{z_{1,j-1} \cdots z_{j-1,j-1}},
\qquad\qquad
M_{j+1,j} = z_{j+1,j+1} \sum_{k=1}^j \dfrac{z_{1,j} \cdots z_{k-1,j}}{z_{1,j-1} \cdots z_{k-1,j-1}} \dfrac{z_{k+1,j+1} \cdots z_{j,j+1}}{z_{k+1,j} \cdots z_{j,j}}.
\]
The formula for $\ov{\gamma}(\mb{z})$ follows immediately. For $\ov{\varepsilon}_j(\mb{z})$, we compute
\begin{align*}
\ov{\varepsilon}_j(\mb{z}) = \dfrac{M_{j+1,j+1}}{M_{j+1,j}} &= \dfrac{\ds z_{j+1,j+1} \prod_{i=1}^j \dfrac{z_{i,j+1}}{z_{i,j}}}{\ds z_{j+1,j+1} \sum_{k=1}^j \prod_{i=1}^{k-1} \dfrac{z_{i,j}}{z_{i,j-1}} \prod_{i=k+1}^j \dfrac{z_{i,j+1}}{z_{i,j}}} \\
&= \dfrac{1}{\ds \sum_{k=1}^j \dfrac{z_{1,j}}{z_{1,j+1}} \prod_{i=1}^{k-1} \dfrac{z_{i,j}z_{i+1,j}}{z_{i,j-1}z_{i+1,j+1}}} = \dfrac{z_{1,j+1}}{z_{1,j}} \gMax_{1 \leq k \leq j} \left(\prod_{i=2}^{k} \phi_{i,j}^{-1}\right).
\end{align*}
A similar computation gives the formula for $\ov{\varphi}_j(\mb{z})$.

It remains to consider the geometric crystal operators. Set
\[
M' = x_j\left((c-1)\ov{\vp}_j(\mb{z})\right) \cdot M \cdot x_j\left((c^{-1}-1) \ov{\ve}_j(\mb{z})\right).
\]
By definition, the entries of the GT pattern $\mb{z}' = \ov{e}_j^c(\mb{z})$ are given by
\[
z_{i,j}' = \left(\frac{\Delta_{[i,j]}(M')}{\Delta_{[i+1,j]}(M')}\right)_{1 \leq i \leq j \leq n}.
\]
The matrix $M'$ is obtained from $M$ by adding a multiple of row $j+1$ to row $j$, and a multiple of column $j$ to column $j+1$. The latter operation does not affect flag minors, and the former operation only affects minors of the form $\Delta_{[s,r]}$ if $r = j$. This implies that $z_{i,r}' = z_{i,r}$ if $r \neq j$, and
\[
z_{i,j}' = \dfrac{\Delta_{[i,j]}(M) + (c-1)\ov{\varphi}_j(\mb{z}) \Delta_{[i,j-1] \cup \{j+1\}}(M)}{\Delta_{[i+1,j]}(M) + (c-1)\ov{\varphi}_j(\mb{z}) \Delta_{[i+1,j-1] \cup \{j+1\}}(M)}
\]
(when $i = j$, the denominator is equal to 1). We will show that
\[
\Delta_{[i,j]}(M) + (c-1)\ov{\varphi}_j(\mb{z}) \Delta_{[i,j-1] \cup \{j+1\}}(M) = \prod_{k=i}^j z_{k,j} \dfrac{\ds \sum_{k=1}^j c^{\mathbbm{1}_{k \geq i}} \prod_{\ell=k+1}^{j} \phi_{\ell, j}^{-1}}{\ds \sum_{k=1}^j \prod_{\ell=k+1}^{j} \phi_{\ell, j}^{-1}} = \prod_{k=i}^j z_{k,j} \; \times \; \dfrac{C_{i,j}}{C_{j+1,j}}
\]
for $i = 1, \ldots, j$, from which it follows that $z_{i,j}' = z_{i,j}\dfrac{C_{i,j}}{C_{i+1,j}}$.

Applying the Lindstr\"om/Gessel--Viennot Lemma to $\Gamma_n$, we obtain
\[
\Delta_{[i,j]}(M) = \prod_{k=i}^j z_{k,j}, \qquad\qquad
\Delta_{[i,j-1] \cup \{j+1\}}(M) = z_{j+1,j+1} \prod_{k=i}^{j-1} z_{k,j-1} \sum_{k=i}^j \prod_{\ell=i}^{k-1} \dfrac{z_{\ell,j}}{z_{\ell,j-1}} \prod_{\ell=k+1}^j \dfrac{z_{\ell,j+1}}{z_{\ell,j}}.
\]
Using these expressions and the formula $\ov{\varphi}_j(\mb{z}) = \dfrac{z_{j,j}}{z_{j+1,j+1}} \gMax_{1 \leq k \leq j} \left( \prod_{\ell=k+1}^{j} \phi_{\ell, j} \right)$, we compute
\begin{align*}
\Delta_{[i,j]}(M) + (c-1)\ov{\varphi}_j(\mb{z}) \Delta_{[i,j-1] \cup \{j+1\}}(M) &=
\prod_{k=i}^j z_{k,j} \left(1 + (c-1) \dfrac{\ds \prod_{\ell=i}^{j-1} \dfrac{z_{\ell,j-1}}{z_{\ell,j}} \sum_{k=i}^j \prod_{\ell=i}^{k-1} \dfrac{z_{\ell,j}}{z_{\ell,j-1}} \prod_{\ell=k+1}^j \dfrac{z_{\ell,j+1}}{z_{\ell,j}}}{\ds \sum_{k=1}^j \prod_{\ell=k+1}^{j} \phi_{\ell, j}^{-1}}\right) \\
&= \prod_{k=i}^j z_{k,j} \left(1 + (c-1) \dfrac{\ds \sum_{k=i}^j \prod_{\ell=k}^{j-1} \dfrac{z_{\ell,j-1}}{z_{\ell,j}} \prod_{\ell=k+1}^j \dfrac{z_{\ell,j+1}}{z_{\ell,j}}}{\ds \sum_{k=1}^j \prod_{\ell=k+1}^{j} \phi_{\ell, j}^{-1}}\right) \\
&= \prod_{k=i}^j z_{k,j} \left(1 + (c-1) \dfrac{\ds \sum_{k=i}^j \prod_{\ell=k+1}^{j} \phi_{\ell, j}^{-1}}{\ds \sum_{k=1}^j \prod_{\ell=k+1}^{j} \phi_{\ell, j}^{-1}}\right) \\
&= \prod_{k=i}^j z_{k,j} \; \times \; \dfrac{C_{i,j}}{C_{j+1,j}},
\end{align*}
as claimed.
\end{proof}

\section{The fundamental theorem of loop symmetric functions (with Thomas Lam)}
\label{app:FTLSF}

As in~\S \ref{sec:LSym}, we consider the polynomial ring $\CC[\lx{i}{j}]$ in an $m \times n$ matrix of variables. We denote by $\LSym$ the subring generated by the loop elementary symmetric functions $\lE{i}{j}$ ($i \in [m], j \in [n]$). Let
\[
\Pol_R = \Inv_R \cap \CC[\lx{i}{j}]
\]
denote the ring of polynomial invariants of the geometric $R$-matrices $R_1, \ldots, R_{m-1}$. We know from Lemma~\ref{lem:E is R invariant} that $\LSym \subseteq \Pol_R$; our goal is to prove that in fact $\LSym = \Pol_R$.

In~\S \ref{sec:dominant}, we reformulate the desired result in terms of the notion of dominant monomials, and we reduce to the case $m=2$ (when $m=1$, we have $\lE{1}{j} = \lx{1}{j}$, and there are no $R$-matrices, so there is nothing to prove). We deal with the $m=2$ case in~\S \ref{sec:m=2}. Our argument does not assume Theorem~\ref{thm:LSym}, which states that $\Frac(\LSym) = \Inv_R$.

\subsection{Dominant monomials}
\label{sec:dominant}

Let $\mb{x}^\mb{p} = \prod (\lx{i}{j})^{p_i^{(j)}}$ be a monomial in the $\lx{i}{j}$. Arrange the exponents $p_i^{(j)}$ into the $m \times n$ matrix
\begin{equation}
\label{eq:p matrix}
\begin{pmatrix}
p_1^{(1)} & p_1^{(2)} & \cdots & p_1^{(n)} \\
p_2^{(2)} & p_2^{(3)} & \cdots & p_2^{(n+1)} \\
\vdots & \vdots & & \vdots \\
p_m^{(m)} & p_m^{(m+1)} & \cdots & p_m^{(m+n-1)} \\
\end{pmatrix}
\end{equation}
(here the entry in position $(a,b)$ is the exponent of $x_a^b = x_a^{(a+b-1)}$). Say that $\mb{x}^\mb{p}$ is \defn{dominant} if each column of the matrix \eqref{eq:p matrix} is weakly decreasing. Note that the monomial $1$ is dominant.

Given a dominant monomial $\mb{x}^{\mb{p}}$, define
\[
E_{\mb{p}} = \prod_{j = 1}^n \prod_{k \geq 1} \lE{\lambda_k^{(j)}}{j},
\]
where $\lambda^{(j)} = (\lambda_1^{(j)}, \lambda_2^{(j)}, \ldots)$ is the conjugate of the partition $(p_1^{(j)}, p_2^{(j+1)}, \ldots, p_m^{(j+m-1)})$ appearing in the $j$th column of $\mb{p}$. The correspondence $\mb{x}^{\mb{p}} \longleftrightarrow E_{\mb{p}}$ is a bijection between dominant monomials in the $\lx{i}{j}$ and monomials in the $\lE{i}{j}$.

\begin{ex}
Let $m=3, n=2$. The monomial $(\lx{1}{1})^3 \lx{2}{2} (\lx{1}{2})^2 (\lx{2}{1})^2 \lx{3}{2}$ has exponent matrix
\[
\mb{p} = \begin{pmatrix}
3 & 2 \\
1 & 2 \\
0 & 1
\end{pmatrix},
\]
so it is dominant. The columns of $\mb{p}$ have conjugate partitions $\la^{(1)} = (2,1,1)$ and $\la^{(2)} = (3,2)$, so
\[
E_{\mb{p}} = \lE{2}{1} (\lE{1}{1})^2 \lE{3}{2} \lE{2}{2}.
\]
\end{ex}

Order monomials in the $\lx{i}{j}$ lexicographically with respect to the variable order
\[
\lx{1}{1} > \cdots > \lx{1}{n} > \lx{2}{2} > \cdots > \lx{2}{n+1} > \cdots > \lx{m}{m} > \cdots > \lx{m}{m+n-1}.
\]
It is clear that with respect to this monomial order, the leading term of $\lE{i}{j}$ is $\lx{1}{j} \lx{2}{j+1} \cdots \lx{i}{j+i-1}$. This implies the following result.

\begin{lemma}
\label{lem:leading term}
With respect to the above monomial ordering, the leading term of $E_{\mb{p}}$ is $x^{\mb{p}}$.
\end{lemma}

Lemma~\ref{lem:leading term} gives a simple proof that the $\lE{i}{j}$ are algebraically independent. In addition, this lemma allows us to reduce the result we are trying to prove to the case $m=2$.

\begin{prop}
\label{prop:reduction to m=2}
The following statements are equivalent:
\begin{enumerate}
\item $\Pol_R = \LSym$.
\item The leading monomial of any nonzero element of $\Pol_R$ is dominant.
\item Every nonzero element of $\Pol_R$ has a dominant term.
\item $\Pol_R = \LSym$ in the case $m=2$.
\end{enumerate}
\end{prop}

\begin{proof}
The implications $(1) \implies (4)$ and $(2) \implies (3)$ are obvious, and $(1) \implies (2)$ follows from Lemma~\ref{lem:leading term}.

$(3) \implies (1)$: Suppose $f \in \Pol_R$ is nonzero. Let $\mb{x}^{\mb{p}}$ be the largest dominant monomial in $f$. By subtracting a multiple of $E_{\mb{p}}$, we obtain another element of $\Pol_R$ whose largest dominant term is strictly less than $\mb{x}^{\mb{p}}$ (by Lemma~\ref{lem:leading term}). Continuing in this manner, we eventually obtain zero, so $f \in \LSym$.

$(4) \implies (2)$: Let $f(\mb{x}_1, \ldots, \mb{x}_m)$ be a nonzero element of $\Pol_R$. Let $\mb{x}^\mb{p}$ be the leading monomial of $f$. Fix $i \in [m-1]$, and let $f_i$ be the sum of all terms in $f$ whose exponent matrices agree with $\mb{p}$ outside of rows $i$ and $i+1$. Since $R_i$ only affects the sets of variables $\mb{x}_i$ and $\mb{x}_{i+1}$, we must have $R_i(f_i) = f_i$. We have already shown that (1) implies (2), so the leading term of $f_i$ must be dominant with respect to $\mb{x}_i$ and $\mb{x}_{i+1}$. This holds for all $i$, so $\mb{x}^{\mb{p}}$ is dominant.
\end{proof}

\subsection{The case $m=2$}
\label{sec:m=2}

We will now show that $\Pol_R = \LSym$ in the case of two sets of variables $\mb{x}_1 = (x_1^1, \ldots, x_1^n), \mb{x}_2 = (x_2^1, \ldots, x_2^n)$. By Proposition~\ref{prop:reduction to m=2}, this implies that $\Pol_R = \LSym$ in general.

Set $a_i = x_1^i$ and $b_i = x_2^i$. The subscripts of $a_i$ and $b_i$ are interpreted modulo $n$. In this notation, the loop elementary symmetric functions are given by
\[
\lE{1}{i} = a_i + b_{i-1}, \qquad \lE{2}{i} = a_ib_i,
\]
the loop homogeneous symmetric functions are given by
\[
\lH{k}{i} = \sum_{j=0}^k a_i a_{i-1} \cdots a_{i-j+1} b_{i-j-1} \cdots b_{i-k+1} b_{i-k},
\]
and the geometric $R$-matrix is given by
\[
a_i \mapsto b_i \frac{\kappa_{i+1}}{\kappa_i}, \qquad b_i \mapsto a_i \frac{\kappa_i}{\kappa_{i+1}}, \qquad
\text{ where } \quad
\kappa_i = \lH{n-1}{i-1}.
\]
For example, when $n=5$, we have
\[
\kappa_3 = \lH{4}{2} = a_2a_1a_5a_4 + a_2a_1a_5b_3 + a_2a_1b_4b_3 + a_2b_5b_4b_3 + b_1b_5b_4b_3.
\]
The Jacobi--Trudi formula (Proposition~\ref{prop:JT}) implies that the $\lH{k}{i}$ (and in particular, the $\kappa_i$) are in $\LSym$.

The first step in the proof is the simple observation that a polynomial $f \in \CC[\mb{a},\mb{b}]$ is $R$-invariant if and only if $f = \frac{1}{2}(f + Rf)$. Thus, to prove that $\Pol_R = \LSym$, it suffices to show that if $f \in \Pol_R$, then $f+Rf \in \LSym$. The next step is the following lemma.

\begin{lemma}
\label{lem:symmetrization}
For any $f \in \CC[\mb{a},\mb{b}]$, we have $f + Rf \in \LSym[\kappa_1^{-1}, \ldots, \kappa_n^{-1}]$.
\end{lemma}

\begin{proof}
It suffices to consider a monomial $f = \prod_{i=1}^n a_i^{\alpha_i} b_i^{\beta_i}$. By definition,
\begin{equation}
\label{eq:f+Rf}
f+Rf = \dfrac{\prod_{i=1}^n (a_i\kappa_i)^{\alpha_i} (b_i \kappa_{i+1})^{\beta_i} + \prod_{i=1}^n (b_i \kappa_{i+1})^{\alpha_i} (a_i\kappa_i)^{\beta_i}}{\prod_{i=1}^n \kappa_i^{\alpha_i+\beta_{i-1}}}.
\end{equation}
Observe that
\[
a_i\kappa_i = \lE{2}{i} \lH{n-2}{i-1} + \pi_1, \qquad b_i\kappa_{i+1} = \lE{2}{i} \lH{n-2}{i-1} + \pi_2,
\]
where $\pi_1 = \prod_{i=1}^n a_i$ and $\pi_2 = \prod_{i=1}^n b_i$. Observe also that
\[
\pi_1 + \pi_2 = \lH{n}{i} - \lE{2}{i}\lH{n-2}{i-1}, \qquad \pi_1\pi_2 = \lE{2}{1} \lE{2}{2} \cdots \lE{2}{n}, 
\]
so $\pi_1+\pi_2$ and $\pi_1\pi_2$ are in $\LSym$. This implies that any polynomial in $\LSym[\pi_1,\pi_2]$ which is symmetric in $\pi_1,\pi_2$ is in $\LSym$. The numerator of~\eqref{eq:f+Rf} is symmetric in $\pi_1$ and $\pi_2$, so we are done.
\end{proof}

Now suppose $f \in \Pol_R$. Lemma~\ref{lem:symmetrization} says that $f = \dfrac{g}{\prod_{j=1}^N \kappa_{i_j}}$ for some $g \in \LSym$ and $i_1, \dotsc, i_N \in [n]$. The following result shows that $f$ is in fact in $\LSym$ by induction on $N$ (note that $\dfrac{g}{\prod_{j=1}^M \kappa_{i_j}} \in \CC[\mb{a}, \mb{b}]$ for any $0 \leq M \leq N$ since $f \in \CC[\mb{a},\mb{b}]$).

\begin{lemma}
\label{lem:kappa divisor}
If $g$ is a nonzero element of $\LSym$ and $g/\kappa_i \in \CC[\mb{a},\mb{b}]$, then $g/\kappa_i \in \LSym$.
\end{lemma}

Recall that a monomial $\prod_{i=1}^n a_i^{\alpha_i} b_i^{\beta_i}$ is dominant if and only if $\alpha_i \geq \beta_i$ for all $i$. To prove Lemma~\ref{lem:kappa divisor}, it suffices to show that if $g$ is a nonzero element of $\LSym$ and $g/\kappa_i \in \CC[\mb{a},\mb{b}]$, then $g/\kappa_i$ has a dominant term. Indeed, we can then reduce $g$ to $0$ by repeatedly subtracting $\kappa_i$ times a monomial in the $\lE{k}{r}$, as in the proof of Proposition~\ref{prop:reduction to m=2}. Lemma~\ref{lem:kappa divisor} is therefore proved by taking $k = n-1$ and $r = 0$ in the following result.

\begin{lemma}
\label{lem:H divisor}
Fix $k \in [0,n-1]$ and $i \in [n]$. Suppose $g$ is a nonzero polynomial in $\LSym|_{b_{i+1} = \cdots = b_{i+r} = 0}$ for some $r \in [0,n-k]$ (that is, $g$ is obtained from an element of $\LSym$ by setting $b_{i+1} = \cdots = b_{i+r} = 0$). If $g/\lH{k}{i+1} \in \CC[\mb{a},\mb{b}]$, then $g/\lH{k}{i+1}$ has a dominant monomial.
\end{lemma}

In the proof, the words ``monomial'' and ``term'' will always refer to a product of the variables $a_i,b_i$ (multiplied by a nonzero element of $\CC$). If $S$ is a subset of $\CC[\mb{a},\mb{b}]$, an \defn{$S$-monomial} is a product of elements of $S$ (multiplied by a nonzero scalar).

\begin{proof}[Proof of Lemma~\ref{lem:H divisor}]
We proceed by induction on $k$. The base case $k=0$ asserts that every nonzero polynomial $g \in \LSym|_{b_{i+1} = \ldots = b_{i+r} = 0}$ has a dominant monomial. By definition, $g$ is a sum of $S$-monomials, where
\[
S = \{a_jb_j, a_{j+1}+b_j \mid j \neq i+1,\ldots,i+r\} \cup \{a_{i+2}, \ldots, a_{i+r+1}\}.
\]
Every $S$-monomial has a distinct leading term with respect to the variable order $a_1 > \cdots > a_n > b_1 > \cdots > b_n$. This implies that the elements of $S$ are algebraically independent, and that the leading term of $g$ is dominant.

Now suppose $k > 0$. Group together the $S$-monomials of $g$ which are divisible by $a_ib_i$, obtaining the expression
\[
g = a_ib_i p + q,
\]
where $p$ is a sum of $S$-monomials and $q$ is a sum of $(S \setminus \{a_ib_i\})$-monomials. If  $q = 0$, then since $\lH{k}{i+1}$ is not divisible by $a_i$ or $b_i$, $\lH{k}{i+1}$ must divide $p$. By induction on degree, we conclude that $p/\lH{k}{i+1}$ has a dominant term, so $g/\lH{k}{i+1}$ has a dominant term as well.

We may therefore assume that $q \neq 0$. Set $h = g/\lH{k}{i+1}$, and write $h = b_i h' + h''$, where $h'' = h|_{b_i=0}$. Setting $b_i=0$ in the equation $\lH{k}{i+1} h = g$, we obtain
\begin{equation}
\label{eq:b_i=0}
a_{i+1} \lH{k-1}{i} h'' = q|_{b_i=0}.
\end{equation}
The elements of $S \setminus \{a_ib_i\}$ remain algebraically independent after setting $b_i = 0$, so $q \neq 0$ implies that $q|_{b_i=0} \neq 0$. Thus, $a_{i+1} h''$ has a dominant term by the inductive hypothesis. If we knew that $h''$ had a dominant term, then this dominant term would appear in $h$, and we would be done. The only way $h''$ could fail to have a dominant term is if every dominant term in $a_{i+1}h''$ had the same number of $b_{i+1}$'s and $a_{i+1}$'s. If $r > 0$, then $g$ does not involve the variable $b_{i+1}$, so this cannot happen.

It remains to consider the case where $r = 0$, that is, where $g \in \LSym$. In this case, setting $a_i=0$ in the equation $\lH{k}{i+1} h = g$, we see that $a_{i+1} + b_i$ divides $q|_{a_i=0}$. By the same reasoning as in the previous paragraph, the assumption $q \neq 0$ implies that $q|_{a_i=0} \neq 0$ (here it is necessary that we have not set any of the $b_j$ variables equal to $0$). Group together the $(S \setminus \{a_ib_i\})$-monomials of $q$ which are divisible by $a_{i+1}+b_i$ to obtain the expression
\[
q = (a_{i+1} + b_i) q' + q'',
\]
where $q'$ is a sum of $(S \setminus \{a_ib_i\})$-monomials and $q''$ is a sum of $(S \setminus \{a_ib_i, a_{i+1}+b_i\})$-monomials. Since $a_{i+1}+b_i$ divides $q|_{a_i=0}$, it also divides $q''|_{a_i=0}$. But $q''$ does not involve the variable $b_i$, so we must have $q''|_{a_i=0} = 0$, which means $q'' = 0$. Thus, $q = (a_{i+1}+b_i) q'$, with $q' \in \LSym$. Going back to~\eqref{eq:b_i=0}, we see that
\[
a_{i+1} \lH{k-1}{i} h'' = a_{i+1} q'|_{b_i=0},
\]
so $h'' = (q'|_{b_i=0})/\lH{k-1}{i}$ has a dominant term by induction.
\end{proof}


\bibliographystyle{alpha}
\bibliography{CIT}

\end{document}